\documentclass[11pt,a4paper]{article}
\usepackage{amssymb}
\usepackage[T1]{fontenc}
\usepackage{lmodern}
\usepackage{flafter,amsmath,amssymb,latexsym,psfrag,graphicx,indentfirst,bm}
\usepackage[amsthm]{ntheorem}
\theorembodyfont{\itshape}
\usepackage{mathrsfs,enumerate,xcolor,float}
\numberwithin{equation}{section}
\newtheorem{theorem}{Theorem}[section]
\newtheorem{proposition}[theorem]{Proposition}
\newtheorem{lemma}[theorem]{Lemma}

\newtheorem{assumption}{Assumption}

\newtheorem{definition}{Definition}
\newtheorem{remark}[theorem]{Remark}

\usepackage[a4paper,top=2cm,bottom=2cm,left=2cm,right=2cm,marginparwidth=1.75cm]{geometry}
\usepackage{tikz}
\usetikzlibrary{arrows.meta,shapes.geometric,patterns,calc}
\usepackage{authblk}
\usepackage[hidelinks]{hyperref}

\definecolor{refinecolor}{RGB}{0, 70, 140}

\definecolor{rev}{RGB}{0, 120, 0}

\hypersetup{
    colorlinks=true,
    linkcolor=refinecolor,
    urlcolor=red
}

\newcommand{\R}{\mathbb{R}}

\newcommand{\veps}{\varepsilon}

\newcommand{\norm}[1]{\left\lVert #1 \right\rVert}
\newcommand{\abs}[1]{\left\lvert #1 \right\rvert}
\newcommand{\ip}[2]{\left\langle #1,#2 \right\rangle}
\usepackage{physics}
\DeclareMathOperator{\Div}{div}

\DeclareMathOperator{\dist}{dist}

\newsavebox\foobox
\newlength{\foodim}

\newcommand{\CMws}{\mathscr C_M^{\mathrm{ws}}}
\newcommand{\LMws}{\mathscr L_{M,T}^{\mathrm{ws}}}
\newcommand{\BM}{\mathcal B_{\partial,M}}
\newcommand{\OMext}{\mathcal O_M^\varepsilon}
\newcommand{\Xnine}{\mathcal X_9}
\newcommand{\LMell}{\mathscr L_{M,T,\ell}^{\mathrm{ws}}}
\AddToHook{env/abstract/begin}{\let\textcolor\abstracttextcolor}

\title{
\textbf{Resonant Microstructures as Dirac-type Sources\\ for Acoustic Wave Control}
}

\author[1]{Arpan Mukherjee}
\author[2]{Mourad Sini\thanks{The work of M. Sini is partially supported by the Austrian Science Fund (FWF): P 32660 and P 36942.}}

\affil[1]{\small MSU-BIT SMBU Joint Research Center of Applied Mathematics, Shenzhen MSU-BIT University, Shenzhen, People's Republic of China (\texttt{arpan.mukherjee@smbu.edu.cn}).}
\affil[2]{\small Johann Radon Institute (RICAM), Austrian Academy of Sciences, Linz, Austria (\texttt{mourad.sini@oeaw.ac.at}).}

\date{\today}

\begin{document}
\maketitle
\begin{abstract}

\noindent 
We study finite-time acoustic wave tracking \textcolor{red}{using} effective interior Dirac-type actuators generated by clusters of subwavelength resonant bubbles driven from exterior transducers. Starting from a time-domain point-interaction approximation, the full scattering problem is reduced to a finite-dimensional delayed system of interacting retarded monopoles. After localization to a bounded observation domain, we formulate the tracking problem on a prescribed finite-dimensional Dirichlet spectral subspace. At the ideal level, a full-rank spatial coupling condition yields exact tracking of arbitrary sufficiently smooth compatible trajectories in this subspace by interior point sources.

\vspace{0.05in}
\noindent
These ideal actuators are realized using small, high-contrast bubbles, whose Laplace-domain transfer operator encodes the Minnaert resonances and collective attenuation. The principal mode of each cluster is isolated by the Perron--Frobenius structure and the intra-cluster spectral gap. Clusters tuned to the same band are then treated as a collective resonant block. The selected Dirichlet bands are matched with the corresponding Minnaert bands.

\vspace{0.05in}
\noindent
For prescribed slowly time-modulated resonant pressure profiles, with sufficiently flat initialization, we construct causal exterior controls and obtain a quantitative tracking estimate on every fixed time interval $[0,T]$. In contrast, on frequency ranges uniformly separated from the Minnaert resonances, the actuator gain is $O(\varepsilon)$ and the corresponding realization cost is asymptotically unbounded as $\varepsilon \ll 1$, $\varepsilon$ being the maximum radius of the bubbles. Thus the Minnaert resonances convert accessible exterior excitations into robust effective interior controls.

\end{abstract}

\newpage

\tableofcontents

\section{Introduction}
\label{sec:intro}

\subsection{Control of Acoustic Waves by Localized Resonant Actuators}
Our aim in this work is to control selected acoustic modes by exciting small resonant bubbles from exterior transducers. The wave propagation problem is posed in the whole space $\R^3$. We take a bounded, smooth domain $\Omega$ containing the bubbles and use its Dirichlet eigenfunctions to describe the modes to be controlled. The boundary of $\Omega$ stays a fixed positive distance from the bubble centers. It serves only to define the observation region: the physical wave keeps its whole-space boundary trace, which enters the projected equation through a causal delay operator.

\medskip
\noindent
The control and observation of wave equations have been studied extensively for boundary and distributed controls; see, for instance, \cite{BardosLebeauRauch, LionsBook,MillerSurvey, TucsnakWeissBook, ZuazuaReview}. In these settings, geometric conditions on the control region and the propagation of rays determine the observability and controllability properties. Interior point sources require a different formulation. \textcolor{black}{In three spatial dimensions, point evaluations are not continuous on $H^1(\Omega)$ due to Sobolev embedding constraints ($H^1(\Omega) \not\hookrightarrow C^0(\overline{\Omega})$). Consequently, Dirac masses $\delta_{y}$ do not belong to the standard dual space $H^{-1}(\Omega)$, rendering a direct weak formulation in the usual energy space ill-posed.} A finite family of such sources provides only finitely many control directions; see \cite{TucsnakWeissBook, TucsnakWeiss2015}.

\medskip
\noindent
Of particular interest to us are actuators whose source strengths are outputs of a physical system, rather than freely prescribed functions. Related constructions for parabolic equations, including distributional actuators and heat generated by plasmonic particles, were studied in \cite{KunischRodriguesWalter,RodriguesSiniHeat}. Here the actuators arise from acoustic multiple scattering. Each cluster of bubbles produces, to leading order, a retarded monopole. The construction has therefore two successive steps: from a prescribed pressure profile we determine the ideal interior monopole strengths, and then we construct exterior incident waves whose interaction with the bubbles produces these strengths asymptotically. This allows desired acoustic fields to be generated by remote excitation of injected bubbles.
\medskip

\noindent A main motivation comes from imaging modalities using contrast agents. Bubbles or nanoparticles are useful for imaging, but they can also be used for therapeutic purposes; see \cite{Krishnan2010NanoparticleThermalTherapy,     LiChen2015GoldNanoparticlesPhotoacoustic, Qin2009UltrasoundContrastMicrobubbles,
      SheeranDayton2012PhaseChange,
      Versluis2020UltrasoundContrastAgentModeling,
      Xu2020OrganicNanoparticleMicrodroplets}. Once a region of interest has been identified, possibly by a contrast-agent imaging modality, see \cite{DabrowskiGhandricheSini2021AcousticImaging, GhandricheSini2022DielectricNanoparticles,
      GhandricheSini2022PhotoacousticMaxwell,
      GhandricheSini2023SimultaneousReconstruction,
      SenapatiSiniWangContrastingDroplets}, the same type of agents can be remotely excited to create a desired acoustic, optical, or thermal field in that region to eradicate the anomaly there. This principle idea is the basis behind sonotherapy, phototherapy, or thermotherapy, with hybrid therapy modalities, respectively.

\bigskip

      \noindent We propose a mathematical framework to model, analyze and quantify such therapy modalities taking the acoustic bubbles model (sonotherapy) as a prototype.

\subsection{Resonant Bubbles and Asymptotic Models}
\noindent
A gas bubble immersed in a liquid is a high-contrast inhomogeneity. Near its Minnaert frequency, the leading scattered wave is a monopole associated with the volume oscillation of the bubble; see \cite{HDAM-1, HDAM-2, Ammari-1, Ammari-2, CaflischBubbly, LiSiniFabryPerot}. For a family of bubbles, the interaction of these oscillations produces collective resonances and may give rise to subwavelength band gaps and effective-medium behavior.

\medskip
\noindent
The time-domain expansions derived in \cite{MukherjeeSiniJEE,MukherjeeSiniSIAM} describe these interactions under suitable assumptions on the contrast, size, and separation of the bubbles. Precisely, on a fixed time interval and away from the bubbles, the acoustic transmission solution with incident wave $u^{\mathrm{in}}$ has the form
\begin{equation}
\label{eq:intro-MS-expansion}
 u^\varepsilon(x,t)
 = u^{\mathrm{in}}(x,t)
   + \sum_{i=1}^M \frac{1}{4\pi|x-z_i|}
     \,q_i^\varepsilon\bigl(t-c_0^{-1}|x-z_i|\bigr)
   + R^\varepsilon(x,t),
\end{equation}
where the amplitudes $q_i^\varepsilon$ solve a finite-dimensional delayed system driven by the incident traces $u^{\mathrm{in}}(z_i,\cdot)$. \textcolor{black}{Fundamentally, these traces couple to a dynamic interaction matrix $\mathcal{Q}(s,\varepsilon)$ and local resonance matrix $\mathcal{D}(s)$ in the Laplace domain.} The remainder $R^\varepsilon$ tends to zero in the norms stated in Proposition~\ref{prop:MS-expansion}.

\medskip
\noindent
We use this expansion in the reverse direction. Given the monopole strengths required for tracking, we seek incident fields that generate them. To do this, we need to analyze the transfer map $(u^{in}(\cdot, \cdot))^{M_e}_{k=1} \rightarrow (q^{\varepsilon}_l)^{M_b}_{l=1}$, where $M_e$ and $M_b$ stand for the number of emitors (of incident waves) and bubbles respectively. The analysis of this transfer matrix, in the frequency domain, is highlighted and dominated  by the Minnaert resonance of each bubble and the generated hybrid resonances (of Perron-Frobenius by construction). A particular outcome of this analysis is that away from the Minnaert frequencies, the \textcolor{black}{capacitance} output scaling is $O(\varepsilon)$, so inversion on a fixed non-resonant band incurs at least a factor $\varepsilon^{-1}$. Near a Minnaert frequency tuned band, a slowly time-modulated source varies on the same scale as the resonant block from that band. This gives an $O(1)$ control bound relative to the intrinsic slow-envelope norm on the fixed tracking interval. This property is translated later as the possibility to track, in finite time, resonant profiles (i.e. profiles of the form linear combinations of (time-space) modes) which are slow time-modulated.

\subsection{Control Framework: From Scattering Amplitudes to Interior Controls}

\noindent
We describe the reduction from the microscopic bubble amplitudes to the effective point sources, and then the construction of the exterior controls.

\medskip
\noindent
\textit{\textbf{Level $1$: Microscopic bubble dynamics.} }
\textcolor{black}{The starting point is the $M$-bubble multiple-scattering system \eqref{eq:intro-MS-expansion}. This microscopic representation is required because the intra-cluster interactions determine the Minnaert poles, frequency shifts, and radiative damping. Consequently, the primary unknowns are the microscopic amplitudes}
\[
        \mathbf q_{\rm mic}^\varepsilon(t)
        =(q_1^\varepsilon(t),\ldots,q_M^\varepsilon(t)),
\]
\textcolor{black}{which solve the delayed algebraic system of the $M$-bubble approximation.}

\medskip
\noindent
\textit{ \textbf{Level $2$: Effective sources at the cluster centers.}}
The $M$ bubbles are arranged into $N$ local clusters,
\[
        M=\sum_{\alpha=1}^N M_\alpha,
        \qquad
        \{1,\ldots,M\}=\bigsqcup_{\alpha=1}^N \mathcal I_\alpha,
\]
with centers $y_\alpha$. Away from the clusters, the \textcolor{black}{background} retarded Green kernel is smooth in the source variable. Expanding it around each $y_\alpha$ gives
\begin{equation}
\label{eq:intro-cluster-output-expansion}
   \sum_{i=1}^M \mathcal W[q_i^\varepsilon;z_i]
   =
   \sum_{\alpha=1}^N \mathcal W[Q_\alpha^\varepsilon;y_\alpha]
   +\mathcal R_{\rm cl}^\varepsilon,
\end{equation}
where $\mathcal W$ denotes the retarded monopole operator and
\begin{equation}
\label{eq:intro-cluster-output-map}
        Q_\alpha^\varepsilon(t)
        =(B_{\rm out}\mathbf q_{\rm mic}^\varepsilon(t))_\alpha
        :=\sum_{i\in\mathcal I_\alpha}q_i^\varepsilon(t),
        \qquad \alpha=1,\ldots,N.
\end{equation}
\textcolor{black}{denotes the effective source strength of the physical cluster $\alpha$.} The remainder $\mathcal R_{\rm cl}^\varepsilon$ is \textcolor{black}{controlled by the maximal cluster diameter and the non-principal microscopic components. Consequently, the effective source $Q_\alpha^\varepsilon$ is primarily governed by a single principal resonant mode. The separation of this principal mode is proven using the Perron-Frobenius theorem and Kato's analytic perturbation theory. Since the unperturbed intra-cluster interaction matrix $\widetilde{C}_\alpha^{(0)}$ is non-negative and irreducible, it has a simple principal eigenvalue $\widetilde{\mu}_{\alpha, 1}^{(0)}$. Applying Kato's perturbation expansion to the characteristic matrix pencil $P(s, \varepsilon) := \mathcal{D}(s) + s^2 \mathcal{Q}(s, \varepsilon)$ perturbs this simple eigenvalue into the physical pole $s_\alpha^\varepsilon = -\eta_\alpha(\varepsilon) + i\omega_\alpha(\varepsilon)$. This determines the conservative frequency shift $\omega_\alpha(\varepsilon) - \omega_{M,\alpha} = \mathcal{O}(\varepsilon^{1-p})$ and the radiative damping $\eta_\alpha(\varepsilon) = \mathcal{O}(\varepsilon)$, yielding an $\mathcal{O}(\varepsilon^{1-p})$ spectral gap that separates $s_\alpha^\varepsilon$ from the complementary non-principal intra-cluster modes. These isolated principal modes serve as reference directions. Modes sharing a common carrier frequency have disjoint supports and are treated together as a collective block.}

\medskip
\noindent 
The scattered field \(u^{\mathrm{sc}, \varepsilon}\) is therefore approximated by a field \(p^{\mathrm{eff},\varepsilon}\) generated by point sources at the cluster centers. Restricted to \(\Omega\), this field satisfies
\[
\left(\frac{1}{c_0^2}\partial_{tt}-\Delta\right)p^{\mathrm{eff},\varepsilon}(t,x)
=
\sum_{\alpha=1}^{N} Q_\alpha^\varepsilon(t)\delta_{y_\alpha}.
\]
The boundary trace is inherited from the whole-space representation. \textcolor{black}{Galerkin} projection onto finitely many \textcolor{black}{smooth} Dirichlet eigenfunctions (of $-\Delta$ in $\Omega$) makes the point-source terms well-defined and gives the \textcolor{black}{regularized} modal system used for control.

\bigskip

\noindent
\textit{\textbf{Level 3. Ideal control, blockwise causal realization, and tracking.}} The control formulation consists of three steps. First, an ideal finite-dimensional tracking problem is solved on a selected Dirichlet spectral subspace, where the macroscopic cluster strengths act as interior Dirac actuators to generate the prescribed modal trajectory. This step is taken from control theory by singular sources, see \cite{TucsnakWeissBook} for instance, and adapted to our setting. Second, these ideal source strengths are approximated, up to an asymptotic error, by resonant bubble clusters driven by exterior transducers. Third, the stability of the projected wave dynamics ensures physical pressure tracking. This physical realization uses two properties of the Minnaert resonance: for each collective block, the matrix residue and the collective generator shift both scale as $\mathcal{O}(\varepsilon)$. Therefore, the physical realization map remains uniformly bounded if the local frequencies are tuned to the assigned bands and the source profiles have a flat initialization at $t=0$. Thus Step 2 concerns the inversion of the physical transfer map, while Step 3 transfers the source realization error to the pressure field. Below, we describe the key argument in implementing this whole construction of the finite time tracking.

\medskip
\noindent
\textcolor{black}{\textit{Step 1. Ideal finite-dimensional tracking.}} For the ideal tracking problem, let
\[
 H_M=\operatorname{span}\{\psi_{k_1},\ldots,\psi_{k_{N_M}}\},
 \qquad \omega_{k_j}=c_0\sqrt{\lambda_{k_j}}>0.
\]
The point sources at $y_\alpha$ couple through \textcolor{black}{the spatial evaluation matrix}
\begin{equation}\label{eq:intro-mode-cluster-matrix}
 C_M(\mathbf y)=\bigl(\psi_{k_j}(y_\alpha)\bigr)_
                         {1\le j\le N_M,\,1\le\alpha\le N} \textcolor{black}{\in \mathbb{R}^{N_M \times N}}.
\end{equation}
For physical realization we take $N=L N_M$ and partition the
clusters into $L$ banks of $N_M$ clusters. Write the $\ell$th bank as
$\mathcal N_\ell=\{\alpha_{\ell,1},\ldots,\alpha_{\ell,N_M}\}$ and define
$J_\ell:\mathbb C^{N_M}\to\mathbb C^N$ by
\begin{equation}\label{projection-J}
 (J_\ell v)_\alpha=
 \begin{cases}
 v_j,& \alpha=\alpha_{\ell,j},\quad j=1,\ldots,N_M,\\
 0,& \alpha\notin\mathcal N_\ell.
 \end{cases}
\end{equation}
Thus $J_\ell^*\mathbf q=(q_{\alpha_{\ell,1}},\ldots,q_{\alpha_{\ell,N_M}})^\top$.
Assume $C_M^{(\ell)}:=C_MJ_\ell$ is invertible and set
$L_M^{(\ell)}=J_\ell(C_M^{(\ell)})^{-1}$, so that
$C_ML_M^{(\ell)}=I_{N_M}$.We start from the prescribed target pressure of the form
\[
 p_r^{I_M}(x,t)=\sum_{j=1}^{N_M}(\mathbf p_r(t))_j\psi_{k_j}(x),
 \qquad
 \mathbf p_r(t)=\sum_{\ell=1}^L e^{i\widehat\omega_\ell t}
 P_\ell^\varepsilon(\varepsilon t),
\]
where $P_\ell^\varepsilon$ takes values in the Dirichlet eigenspace
$H_\ell\subset H_M$ associated with $\widehat\omega_\ell$, with the flat initial conditions \footnote{The initial flat and the slow varying, in time, conditions on the targeted profiles are needed as we deal with the bubbles-model with zero initial conditions and the generated damping, by the clustered bubbles, is small, of order $\varepsilon$, see Figure \ref{fig:comprehensive-math-tracking}, Panel III, and Section \ref{Poles-asymptotics}.} required below, and
$\Lambda_M=\operatorname{diag}(\omega_{k_1}^2,\ldots,\omega_{k_{N_M}}^2)$. Since
$\Lambda_M P_\ell^\varepsilon=\widehat\omega_\ell^2P_\ell^\varepsilon$, we obtain
\[
 \ddot{\mathbf p}_r+\Lambda_M\mathbf p_r
 =\sum_{\ell=1}^L f_\ell,
 \qquad
 f_\ell(t)=e^{i\widehat\omega_\ell t}F_\ell^\varepsilon(\varepsilon t),
\]
with
\[
 F_\ell^\varepsilon(\tau)
 =2i\widehat\omega_\ell\varepsilon(P_\ell^\varepsilon)'(\tau)
  +\varepsilon^2(P_\ell^\varepsilon)''(\tau).
\]
The corresponding ideal cluster source is then, see Theorem \ref{thm:finite-tracking}, 
\[
 \mathbf q^{\rm ideal}=c_0^{-2}\sum_{\ell=1}^L\LMell f_\ell.
\]
The operator $\LMell$ incorporates the delayed boundary trace \textcolor{black}{$\mathcal{B}_{\partial,M}$ into the right inverse via a finite causal Neumann series $(I - \mathcal{B}_{\partial,M} L_M^{(\ell)})^{-1}$.} It maps each packet to sources in its assigned bank. This gives exact inversion of the modal equation while preserving the carrier of every source, subject to the stated envelope regularity and initial conditions. Let us mention here that Section~\ref{sec:abstract-wave} also treats ideal tracking with a general right inverse $\LMws$.

\medskip
\noindent
\textcolor{black}{\textit{Step 2. Causal realization through a complete resonant block.}} To realize these sources by scattering, let $\mathcal N_\ell$ be the physical clusters assigned to
$\widehat\omega_\ell$, with $n_\ell=|\mathcal N_\ell|=N_M$.
The dominant part of the transfer map, of this group, has the \textcolor{black}{dimensionally exact meromorphic} form, see Proposition \ref{prop:collective-block-reduction}, 
\[
 C_\ell^\varepsilon(sI-S_\ell^\varepsilon)^{-1}
          R_\ell^\varepsilon B_\ell^\varepsilon(s),\qquad
 S_\ell^\varepsilon=i\widehat\omega_\ell I_{n_\ell}
                    +\varepsilon A_\ell^\varepsilon,
 \qquad R_\ell^\varepsilon=\varepsilon\widetilde R_\ell^\varepsilon.
\]
Here \textcolor{black}{$S_\ell^\varepsilon \in \mathbb{C}^{n_\ell \times n_\ell}$ is the collective generator}, $A_\ell^\varepsilon$ is uniformly bounded and
$\widetilde R_\ell^\varepsilon$ is uniformly invertible.

\medskip
\noindent
Write $\mathbf q_\ell^{\rm ideal}
=e^{i\widehat\omega_\ell t}J_\ell\mathbf Q_\ell^\varepsilon(\varepsilon t)$,
where $J_\ell$ is the coordinate of the group, as defined above. This is justified since we took $N=LN_M$, see Proposition \ref{prop:bank-right-inverse}. We successively invert the output map, the block evolution, and the transducer input:
\[
 \begin{aligned}
 z_\ell&=(C_\ell^\varepsilon)^\dagger\mathbf q_\ell^{\rm ideal},\\
 u_{{\rm req},\ell}&=(R_\ell^\varepsilon)^{-1}
                  (\partial_t-S_\ell^\varepsilon)z_\ell,\\
 \lambda_\ell&=(B_\ell^\varepsilon(i\widehat\omega_\ell))^\dagger
                                                u_{{\rm req},\ell}.
 \end{aligned}
\]
With $z_\ell=e^{i\widehat\omega_\ell t}Z_\ell(\varepsilon t)$, a form that is inherited from $q^{ideal}_l$, \textcolor{black}{expanding the differential operator explicitly demonstrates the resonant cancellation}
\begin{align*}
 \textcolor{black}{(\partial_t-S_\ell^\varepsilon) \left[ e^{i\widehat\omega_\ell t} Z_\ell(\varepsilon t) \right]} &\textcolor{black}{= e^{i\widehat\omega_\ell t} \left[ \varepsilon Z_\ell'(\varepsilon t) + i\widehat\omega_\ell Z_\ell(\varepsilon t) - (i\widehat\omega_\ell I + \varepsilon A_\ell^\varepsilon) Z_\ell(\varepsilon t) \right]} \\
 &\textcolor{black}{= \varepsilon e^{i\widehat\omega_\ell t} \left[ Z_\ell'(\varepsilon t) - A_\ell^\varepsilon Z_\ell(\varepsilon t) \right].}
\end{align*}
\textcolor{black}{Dividing this exact derivative by $R_\ell^\varepsilon = \varepsilon \widetilde{R}_\ell^\varepsilon$ isolates the uniform $\mathcal{O}(1)$ forcing}
\[
 u_{{\rm req},\ell}(t)
 =e^{i\widehat\omega_\ell t}(\widetilde R_\ell^\varepsilon)^{-1}
       [Z_\ell'(\varepsilon t)-A_\ell^\varepsilon Z_\ell(\varepsilon t)].
\]
Now, use $\lambda=(\lambda_l)_l$, with $\lambda_\ell=(B_\ell^\varepsilon(i\widehat\omega_\ell))^\dagger u_{{\rm req},\ell}$, to design the incident wave and then the corresponding monopole $q^\varepsilon_{cl}$, from the original physical model. Estimates for the delayed system and the exact principal kernel then give, see Proposition \ref{prop:surjectivity-to-qdiff-J}
\[
 \|\mathbf q_{\rm cl}^\varepsilon-\mathbf q^{\rm ideal}\|_\infty
 \le C_T(\varepsilon+\varepsilon^p)
                  \|\mathbf q^{\rm ideal}\|_{\mathcal V_\varepsilon^m},
 \qquad
 \|\boldsymbol\lambda^\varepsilon\|_{\Xnine}
 \le C_T\|\mathbf q^{\rm ideal}\|_{\mathcal V_\varepsilon^m}.
\]
Moreover,
$\|\mathbf q^{\rm ideal}\|_{\mathcal V_\varepsilon^m}
\le C_T\|\boldsymbol f\|_{\mathfrak F_\varepsilon^m}$, where
$\boldsymbol f=(f_1,\ldots,f_L)$ is determined by the prescribed pressure profile above.
The cancellation of the carrier removes the $\varepsilon^{-1}$ factor from the control bound in the intrinsic norm. The actual amplitude and slow derivatives of the prescribed sources remain in that norm. This is the finite-time resonant gain used below; the comparison with non-resonant harmonic inversion is given in Section~\ref{sec:resonant-gain}. 

\medskip
\noindent
\textcolor{black}{\textit{Step 3. Tracking of the physical scattered wave.}}
Let $P_{I_M}$ be the Dirichlet spectral projection onto $H_M$. The modal wave equation transfers the source error $\mathbf q_{\rm cl}^\varepsilon-\mathbf q^{\rm ideal}$ to an error in the effective pressure. Combining this estimate with the cluster-center approximation and the Foldy--Lax remainder gives, for some $\gamma>0$,
\[
 \sup_{t\le T}\|\OMext p^\varepsilon(t)-p_r^{I_M}(t)\|_{H^1(\Omega)}
 \le C_T\varepsilon^\gamma
                 \|\mathbf q^{\mathrm{ideal}}\|_{\mathcal V_\varepsilon^m}
 \le C_T\varepsilon^\gamma
                 \|\boldsymbol f\|_{\mathfrak F_\varepsilon^m},
\]
where $\OMext$ is the exterior-observation projection, see Step 3 in the proof of Proposition \ref{prop:surjectivity-to-qdiff-J}. Thus the right-hand side is controlled by data determined from the prescribed target profile. The amplitude normalization in \eqref{eq:amplitude-normalization}--\eqref{eq:normalized-envelope-shape} specifies the families for which this homogeneous estimate gives a vanishing tracking error.

\medskip
\noindent
\textcolor{black}{The three levels therefore connect the full microscopic scattering problem to a finite-dimensional effective actuator model and finally to a causal resonant control construction with quantified physical tracking error.} Figure~\ref{fig:comprehensive-math-tracking} summarizes this passage from the scattering problem to the effective sources and the exterior controls.

\begin{figure}[htbp]
    \centering
    \begin{tikzpicture}[
        >=stealth,
        x=1cm, y=1cm,
        every node/.style={font=\footnotesize}, 
        cascadebox/.style={align=center, text width=4.4cm, rounded corners=3pt, inner sep=6pt, thick},
        mathbox/.style={align=left, text width=7.2cm, inner sep=0pt},
        axis/.style={->, line width=0.6pt, draw=black!80},
        pole/.style={circle, fill=black, inner sep=1.2pt},
        mainpole/.style={circle, fill=red!70!black, inner sep=1.5pt}
    ]

        \draw[loosely dotted, thick, black!20] (-7.5, 8.5) -- (7.5, 8.5);
        \draw[loosely dotted, thick, black!20] (-7.5, 0.8) -- (7.5, 0.8); 
        \draw[loosely dotted, thick, black!20] (0, 13.8) -- (0, 0.8); 

        \node[anchor=south west, font=\scshape\bfseries] at (-7.5, 13.5) {I. Bubble geometry};
        
        \begin{scope}[xshift=-0.5cm, yshift=0.5cm]
            
            \draw[thick, draw=blue!40!black, fill=blue!3] (-4.6, 10.7) ellipse (1.8cm and 1.5cm);
            \node[font=\normalsize, text=blue!80!black] at (-4.6, 11.8) {$\Omega$};

            \node[pole, label={[font=\scriptsize, label distance=0pt]left:$x_1^{\mathrm{tr}}$}] (T1) at (-6.16, 11.45) {};
            \node[pole, label={[font=\scriptsize, label distance=0pt]left:$x_m^{\mathrm{tr}}$}] (T2) at (-6.4, 10.7) {};
            \node[pole, label={[font=\scriptsize, label distance=0pt]left:$x_{M_{\mathrm{tr}}}^{\mathrm{tr}}$}] (T3) at (-6.16, 9.95) {};
            \node[font=\scriptsize, align=center, anchor=east] at (-6.7, 11.9) {Transducers\\$U_{\mathrm{tr}}^\varepsilon$};
            \draw[black!60, thick] (-6.16, 11.45) to[out=240, in=90] (-6.4, 10.7) to[out=-90, in=120] (-6.16, 9.95);

            \draw[blue!50!black, densely dashed, thick] (-6.4, 10.7) ++(-50:0.5) arc (-50:50:0.5);
            \draw[blue!50!black, densely dashed, thick] (-6.4, 10.7) ++(-45:0.8) arc (-45:45:0.8);
            \draw[blue!50!black, densely dashed, thick] (-6.4, 10.7) ++(-40:1.1) arc (-40:40:1.1);
            \node[font=\scriptsize, text=blue!80!black, anchor=west] at (-5.2, 10.7) {$p^\varepsilon(x,t)$};

            \coordinate (Y1) at (-3.4, 11.2); 
            \coordinate (Y2) at (-3.4, 10.2);  
            \draw[dashed, thick, black!60, fill=black!5] (Y1) circle (0.3cm);
            \draw[dashed, thick, black!60, fill=black!5] (Y2) circle (0.3cm);
            \node[font=\scriptsize, text=black!80, anchor=south] at (-3.4, 11.5) {Cluster $y_\alpha$};
            \draw[<->, thin, black!60] (-3.4, 10.8) -- (-3.4, 10.6) node[midway, right=1pt, font=\scriptsize] {$\mathcal{O}(1)$};

            \draw[thick, draw=orange!50!black, fill=orange!3] (-1.0, 10.7) circle (1.2cm);
            \node[font=\scriptsize, align=center, text=orange!80!black] at (-1.0, 12.1) {Bubbles ($z_i$)};
            \draw[densely dotted, thick, black!50] (-3.4, 11.5) -- (-1.0, 11.9); 
            \draw[densely dotted, thick, black!50] (-3.4, 10.9) -- (-1.0, 9.5);  
            
            \node[pole, fill=black!60] (B1) at (-0.6, 11.2) {};
            \node[pole, fill=black!60] (B2) at (-1.5, 10.5) {};
            \node[pole, fill=black!60] (B3) at (-0.9, 9.9) {};
            \draw[<->, thin] (B1) -- (B2) node[midway, above=1pt, font=\scriptsize, sloped] {$d_{ij} \sim \varepsilon^p$};
            \draw[<->, thin, black!60] (B2) -- (B3);
            \draw[<->, thin, black!60] (B1) -- (B3);
        \end{scope}

        \node[anchor=south west, font=\scshape\bfseries] at (0.2, 13.5) {II. Foldy--Lax approximation};
        
        \node[mathbox, anchor=north west] at (0.2, 13.0) {
            The full scattering problem is reduced by the time-domain Foldy--Lax approximation to interacting retarded monopoles:\par\vspace{1.0ex}
            \hfil $\displaystyle u^{\mathrm{sc},\varepsilon}(x,t) = \sum_{i=1}^M \mathcal{W}[q_i^\varepsilon; z_i](x,t) + R^\varepsilon(x,t)$ \hfil\par\vspace{1.0ex}
            Local interactions select the Perron modes; weak inter-cluster interactions form complete resonant blocks:\par\vspace{1.0ex}
            \hfil $\displaystyle \mathbf q_{\rm mic}^\varepsilon \ \longleftrightarrow\ \text{delayed bubble dynamics}$ \hfil
        };

        \node[anchor=south west, font=\scshape\bfseries] at (-7.5, 7.8) {III. Local resonance};
        
        \draw[axis] (-7.0, 2.8) -- (-1.5, 2.8) node[right] {$\Re(s)$};
        \draw[axis] (-3.5, 2.3) -- (-3.5, 7.2) node[above] {$\Im(s)$};
        
        \draw[line width=4pt, blue!30] (-3.5, 3.8) -- (-3.5, 5.8);
        \node[right, text=blue!80!black] at (-3.4, 4.0) {Band $I_M$};

        \node[pole] (Unp) at (-3.5, 6.6) {};
        \node[right, align=left] at (-3.4, 6.7) {$i\omega_{M,\alpha}$};
        
        \node[mainpole] (Per) at (-5.5, 4.6) {};
        \node[anchor=south east, text=red!70!black, inner sep=2pt] at (-5.6, 4.7) {$\mathbf{s_\alpha^\varepsilon}$};
        
        \draw[dotted, thick] (Per) -- (-3.5, 4.6) node[right] {$i\widehat\omega_{\tau(\alpha)}$};
        \draw[dotted, thick] (Per) -- (-5.5, 2.8) node[below] {$-\eta_\alpha^\varepsilon$};
        
        \draw[<->, thin, black!60] (-3.5, 6.6) -- (-3.5, 4.6) node[midway, right] {$\Delta\omega \sim \mathcal{O}(\varepsilon^{1-p})$};
        \draw[<->, thin, black!60] (-3.5, 4.6) -- (-5.5, 4.6) node[midway, below] {$\eta \sim \mathcal{O}(\varepsilon)$};
        
        \draw[->, thick, black!50, dashed] (-3.6, 6.4) -- (-5.3, 4.8) node[midway, sloped, above] {Resonance shift};

        \node[anchor=south west, font=\scshape\bfseries, align=left] at (0.2, 7.3) {IV. Collective transfer\\operator};
        
        \node[mathbox, anchor=north west] at (0.2, 7.2) {
            At a common carrier, the principal transfer is a matrix block:\par\vspace{1.0ex}
            \hfil $\displaystyle H_{\rm prin,\ell}(s)=\widetilde{\mathcal A}_bV_\ell^\varepsilon(sI-S_\ell^\varepsilon)^{-1}R_\ell^\varepsilon(W_\ell^\varepsilon)^\top$ \hfil\par\vspace{1.0ex}
            \hfil $\displaystyle S_\ell^\varepsilon=i\widehat\omega_\ell I+\varepsilon A_\ell^\varepsilon,\quad R_\ell^\varepsilon=\varepsilon\widetilde R_\ell^\varepsilon$ \hfil\par\vspace{1.0ex}
            The collective generator is bounded. Its complementary gap is $O(\varepsilon^{1-p})$.\par\vspace{1.0ex}
            \hfil $\displaystyle \sup_{0\le t\le T}\|e^{S_\ell^\varepsilon t}\|\le e^{C\varepsilon T}$ \hfil
        };

        \node[anchor=south west, font=\scshape\bfseries] at (-7.5, 0.2) {V. Causal controls and physical tracking};

        \node[cascadebox, anchor=center, fill=blue!3, draw=blue!30] (Box1) at (-5.1, -3.8) {
            {\textbf{Ideal modal tracking}\par\vspace{0.8ex}
            Prescribe the slowly modulated target\par\vspace{0.6ex}
            $\displaystyle \mathbf p_r(t)=\sum_\ell e^{i\widehat\omega_\ell t}P_\ell^\varepsilon(\varepsilon t)$.\par\vspace{0.8ex}
            Then $\ddot{\mathbf p}_r+\Lambda_M\mathbf p_r=\sum_\ell f_\ell$ and the bankwise inverse gives\par\vspace{0.8ex}
            $\displaystyle \mathbf q^{\mathrm{ideal}}=c_0^{-2}\sum_\ell\LMell f_\ell.$\par\vspace{0.8ex}
            The boundary-corrected inverse preserves the assigned carrier.}
        };

        \node[cascadebox, anchor=center, fill=orange!3, draw=orange!40] (Box2) at (0.0, -3.8) {
            {\textbf{Resonant realization}\par\vspace{1.0ex}
            For the slowly modulated resonant profile $q_\alpha^{\mathrm{ideal}}(t) \!=\! Q_\alpha^\varepsilon(\varepsilon t)e^{i\widehat\omega_{\tau(\alpha)} t}$, the modal inverse has an intrinsic-norm bound:\par\vspace{0.8ex}
            $\displaystyle \mathbf u_{{\rm req},\ell}=(R_\ell^\varepsilon)^{-1}(\partial_t-S_\ell^\varepsilon)\mathbf z_\ell$\par\vspace{1.0ex}
            Here $\mathbf z_\ell=(C_\ell^\varepsilon)^\dagger\mathbf q_\ell$. The initial conditions $(Q_\alpha^\varepsilon)^{(k)}(0) \!=\! 0$ ($k \!\le\! m$) give the required causal regularity.}
        };

        \node[cascadebox, anchor=center, fill=black!3, draw=black!30] (Box3) at (5.1, -3.8) {
            {\textbf{Physical tracking}\par\vspace{1.0ex}
            The causal intrinsic-norm cost estimate is:\par\vspace{0.8ex}
            $\displaystyle \|\boldsymbol{\lambda}^\varepsilon\|_{U_{\mathrm{tr}}^\varepsilon} \le \kappa_M \|\mathbf{q}^{\mathrm{ideal}}\|_{\mathcal{V}_\varepsilon^m}$\par\vspace{1.0ex}
            The realized cluster sources approximate the ideal profile:\par\vspace{0.8ex}
            $\displaystyle \|\mathbf{q}_{\mathrm{cl}}^\varepsilon \!-\! \mathbf{q}^{\mathrm{ideal}}\|_{L^\infty} \!=\! \mathcal{O}\bigl((\varepsilon \!+\! \varepsilon^p) \|\mathbf{q}^{\mathrm{ideal}}\|_{\mathcal{V}_\varepsilon^m}\bigr)$\par\vspace{1.0ex}
            Stability of the projected wave dynamics yields physical tracking:\par\vspace{0.8ex}
            $\displaystyle \sup_{t} \big\| \OMext p^\varepsilon-p_r^{I_M} \big\|_{H^1} \!=\! \mathcal{O}\bigl(\varepsilon^{\gamma} \|\mathbf{q}^{\mathrm{ideal}}\|_{\mathcal{V}_\varepsilon^m}\bigr)$}
        };

        \draw[->, thick, black!40] (Box1.east) -- (Box2.west);
        \draw[->, thick, black!40] (Box2.east) -- (Box3.west);

    \end{tikzpicture}
    \caption{From bubble scattering to acoustic tracking. Panels I--II correspond to the microscopic dynamics (Level 1), with $Q_\alpha^\varepsilon=\sum_{i\in\mathcal I_\alpha}q_i^\varepsilon$ giving the effective cluster source (Level 2); panels III--V describe the resonant realization and tracking construction (Level 3).The spectral sketch depicts an isolated local branch; a common-frequency group is treated as a matrix block.}
    \label{fig:comprehensive-math-tracking}
\end{figure}

\newpage

\section{Mathematical Formulation and Main Results}

\subsection{Geometric and Functional Setting}

\noindent
Let $\Omega\subset\R^3$ be a bounded, connected $C^{1,1}$ domain containing the bubbles. The acoustic problem is posed in the whole space, while $\Omega$ defines the observation region and the spectral coordinates for tracking. We fix $T>0$ and assume
\begin{equation}\label{eq:finite-propagation-localization}
 \dist\bigl(D\cup\{y_1,\ldots,y_N\},\partial\Omega\bigr)\ge d_0>0,
\end{equation}
with $d_0$ independent of $\varepsilon$. The numbers of bubbles, clusters, and selected modes remain fixed as $\varepsilon\to0$. We use the eigenfunctions of the Dirichlet Laplacian as a basis, keeping the actual boundary trace of the whole-space wave through \eqref{eq:boundary-delay-operator}. Thus no homogeneous boundary condition is imposed on the physical field, and waves may reach $\partial\Omega$ during the observation interval.

\noindent We set
\[
 H:=L^2(\Omega),\qquad V:=H^1_0(\Omega),
\]
and define the energy space
\[
 \mathcal{H}:=V\times H,\qquad
 \norm{(u,v)}_{\mathcal{H}}^2
 :=
 \norm{\nabla u}_{L^2(\Omega)}^2 + \frac{1}{c_0^2}\,\norm{v}_{L^2(\Omega)}^2,
\]
where $c_0>0$ is the background sound speed.

\medskip
\noindent
The Dirichlet Laplacian $-\Delta_D$ on $\Omega$ is defined by
\[
 -\Delta_D u := -\Delta u,\qquad
 D(-\Delta_D):=H^2(\Omega)\cap H^1_0(\Omega).
\]
It is self-adjoint and strictly positive on $H$ and admits an orthonormal basis of eigenfunctions $\{\psi_k\}_{k\ge1}\subset D(-\Delta_D)$ with
\[
 -\Delta_D\psi_k = \lambda_k\psi_k,\qquad
 0<\lambda_1\le\lambda_2\le\cdots,\quad \lambda_k\to\infty.
\]
We define
\[
 \omega_k:=c_0\sqrt{\lambda_k}>0,\qquad k\ge1,
\]
and call $\omega_k$ the angular frequencies of the Dirichlet modes.

\medskip
\noindent
The ideal actuators are the Dirac masses $\delta_{y_\alpha}$ at the cluster centers $y_\alpha\in\Omega$. In three dimensions, point evaluation is not continuous on $H^1_0(\Omega)$, and hence $\delta_{y_\alpha}\notin V'=H^{-1}(\Omega)$. We therefore work with the projected wave equation rather than an energy-space formulation of the point-source problem. For $\varphi\in H_{N_M}:=\mathrm{span}\{\psi_k:k\in K_M\}$, the smoothness of the eigenfunctions allows us to set $\ip{\delta_{y_\alpha}}{\varphi}:=\varphi(y_\alpha).$

\subsection{The Ideal Point-Source Model and its Physical Realization}
\label{subsec:math-formulation}
\noindent
We first define the ideal source field and the physical transmission problem. Their relation will determine which modal trajectories can be realized by exterior transducers. The source envelopes are taken flat at the initial time so that the constructed controls extend smoothly by zero. The estimates account for the propagation delays and are expressed in the intrinsic slow-envelope norm.

\begin{definition}[Ideal point-actuated whole-space model]
The ideal field is the causal whole-space solution generated by
$\mathbf q=(q_\alpha)_{\alpha=1}^N$ at $y_\alpha\in\Omega$. Its restriction
is characterized by
\begin{equation}\label{eq:ideal-controlled-wave}
 \begin{cases}
 c_0^{-2}\partial_t^2p-\Delta p=\displaystyle\sum_\alpha q_\alpha(t)\delta_{y_\alpha}
       &\text{in }\Omega\times(0,T),\\
 p=g_{\mathbf q}&\text{on }\partial\Omega\times(0,T),\\
 p(\cdot,0)=\partial_t p(\cdot,0)=0,
 \end{cases}
\end{equation}
where
$g_{\mathbf q}(x,t)=\sum_\alpha
q_\alpha(t-|x-y_\alpha|/c_0)/(4\pi|x-y_\alpha|)$.
The sources determine both the field and its boundary trace. We use their finite-dimensional projection below; $g_{\mathbf q}$ is not an independent control.
\end{definition}

\begin{definition}[Spectral subspace]
Let $\{\lambda_k, \psi_k\}_{k=1}^\infty$ be the Dirichlet eigenpairs of $-\Delta$ in $\Omega$, with $\omega_k=c_0\sqrt{\lambda_k}$. For a given finite frequency band $J\subset(0,\infty)$, we define the index set $K_J := \{k : \omega_k \in J\}$ and denote its cardinality by $|K_J|$. The corresponding finite-dimensional spectral subspace and its $L^2(\Omega)$-orthogonal projector are
\begin{equation*}
 H_J := \mathrm{span}\{\psi_k : k \in K_J\}, \qquad P_J : L^2(\Omega) \to H_J.
\end{equation*}
Projecting \eqref{eq:ideal-controlled-wave} onto $H_J$, the modal coefficients $p_k(t) = \langle p(t, \cdot), \psi_k \rangle$ satisfy the finite-dimensional forced oscillator system
\begin{equation}
\label{eq:modal-ode}
 \ddot p_k(t)+\omega_k^2 p_k(t)=c_0^2\sum_{\alpha=1}^N q_\alpha(t)\,\psi_k(y_\alpha)-c_0^2\int_{\partial\Omega}g_{\mathbf q}(x,t)\partial_\nu\psi_k(x)\,dS_x,\quad k\in K_J.
\end{equation}
\end{definition}

\begin{definition}[Physical transmission problem and control spaces]
Let $D = \cup_{i=1}^M D_i$ be a finite family of high-contrast bubbles of size $\varepsilon$. Their centers $z_i$ are grouped around the cluster centers $y_\alpha$. The total acoustic pressure $u(x,t)$ solves
\begin{equation}
\label{eq:MS-physical}
\left\{
\begin{aligned}
 &\kappa^{-1}(x)\,\partial_t^2 u(x,t)-\Div\bigl(\rho^{-1}(x)\nabla u(x,t)\bigr)
   =\sum_{m=1}^{M_{\rm tr}} \lambda_m(t)\delta_{x_m^{\mathrm{tr}}}(x)
   &&\text{in }(\R^3\setminus\partial D)\times({0},T),\\[0.2em]
 &u|_+ = u|_- &&\text{on }\partial D\times({0},T),\\[0.2em]
 &\rho_c^{-1}\partial_\nu u|_+ = \rho_b^{-1}\partial_\nu u|_- &&\text{on }\partial D\times({0},T),\\[0.2em]
 &u(\cdot,{0})=0,\quad \partial_t u(\cdot,{0})=0 &&\text{in }\R^3,
\end{aligned}
\right.
\end{equation}
where $\rho_b, \rho_c$ and $\kappa_b, \kappa_c$ denote the densities and bulk moduli inside and outside the bubbles, respectively. 
\medskip

\noindent
For exterior transducers located at $\{x_m^{\mathrm{tr}}\}_{m=1}^{M_{\mathrm{tr}}} \subset \mathbb{R}^3 \setminus \overline{\Omega}$, we define the physical control space as
\begin{equation}
\label{eq:Utr-definition}
    { U_{\mathrm{tr}}^\varepsilon := \bigl\{ \boldsymbol{\lambda} \in W^{2,\infty}(0,T; \mathbb{R}^{M_{\mathrm{tr}}}) : \boldsymbol{\lambda}(0) = \dot{\boldsymbol{\lambda}}(0) = \mathbf{0} \bigr\}.}
\end{equation}
We construct the controls in the complexification $U_{\mathrm{tr}, \mathbb{C}}^\varepsilon$. Since the equations have real coefficients, taking real parts gives physical signals. The approximation below respects the initial conditions and estimates the effect of the transducer travel times on $[0,T]$.

\medskip
\noindent
The time-domain scattering expansion requires the stronger causal class
\begin{equation}\label{eq:causal-C9-class}
 \Xnine:=\{\lambda|_{[0,T]}:\lambda\in C^9(\mathbb R;\mathbb C^{M_{\rm tr}}),
                  \ \lambda(t)=0\ (t<0)\},\qquad
 \|\lambda\|_{\Xnine}:=\sum_{j=0}^9\|\partial_t^j\lambda\|_{L^\infty(0,T)}.
\end{equation}
Functions in $\Xnine$ have vanishing initial derivatives through order nine. We use the $W^{2,\infty}$ norm for the control space $U_{\rm tr}^\varepsilon$ and the $\Xnine$ norm to estimate the physical scattering remainder.

\end{definition}

\subsection{Two-scale Geometry and Asymptotic Regime}\label{secdual}

\begin{assumption}[Material scaling]
\label{as111}
Let $D := \bigcup_{i=1}^M D_i$. The bulk modulus $\kappa$ and mass density $\rho$ are piecewise constant and scale with $\varepsilon$ as follows:
\begin{equation}\label{scale}
    [\kappa(x, \varepsilon), \rho(x, \varepsilon)] := 
    \begin{cases}
        [\kappa_c, \rho_c], & x \in \mathbb{R}^3 \setminus \overline{D}, \\
        [\overline{\kappa}_{b_i}(\varepsilon) \varepsilon^2, \overline{\rho}_{b_i} \varepsilon^2], & x \in D_i.
    \end{cases}
\end{equation}
The background parameters $(\kappa_c, \rho_c)$ and the scaled densities $\overline{\rho}_{b_i}$ are positive constants. We allow $\overline{\kappa}_{b_i}(\varepsilon)$ to depend on $\varepsilon$ in order to compensate for the $\mathcal{O}(\varepsilon^{1-p})$ resonance shifts, and assume $c \le \overline{\kappa}_{b_i}(\varepsilon) \le C$ for fixed constants $0<c\le C<\infty$. 

\medskip
\noindent
The geometric and resolvent smallness conditions of the forward expansion are assumed with uniform margins. Together with the material bounds, they give uniform constants in the Foldy--Lax remainder. The bare Minnaert frequency $\omega_{M,i}(\varepsilon) = \left( \frac{\rho_c}{2\overline{\kappa}_{b_i}(\varepsilon)} \Lambda_{\partial B_i} \right)^{-1/2}$ is then an $\mathcal{O}(1)$ parameter that can be tuned through the bulk modulus.
\end{assumption}

\begin{assumption}[Geometric and spectral separation]
\label{assum:dual-scale-geometry}
For a given {$p \in (0,2/3)$} and scale parameter $\varepsilon \ll 1$, let the index set $\mathcal{I} = \{1, \dots, M\}$ of $M$ bubbles be partitioned as $\mathcal{I} = \bigsqcup_{\alpha=1}^N \mathcal{I}_\alpha$. Thus $M$ is the number of microscopic bubbles, whereas $N$ is the number of macroscopic local clusters. We denote $M_\alpha:=|\mathcal I_\alpha|$ and require $M_\alpha\ge2$, so that $\sum_{\alpha=1}^N M_\alpha=M$. Let $z_i \in \mathbb{R}^3$ denote the center of bubble $D_i$, and let $y_\alpha$ denote the center of the local cluster $\mathcal I_\alpha$. We assume
\begin{enumerate}[(i)]
    \item \textit{Spectral detuning:} There exist physical Minnaert frequencies $\{\omega_{M,\alpha}\}_{\alpha=1}^N \subset \mathbb{R}^+$ such that $\omega_{M,i} = \omega_{M,\alpha}$ for all $i \in \mathcal{I}_\alpha$, with $\omega_{M,\alpha} \neq \omega_{M,\beta}$ for $\alpha, \beta$ belonging to distinct target frequency groups $\mathcal{N}_\ell \neq \mathcal{N}_k$. 
    \item \textit{Dual-scale geometry:} There exist $\varepsilon$-independent constants $c_1, c_2, D_{\min}, D_{\max} > 0$ such that the pairwise distances $d_{ij} := |z_i - z_j|$ satisfy
    \begin{equation*}
        d_{ij} \in 
        \begin{cases} 
        [c_1 \varepsilon^p, c_2 \varepsilon^p], & \forall i, j \in \mathcal{I}_\alpha, \; i \neq j \\[0.5em]
        [D_{\min}, D_{\max}], & \forall i \in \mathcal{I}_\alpha, \; \forall j \in \mathcal{I}_\beta, \; \alpha \neq \beta.
        \end{cases}
    \end{equation*}
\end{enumerate}
We use $0<p<2/3$ throughout the physical tracking result. This is the range of the forward asymptotics and the isolated resonance expansion in Section~\ref{proofprop}, where $h^3=o(\varepsilon)$. The collective block reduction and the fixed-time causal estimates only require $0<p<1$ and do not impose a collective damping condition.
\end{assumption}

\bigskip
\noindent
The effective number of resonant actuators is therefore $N$. In the target Minnaert band, each local cluster supplies a Perron mode spanning a one-dimensional subspace $E_\alpha \subset \mathbb{C}^{M_\alpha}$. The dominant cluster subspace $E_{\mathrm{dom}}$ and its associated spectral projection $\Pi_{\mathrm{dom}}$ satisfy
\begin{equation}
\label{eq:dominant-cluster-subspace}
  E_{\mathrm{dom}} := \bigoplus_{\alpha=1}^N E_\alpha \subset \mathbb{C}^M, \qquad \dim E_{\mathrm{dom}} = \operatorname{rank} \Pi_{\mathrm{dom}} = N.
\end{equation}
The principal subspace is a perturbation of this reference Perron subspace; directions associated with a common carrier may mix. We impose accessibility on the resulting $N$-dimensional input and output maps. This gives independent excitation and observation of the principal directions without choosing a basis of collective eigenvectors.

\noindent
Under this structural assumption, the exterior time-domain field admits the asymptotic expansion
\begin{equation}
\label{eq:physical-eff-error-local}
  u(x,t) = u^{\mathrm{in}}(x,t) + p^{\mathrm{eff},\varepsilon}(x,t) + \operatorname{err}_{\varepsilon}(x,t),
  \qquad (x,t)\in (\Omega\setminus \overline D)\times(0,T),
\end{equation}
with the incident-signal dependence of the remainder given in \eqref{eq:f-eps-H-1}. In particular, convergence holds for families for which that bound tends to zero. The effective scattered pressure is the point-source field
\begin{equation}
\label{eq:def-peff-intro}
 p^{\mathrm{eff},\varepsilon}(x,t) := \sum_{\alpha=1}^N \frac{1}{4\pi|x-y_\alpha|} Q_\alpha^\varepsilon\bigl(t - c_0^{-1}|x-y_\alpha|\bigr), \qquad Q_\alpha^\varepsilon := \sum_{i \in \mathcal{I}_\alpha} q_i^\varepsilon.
\end{equation}
\noindent
We take $V_{\mathrm{cl}}:=C([0,T];\mathbb R^N)$ as the output space. The compatible sources prescribed for the construction will belong to a smoother subclass, specified below.

\noindent
The physical realization operator is the time-domain map 
\[
  \mathcal{T}^\varepsilon : U_{\mathrm{tr}}^\varepsilon \longrightarrow V_{\mathrm{cl}}, \qquad \boldsymbol{\lambda} \longmapsto \mathbf{q}_{\mathrm{cl}}^\varepsilon,
\]
where $U_{\mathrm{tr}}^\varepsilon$ is the physical control input space defined in \eqref{eq:Utr-definition}, and $\mathbf{q}_{\mathrm{cl}}^\varepsilon := (Q_1^\varepsilon, \dots, Q_N^\varepsilon)$ denotes the vector of $N$ cluster-level source strengths. In the Laplace domain, the input-output mapping $\widehat{\mathbf{q}}_{\mathrm{cl}}^\varepsilon(s) = \mathcal{H}_{\mathrm{ext}}^\varepsilon(s)\widehat{\boldsymbol{\lambda}}(s)$ is governed by the exterior-to-cluster transfer matrix. 
The derived block reduction gives, in neighborhoods of the selected
positive carriers, the exact normal form
\begin{equation}
\label{eq:Hext-normal-form}
 \mathcal H_{\rm ext}^\varepsilon(s)
 =\sum_{\ell=1}^{L}C_\ell^\varepsilon
       (sI-S_\ell^\varepsilon)^{-1}R_\ell^\varepsilon
       B_\ell^\varepsilon(s)
       +\mathcal H_{\rm reg}^\varepsilon(s).
\end{equation}
Here $L$ is the number of distinct selected carriers, $n_\ell$ is
the number of clusters assigned to carrier $\widehat\omega_\ell$, and
\[
 \begin{aligned}
 C_\ell^\varepsilon
   &=B_{\rm out}\widetilde{\mathcal A}_bV_\ell^\varepsilon
                          \in\mathbb C^{N\times n_\ell},\\
 B_\ell^\varepsilon(s)
   &=(W_\ell^\varepsilon)^\top G_{\rm tr}(s)
                          \in\mathbb C^{n_\ell\times M_{\rm tr}},\\
 S_\ell^\varepsilon&=i\widehat\omega_\ell I
                            +\varepsilon A_\ell^\varepsilon,
 \qquad R_\ell^\varepsilon=\varepsilon\widetilde R_\ell^\varepsilon.
 \end{aligned}
\]
Proposition~\ref{prop:collective-block-reduction} gives uniform bounds for $A_\ell^\varepsilon$, uniform invertibility of $\widetilde R_\ell^\varepsilon$, and analyticity of the remainder near the principal packets. For $n_\ell=1$, the principal part reduces to a rank-one pole. For several clusters at a common carrier, we use the full block.

\medskip
\noindent
The matrix $G_{\mathrm{tr}}(s) := \bigl(G_s(z_i,x_m^{\mathrm{tr}})\bigr) \in \mathbb{C}^{M \times M_{\mathrm{tr}}}$ maps the transducer signals to the incident traces at the bubble centers. With the background Green kernel $G_s(z,x) = \frac{\rho_c}{4\pi |z-x|} e^{-s|z-x|/c_0}$, we have
\begin{equation}
\label{eq:incident-traces-transducers}
  \widehat{\mathbf{u}}^{\mathrm{in}}(s)\big|_{\{z_i\}_{i=1}^M} = G_{\mathrm{tr}}(s)\widehat{\boldsymbol{\lambda}}(s).
\end{equation}
The bubble dynamics map these incident traces to microscopic source strengths through
\begin{equation}
\label{eq:Hb-definition}
  H_b(s) := \mathcal{A}_b \mathcal{K}_{\mathrm{dyn}}(s)^{-1} B(s), \quad \text{where} \quad \mathcal{K}_{\mathrm{dyn}}(s) := I + \mathcal{D}(s)^{-1} s^2 \mathcal{Q}(s).
\end{equation}
In this formulation, $\mathcal{D}(s) = \operatorname{diag}(\omega_{M,i}^{-2} s^2 + 1)_{i=1}^M$, the local input operator $B(s) = \mathcal{D}(s)^{-1} s^2 B^{\mathrm{in}}(s)$ converts the incident traces into the forcing vector, and the delayed physical interaction matrix is defined as $\mathcal{Q}_{ij}(s) = \varepsilon\frac{\widetilde{\mathcal{C}}_j e^{-s\tau_{ij}}}{4\pi d_{ij}}$ for $i \neq j$. Finally, the cluster-output projection $B_{\mathrm{out}} : \mathbb{C}^M \to \mathbb{C}^N$ aggregates the resulting microscopic source amplitudes into macroscopic strengths via
\[
  (B_{\mathrm{out}}\mathbf{q})_\alpha = \sum_{i \in \mathcal{I}_\alpha} q_i.
\]
\noindent
We shall use the following properties of the singular part of $H_b(s)$.

\begin{figure}[H]
    \centering
\begin{tikzpicture}[scale=1.0]
    
    \def\sigmaZero{-8.0}  

    \def\omegaMalpha{6.6}     
    \def\omegaMbeta{1.4}     

    \def\xAlpha{-2.4}  
    \def\yAlpha{6.0}   

    \def\xBeta{-3.2}   
    \def\yBeta{2.0}    

    \def\gap{1.2}        

    \fill[gray!10] (\sigmaZero, -0.5) rectangle (2.5, 8.0);
    \draw[dashed, blue!60!black, thick] (\sigmaZero, -0.5) -- (\sigmaZero, 8.0);
    \node[blue!60!black, rotate=90] at (\sigmaZero-0.4, 3.5) {Strip $\mathcal{S}_{-\sigma_0}$};

    \draw[-stealth, line width=0.8pt] (-8.5, 0) -- (3.0, 0) node[right] {$\Re s$ (Damping)};
    \draw[-stealth, line width=0.8pt] (0, -0.5) -- (0, 8.0) node[above] {$\Im s$ (Frequency)};

    \draw[fill=black] (0, \omegaMalpha) circle (1.5pt);
    \node[right=2pt, font=\small, fill=white, inner sep=1pt] at (0, \omegaMalpha) {$i\omega_{M,\alpha}$};
    \draw[dashed, gray!60] (-4.5, \omegaMalpha) -- (0, \omegaMalpha);

    \draw[fill=black] (0, \omegaMbeta) circle (1.5pt);
    \node[right=2pt, font=\small, fill=white, inner sep=1pt] at (0, \omegaMbeta) {$i\omega_{M,\beta}$};
    \draw[dashed, gray!60] (-4.5, \omegaMbeta) -- (0, \omegaMbeta);

    \draw[thick, red!80!black, dashed, fill=red!10, fill opacity=0.4] (\xAlpha, \yAlpha) circle (\gap);

    \draw[fill=red!80!black] (\xAlpha, \yAlpha) circle (2pt);
    \node[above left=1pt, fill=white, inner sep=1pt, font=\small, text=red!80!black] at (\xAlpha, \yAlpha) {$\mathbf{s_\alpha^\varepsilon}$};

    \draw[-latex, thick, red!80!black] (\xAlpha, \yAlpha) -- ++(135:\gap) 
        node[midway, sloped, above=-1pt, font=\small] {$g_0 \varepsilon^{1-p}$};

    \draw[dotted, thick, red!80!black] (\xAlpha, \yAlpha) -- (0, \yAlpha);
    \draw[fill=red!80!black] (0, \yAlpha) circle (1.5pt);
    \node[right=2pt, font=\small, fill=white, inner sep=1pt, text=red!80!black] at (0, \yAlpha) {$i\omega_{\alpha}(\varepsilon)$};

    \draw[dotted, thick, red!80!black] (\xAlpha, \yAlpha) -- (\xAlpha, 0);
    \draw[fill=red!80!black] (\xAlpha, 0) circle (1.5pt);
    \node[below=2pt, font=\small, text=red!80!black] at (\xAlpha, 0) {$-\eta_{\alpha}$};

    \node[scale=1.2, red!80!black] at (-3.4, 7.0) {$\times$};
    \node[scale=1.2, red!80!black] at (-3.8, 5.2) {$\times$};
    \node[scale=1.2, red!80!black] at (-2.0, 4.6) {$\times$};

    \draw[thick, blue!70!black, dashed, fill=blue!10, fill opacity=0.4] (\xBeta, \yBeta) circle (\gap);

    \draw[fill=blue!70!black] (\xBeta, \yBeta) circle (2pt);
    \node[above left=1pt, fill=white, inner sep=1pt, font=\small, text=blue!70!black] at (\xBeta, \yBeta) {$\mathbf{s_\beta^\varepsilon}$};

    \draw[-latex, thick, blue!70!black] (\xBeta, \yBeta) -- ++(225:\gap) 
        node[midway, sloped, below=-1pt, font=\small] {$g_0 \varepsilon^{1-p}$};

    \draw[dotted, thick, blue!70!black] (\xBeta, \yBeta) -- (0, \yBeta);
    \draw[fill=blue!70!black] (0, \yBeta) circle (1.5pt);
    \node[right=2pt, font=\small, fill=white, inner sep=1pt, text=blue!70!black] at (0, \yBeta) {$i\omega_{\beta}(\varepsilon)$};

    \draw[dotted, thick, blue!70!black] (\xBeta, \yBeta) -- (\xBeta, 0);
    \draw[fill=blue!70!black] (\xBeta, 0) circle (1.5pt);
    \node[below=2pt, font=\small, text=blue!70!black] at (\xBeta, 0) {$-\eta_{\beta}$};

    \node[scale=1.2, blue!70!black] at (-4.2, 1.2) {$\times$};
    \node[scale=1.2, blue!70!black] at (-4.6, 2.7) {$\times$};

    \draw[<->, thick, purple!80!black, >=stealth] (-1.0, \yBeta+0.1) -- (-1.0, \yAlpha-0.1)
        node[midway, fill=gray!10, inner sep=2pt, font=\small, align=center] {Inter-carrier gap\\$ \gtrsim \Delta_\omega$};
    \draw[dotted, purple!80!black] (\xAlpha, \yAlpha) -- (-1.0, \yAlpha);
    \draw[dotted, purple!80!black] (\xBeta, \yBeta) -- (-1.0, \yBeta);

    \foreach \p in {(-6.5, 7.0), (-6.0, 3.5), (-5.0, 0.5), (-7.0, 2.0)}{
        \node[gray, scale=1.5] at \p {$\times$};
    }
    \node[gray, left=4pt, font=\footnotesize] at (-4.0, 4) {Background modes $\mathcal{P}$};

\end{tikzpicture}
\caption{Isolated local resonance branches of $H_b(s)$. The left-half-plane poles correspond to the isolated-cluster calculation. Groups at a common carrier are represented by matrix blocks, which may have repeated roots. The remaining symbols indicate complementary roots without specifying their half-plane location.} 
    \label{fig:pole-structure-multi-cluster-perfect}
\end{figure}

\begin{proposition}
\label{proposition}
For the physical input $B^{\rm in}=I$, under the material, geometric,
and local tuning assumptions, the microscopic transfer $H_b$ in
\eqref{eq:Hb-definition} has the following properties.
\begin{enumerate}[(i)]
\item It extends meromorphically to $\mathbb C$. Its possible poles
are characteristic roots of
$P(s,\varepsilon)=\mathcal D(s)+s^2\mathcal Q(s,\varepsilon)$;
there are finitely many in every bounded frequency region.
\item For the $n_\ell$ physical clusters assigned to a carrier
$\widehat\omega_\ell$, a complete principal packet lies in
$|s-i\widehat\omega_\ell|\le C\varepsilon$ and contains $n_\ell$
characteristic roots counted with algebraic multiplicity. Its gap from
the complementary spectrum is at least $ch$, $h=\varepsilon^{1-p}$;
different carrier packets have a fixed positive separation.
No separation or simplicity of roots within a packet is required.
\item Each packet has the matrix realization
\[
 H_b(s)=\widetilde{\mathcal A}_bV_\ell^\varepsilon
    (sI-S_\ell^\varepsilon)^{-1}R_\ell^\varepsilon
          (W_\ell^\varepsilon)^\top+H_{{\rm reg},\ell}(s),
\]
where $S_\ell^\varepsilon=i\widehat\omega_\ell I
+\varepsilon A_\ell^\varepsilon$, $\|A_\ell^\varepsilon\|\le C$,
and $R_\ell^\varepsilon=\varepsilon\widetilde R_\ell^\varepsilon$
with uniformly bounded $\widetilde R_\ell^\varepsilon$ and its inverse.
The input and output coordinates are $O(\varepsilon^p)$ perturbations
of the local Perron bases. The remainder is locally analytic.
\end{enumerate}
For a single active cluster, the block reduces to an isolated simple pole. For a group of clusters, the conclusion applies to the complete principal packet, with no assumption on the simplicity or damping of its individual roots.
\end{proposition}
\begin{proof}
Meromorphic continuation is proved in Section~\ref{proofprop}.
The derived complementary inverse, exact block factorization, and
packet realization are established in
Proposition~\ref{prop:collective-block-reduction}; the local analytic
remainder is estimated in Proposition~\ref{prop:Minnaert-gain}.
\par\noindent\hfill$\blacksquare$
\end{proof}

\begin{assumption}[Finite-distance transducer accessibility]
\label{ass:transducer-accessibility}
Let $\mathcal K_{\rm tr}\Subset(0,\infty)$ contain the operational bands,
and let $M_{\rm tr}\ge N$ transducers be fixed outside $\overline\Omega$.
With the carrier banks introduced below, $N=L N_M$, so we require $M_{\rm tr}\ge L N_M$ and full row rank over all cluster coordinates.
Define the center sampling matrix
\[
 \widetilde G_{\rm tr}(s)
 =\left(\frac{\rho_c}{4\pi|y_\alpha-x_m^{\rm tr}|}
             e^{-s|y_\alpha-x_m^{\rm tr}|/c_0}\right)_{
                         1\le\alpha\le N,\ 1\le m\le M_{\rm tr}}.
\]
We assume the \emph{full row-rank} inequality
\[
 \widetilde G_{\rm tr}(i\omega)\widetilde G_{\rm tr}(i\omega)^*
 \succeq c_{\rm tr}^2 I_N,
 \qquad \omega\in\mathcal K_{\rm tr},\qquad c_{\rm tr}>0.
\]
The inequality requires the smallest of the $N$ singular values to be uniformly positive at the finite transducer distances. 
\end{assumption}

\medskip
\noindent
\textit{Realizability of Assumption~\ref{ass:transducer-accessibility}.} 
The full-row-rank condition is satisfied by placing sufficiently many transducers on a sphere $x_m^{\mathrm{tr}} = R \theta_m$ with $R \gg 1$ and $\theta_m \in \mathbb{S}^2$. Uniformly for $\omega \in \mathcal{K}_{\mathrm{tr}}$, the retarded Green kernel admits the far-field expansion
\begin{equation*}
    G_{i\omega}(y_\alpha, x_m^{\mathrm{tr}}) = \frac{\rho_c e^{-i(\omega/c_0)R}}{4\pi R} e^{i(\omega/c_0)\theta_m \cdot y_\alpha} + \mathcal{O}(R^{-2}).
\end{equation*}
Up to an invertible scalar factor, the macroscopic evaluation matrix $\widetilde{G}_{\mathrm{tr}}(i\omega) \in \mathbb{C}^{N \times M_{\mathrm{tr}}}$ evaluates the functions $e_\alpha(\theta; \omega) := e^{i(\omega/c_0)\theta \cdot y_\alpha}$ at the discrete directions $\theta_m$. For any fixed $\omega>0$, these functions are linearly independent on $\mathbb{S}^2$: if $\sum_{\alpha=1}^N c_\alpha e_\alpha(\theta; \omega) \equiv 0$, Rellich's lemma and unique continuation force the associated radiating field to vanish identically. Matching the isolated point singularities at the distinct centers $y_\alpha$ gives $c_\alpha=0$. \newline
By the linear independence of these $N$ continuous functions, there exists a finite set of directions such that the discrete evaluation matrix $\mathbf{E}(\omega) := [e_\alpha(\theta_m; \omega)]_{N \times M_{\mathrm{tr}}}$ has full row rank. By the continuity of the determinant and the compactness of $\mathcal{K}_{\mathrm{tr}}$, a finite set $\{\theta_m\}_{m=1}^{M_{\mathrm{tr}}}$ can be chosen so that $\mathbf{E}(\omega)\mathbf{E}(\omega)^* \succeq \gamma_0 I_N$ for a uniform constant $\gamma_0 > 0$. The asymptotic expansion then yields $\widetilde{G}_{\mathrm{tr}}(i\omega)\widetilde{G}_{\mathrm{tr}}(i\omega)^* \succeq c_{\mathrm{tr}}^2 I_N$ for a fixed $R \gg 1$. The constant $c_{\mathrm{tr}} = \mathcal{O}(R^{-1})$ is strictly positive and independent of $\varepsilon$.

\bigskip

\noindent
The microscopic trace matrix $G_{\mathrm{tr}}(s) := \bigl(G_s(z_i,x_m^{\mathrm{tr}})\bigr) \in \mathbb{C}^{M\times M_{\mathrm{tr}}}$ satisfies
\begin{equation}
\label{eq:Gtr-cluster-perturbation}
  G_s(z_i,x_m^{\rm tr})=G_s(y_\alpha,x_m^{\rm tr})+\mathcal O(r_\varepsilon),
  \qquad
  r_\varepsilon:=\max_\alpha\max_{i\in\mathcal I_\alpha}|z_i-y_\alpha|,
\end{equation}
for all $i \in \mathcal{I}_\alpha$ and $s \in i\mathcal{K}_{\mathrm{tr}}$.

\medskip
\noindent
Let $\mathcal N_\ell$ be the physical clusters assigned to
$\widehat\omega_\ell$, and let $J_\ell:\mathbb C^{n_\ell}\to\mathbb C^N$
be their coordinate inclusion, as defined in (\ref{projection-J}). The block factorization derived in
Section~\ref{proofprop} provides spatial matrices
\begin{equation}\label{eq:spatial-components-def}
 C_\ell^\varepsilon=B_{\rm out}\widetilde{\mathcal A}_bV_\ell^\varepsilon,
 \qquad
 B_\ell^\varepsilon(s)=(W_\ell^\varepsilon)^\top G_{\rm tr}(s),
\end{equation}
with sizes $N\times n_\ell$ and $n_\ell\times M_{\rm tr}$.
Here $\widetilde{\mathcal A}_b=\varepsilon^{-1}\mathcal A_b$, and $V_\ell^\varepsilon,W_\ell^\varepsilon$ are well-conditioned coordinates obtained from the local Perron bases. Positivity and finite-distance sampling give the uniform spatial inverses in Proposition~\ref{prop:right-inverse-band}. We use these inverses for finite-time realization. Invertibility of the full harmonic transfer at the carrier would additionally require a collective resolvent condition.

\subsection{Spectral Separation and Spatial Controllability}
\noindent
Choose distinct Dirichlet frequencies $\Lambda := \{\hat{\omega}_\ell\}_{\ell=1}^L$, with eigenspaces $H_\ell := \operatorname{span}\{\psi_\ell^{(j)}\}_{j=1}^{m_\ell}$. The target space and its dimension are
$$H_M := \bigoplus_{\ell=1}^L H_\ell, \quad N_M := \sum_{\ell=1}^L m_\ell.$$
For the physical tracking theorem we take $N=L N_M$ macroscopic
clusters, partitioned as
\[
 \{1,\ldots,N\}=\bigsqcup_{\ell=1}^L\mathcal N_\ell,
 \qquad n_\ell:=|\mathcal N_\ell|=N_M,
 \qquad \tau(\alpha)=\ell\quad(\alpha\in\mathcal N_\ell).
\]
We denote the centers in bank $\ell$ by $y_{\ell,j}$, $j=1,\ldots,N_M$, and use $y_\alpha$ when a single index is sufficient. Each bank is tuned to one carrier and contains as many clusters as the dimension of the full target space. In particular, a dimer construction uses $M=2L N_M$ bubbles. This sufficient design, whose size need not be minimal, is fixed as $\varepsilon\to0$, together with the separation constants. The ideal point-source problem allows any $N\ge N_M$ for which the modal matrix has full row rank.

\begin{assumption}[Local principal resonance tuning]
\label{assum:precise-tuning}
Put $h=\varepsilon^{1-p}$, and let $\mathcal N_\ell=\tau^{-1}(\ell)$
be the physical clusters assigned to $\widehat\omega_\ell$.
For the physical tracking construction, $|\mathcal N_\ell|=N_M$ and $N=L N_M$. We impose the following tuning condition on each cluster.
For each cluster $\alpha$, write
\[
 d_\alpha=\omega_{M,\alpha}^{-2},\qquad
 \widetilde C_\alpha^{(0)}
 =\left(\frac{(1-\delta_{ij})\widetilde C_j}
                   {4\pi\widetilde d_{ij}}\right)_{i,j\in\mathcal I_\alpha},
 \qquad \widetilde d_{ij}=\varepsilon^{-p}|z_i-z_j|,
\]
and denote its Perron eigenvalue by $\widetilde\mu_{\alpha,1}^{(0)}$.
All these quantities are evaluated at the current admissible material
parameters; their dependence on $\varepsilon$ is allowed.
The parameters in Assumption~\ref{as111} are chosen so that
\begin{equation}\label{eq:local-principal-tuning}
 \left|1-\widehat\omega_{\tau(\alpha)}^2
       \bigl(d_\alpha+h\widetilde\mu_{\alpha,1}^{(0)}\bigr)\right|
       \le C\varepsilon ,\qquad \alpha=1,\ldots,N.
\end{equation}
Equivalently, the positive principal frequency
$(d_\alpha+h\widetilde\mu_{\alpha,1}^{(0)})^{-1/2}$ is within
$O(\varepsilon)$ of the assigned carrier. This entails the usual
$\omega_{M,\alpha}=\widehat\omega_{\tau(\alpha)}+O(h)$ correction.
This assumption concerns only the local conservative frequencies. Section~\ref{proofprop} derives the complementary inverse and collective block, allowing arbitrary eigenvalue multiplicities and without imposing invertibility or a numerical-range condition on the collective matrix.
\end{assumption}

\medskip
\noindent
\textit{Realizability of Assumption~\ref{assum:precise-tuning}.}
This condition is satisfied by the $\varepsilon$-dependence of the bubble bulk modulus from Assumption~\ref{as111}. For a fixed reference shape,
\[
 d_\alpha=\omega_{M,\alpha}^{-2}
 =\frac{\rho_c\Lambda_{\partial B_\alpha}}{2\overline\kappa_{b,\alpha}},
 \qquad
 \partial_{\overline\kappa_{b,\alpha}}d_\alpha
 =-\frac{\rho_c\Lambda_{\partial B_\alpha}}
         {2\overline\kappa_{b,\alpha}^{\,2}}\neq0.
\]
Given a reference bulk modulus satisfying $d_\alpha=\widehat\omega_{\tau(\alpha)}^{-2}$, an $\mathcal{O}(h)$ perturbation of $\overline\kappa_{b,\alpha}$ cancels the $h\widetilde\mu_{\alpha,1}^{(0)}$ interaction shift. Because this derivative is non-zero, if the Perron eigenvalue depends smoothly on $\overline\kappa_{b,\alpha}$, the implicit-function theorem yields a branch $\overline\kappa_{b,\alpha}(\varepsilon)$ that satisfies the conservative equality in \eqref{eq:local-principal-tuning} exactly for sufficiently small $\varepsilon$. This requires only that the unperturbed parameter resides in the interior of the allowed material domain. Consequently, Assumption~\ref{assum:precise-tuning} represents a parameter selection rather than an independent spectral hypothesis.

{
\begin{proposition}
\label{ass:target-Dirichlet}
Let $g_D := \frac{1}{2} \min_\ell \operatorname{dist}(\hat{\omega}_\ell, \{\omega_k\}_{k\ge1} \setminus \{\hat{\omega}_\ell\}) > 0$ represent the minimum spectral isolation from the unselected Dirichlet eigenfrequencies. Let $I_\alpha(\varepsilon) := [\hat{\omega}_{\tau(\alpha)} - \delta(\varepsilon), \hat{\omega}_{\tau(\alpha)} + \delta(\varepsilon)]$ with $\delta(\varepsilon) \in \left(0, \frac{g_D}{2} \varepsilon^{1-p}\right)$. Then, we have
\begin{equation}
    \operatorname{dist}\bigl(I_\alpha(\varepsilon), \{\omega_k\}_{k\ge1} \setminus \{\hat{\omega}_{\tau(\alpha)}\}\bigr) \ge \frac{g_D}{2} > 0, \quad \forall \alpha \in \{1,\dots,N\},
\end{equation}
independently of the gap $\Delta_\omega$ between distinct
assigned carriers.
\end{proposition}
}
\begin{proof}
    Every other Dirichlet angular frequency is at distance
    at least $2g_D$ from the selected carrier. For $0<\varepsilon\le1$,
    the triangle inequality therefore gives a distance at least
    $2g_D-\delta(\varepsilon)\ge g_D/2$. This separation concerns the Dirichlet spectrum independently of the bubble resonances.
\par\noindent\hfill$\blacksquare$
\end{proof}

\noindent
We define the composite Minnaert resonant band as the union of these isolated local bands:
\begin{equation}
    I_M := \bigcup_{\alpha=1}^N I_\alpha(\varepsilon).
\end{equation}

\noindent
Given an orthonormal basis $\Psi := (\psi_1, \dots, \psi_{N_M})^\top$ of $H_M$, define the spatial coupling matrix $C_M(\mathbf{y}) \in \mathbb{R}^{N_M \times N}$ by $[C_M(\mathbf{y})]_{i\alpha} := \psi_i(y_\alpha)$.

\begin{proposition}
\label{prop:generic-rank-CM}
For $N \ge N_M$, $\operatorname{rank}(C_M(\mathbf{y})) = N_M$ for Lebesgue almost every macroscopic cluster configuration $\mathbf{y} \in \Omega^N$.
\end{proposition}
\begin{proof}
Set $F: \Omega^{N_M} \to \mathbb{R}$, $F(y_1, \dots, y_{N_M}) := \det(\psi_i(y_j))_{i,j=1}^{N_M}$. The eigenfunctions are real-analytic in $\Omega$, so the same holds for $F$. Their linear independence gives $F\not\equiv0$. Hence $\mathcal{Z}_{N_M}:=F^{-1}(0)$ has Lebesgue measure zero by \cite[Ch. 6, Th. 6.3.3]{krantz2002primer}. For $\mathbf{y}\in(\Omega^{N_M}\setminus\mathcal{Z}_{N_M})\times\Omega^{N-N_M}$, the first $N_M$ columns of $C_M(\mathbf y)$ form an invertible matrix.
For $N=L N_M$, apply the same argument separately to the
$L$ banks: the union of their determinant-zero sets still has measure
zero. Hence all bank evaluation matrices are simultaneously
nonsingular for almost every configuration. Starting from a
nonsingular $N_M$-point configuration, small, fixed perturbations give
$L$ nonsingular banks with distinct centers. Their finite minimum
separations and boundary distances are positive and independent of
$\varepsilon$. Uniform conditioning over a family of configurations must be required separately from generic nonsingularity.
\par\noindent\hfill$\blacksquare$
\end{proof}

\noindent
To relate this rank condition to finite-dimensional controllability, consider the instantaneous oscillator associated with $C_M$; see \cite[Ch. 1]{TucsnakWeissBook}. Its first-order form is 
\begin{align*}
    \dot{z}(t) = Az(t) + B(\mathbf{y})u(t)
\end{align*}
on the state space $\mathbb{R}^{2N_M}$, where $z = (\mathbf{p}^\top, \dot{\mathbf{p}}^\top)^\top$. The system operator is the block matrix $A = \left(\begin{smallmatrix}0&I\\-\Lambda_M&0\end{smallmatrix}\right)$ with spectrum $\sigma(A)=\{\pm i\omega_k\}_{k=1}^{N_M}$, and the input operator is $B(\mathbf{y}) = \left(\begin{smallmatrix} 0 \\ c_0^2 C_M(\mathbf{y}) \end{smallmatrix}\right)$. By the Hautus test \cite[Sec. 1.5]{TucsnakWeissBook}, exact controllability requires 
\begin{align*}
    \operatorname{rank}\Big[\lambda I_{2N_M} - A \mid B(\mathbf{y})\Big] = 2N_M
\end{align*}
for all $\lambda\in\mathbb C$. The block structure shows that $\operatorname{rank}C_M=N_M$ suffices for the Hautus condition and for tracking arbitrary trajectories of this oscillator. Endpoint controllability can hold under weaker conditions when the frequencies differ. For the whole-space wave, we also account for the boundary trace through \eqref{eq:ws-right-inverse}.
\bigskip

\noindent For simple domains, such as cubes, full-rank configurations can be constructed explicitly; see the appendix of \cite{RodriguesSiniHeat}.

\begin{assumption}[Bank-wise subspace controllability]
\label{ass:Minnaert-coupling}
For the $N=L N_M$ clusters, let
$J_\ell:\mathbb C^{N_M}\to\mathbb C^N$ be the coordinate inclusion
of bank $\mathcal N_\ell$ and set
\[
 C_M^{(\ell)}:=C_M(\mathbf y)J_\ell
    =\bigl(\psi_i(y_{\ell,j})\bigr)_{1\le i,j\le N_M},
 \qquad \sigma_{\min}(C_M^{(\ell)})\ge c_{\rm bank}>0
 \quad(\ell=1,\ldots,L).
\]
The constant is independent of $\varepsilon$. In particular,
\[
 L_M^{(\ell)}:=J_\ell(C_M^{(\ell)})^{-1},
 \qquad C_ML_M^{(\ell)}=I_{N_M},
 \qquad \|L_M^{(\ell)}\|\le c_{\rm bank}^{-1}.
\]
Each bank spans the full target space. In particular, applying $(C_M^{(\ell)})^{-1}$ to a vector supported in $H_\ell$ cancels its components in the other selected eigenspaces, even when $C_M$ is not block diagonal. The evaluation matrix $C_M^{(\ell)}$ depends only on the centers and the Dirichlet eigenfunctions. The physical output matrix $C_\ell^\varepsilon$ in \eqref{eq:spatial-components-def} also contains the capacitance weights and Perron factors.

\medskip
\noindent
For the general ideal point-source theorem, $L_M$ may still denote
any fixed right inverse of $C_M$, for example
$L_M=C_M^\top(C_MC_M^\top)^{-1}$. Physical carrier matching instead
uses the separate inverses $L_M^{(\ell)}$ above.
\end{assumption}

\medskip
\noindent
\textit{Realizability of Assumption~\ref{ass:Minnaert-coupling}.}
This bank-wise rank condition is generic. Proposition~\ref{prop:generic-rank-CM} implies that an $N_M\times N_M$ evaluation matrix is nonsingular for almost every choice of $N_M$ centers. Applying this proposition independently to the $L$ banks ensures that all matrices $C_M^{(\ell)}$ can be chosen simultaneously invertible. For any fixed configuration, the minimum singular value $\min\limits_{1\le\ell\le L}\sigma_{\min}(C_M^{(\ell)})$ is strictly positive and independent of $\varepsilon$, since the macroscopic centers are fixed as $\varepsilon\to0$. For product domains, such as rectangles and cubes, full-rank configurations can be constructed explicitly from the tensor-product eigenfunctions.

\paragraph{The source-to-mode map.}
To include wave propagation across the boundary of the observation region, extend the source histories by zero to negative times and, for
$\mathbf q=(q_\alpha)_{\alpha=1}^N$, write
\[
 p^{\mathrm{ws}}[\mathbf q](x,t)
 =\sum_{\alpha=1}^N\frac{q_\alpha(t-|x-y_\alpha|/c_0)}{4\pi|x-y_\alpha|}.
\]
Denote its boundary trace by $g_{\mathbf q}$. Applying Green's second identity to regularized sources and passing to the limit gives
\begin{equation}\label{eq:ws-modal-identity}
 \ddot{\mathbf p}^{\mathrm{ws}}_M+\Lambda_M\mathbf p^{\mathrm{ws}}_M
 =c_0^2\CMws\mathbf q,
 \qquad \CMws:=C_M-\BM,
\end{equation}
where the exact causal boundary operator is
\begin{equation}\label{eq:boundary-delay-operator}
 (\BM\mathbf q)_k(t)
 =\sum_{\alpha=1}^N\int_{\partial\Omega}
 \frac{\partial_\nu\psi_k(x)}{4\pi|x-y_\alpha|}
 q_\alpha\!\left(t-\frac{|x-y_\alpha|}{c_0}\right)\,dS_x .
\end{equation}
Here $\mathbf p^{\mathrm{ws}}_M$ contains the coefficients of
$P_{I_M}p^{\mathrm{ws}}[\mathbf q]$. The sign in
\eqref{eq:ws-modal-identity} follows from
$\displaystyle\int_\Omega(-\Delta p)\psi_k
 =\lambda_k\int_\Omega p\psi_k+
 \int_{\partial\Omega}p\partial_\nu\psi_k$.
Since $\partial\Omega$ is $C^{1,1}$, we have $\partial_\nu\psi_k\in L^2(\partial\Omega)$. The separation of the source centers from the boundary makes these integrals well-defined and justifies the regularization limit.

\medskip
\noindent
Set
\[
 \tau_\partial=\min_\alpha\frac{\dist(y_\alpha,\partial\Omega)}{c_0}>0,
 \qquad
 b_M=\max_k\sum_\alpha\int_{\partial\Omega}
 \frac{|\partial_\nu\psi_k(x)|}{4\pi|x-y_\alpha|}\,dS_x<\infty.
\]
In component maximum norms, $\|\BM\|_{L^\infty\to L^\infty}\le b_M$; Euclidean norms introduce only a fixed dimensional factor. The finite-measure convolution inequality gives the corresponding $L^2(0,T)$ bound. Moreover, $\BM$ has a minimum delay $\tau_\partial$: its value at $t$ depends only on source values up to $t-\tau_\partial$. This positive delay allows us to invert the modal map by a finite series.

\begin{proposition}
\label{prop:ws-right-inverse}
Let $C_ML_M=I_{N_M}$ and let \eqref{eq:finite-propagation-localization}
hold. Define $K_M=\BM L_M$ and $n_T=\lfloor T/\tau_\partial\rfloor$.
Then
\begin{equation}\label{eq:ws-right-inverse}
 \LMws=L_M\sum_{n=0}^{n_T}K_M^n
       =L_M(I-K_M)^{-1},
 \qquad \CMws\LMws=I_{N_M}
\end{equation}
on causal inputs restricted to $[0,T]$. The inverse is bounded on
$L^2(0,T)$ and $L^\infty(0,T)$, without a smallness assumption on $K_M$.
It also preserves $C^r$ causal histories whose zero extensions are
$C^r$, and commutes with their time derivatives.
\end{proposition}
\begin{proof}
The operator $K_M^n$ delays its input by at least $n\tau_\partial$.
Thus $K_M^{n_T+1}=0$ on $[0,T]$, and multiplication of the finite geometric
sum by $I-K_M$ proves \eqref{eq:ws-right-inverse}. For either of the stated
norms,
\[
 \|\LMws f\|\le\|L_M\|
       \sum_{n=0}^{n_T}\|K_M\|^n\|f\|.
\]

\par\noindent\hfill$\blacksquare$
\end{proof}

{
\begin{definition}[Compatible source envelopes]
\label{def:multi-band-source-space}
Let $N=L N_M$, and assign carrier $\widehat\omega_{\tau(\alpha)}$ to cluster $\alpha$. For $m\ge2$ and $\varepsilon>0$, define
\begin{equation}
    V_{I_M}^m := \left\{ \mathbf{q}: [0,T] \to \mathbb{C}^N \;\middle|\; q_\alpha(t) = Q_\alpha^\varepsilon(\varepsilon t) e^{i\hat{\omega}_{\tau(\alpha)} t}, \quad \alpha \in \{1, \dots, N\} \right\},
\end{equation}
where $Q_\alpha^\varepsilon(\tau) := \beta(\varepsilon) A_\alpha(\tau)$, with $A_\alpha \in C^{m+1}([0,T]; \mathbb{C})$ and $\beta(\varepsilon) \in \mathbb{C}$, satisfying the flat initial conditions
\begin{equation}
\label{eq:VIM-time-compatibility}
    (Q_\alpha^\varepsilon)^{(k)}(0) = 0, \quad \forall k \in \{0, 1, \dots, m\}.
\end{equation}
We define the intrinsic norm on the fixed slow-time interval $[0,T]$ by
\begin{equation}
    \| \mathbf{q} \|_{\mathcal{V}_\varepsilon^m}
    := \sum_{\alpha=1}^N \sum_{j=0}^{m+1}
       \left\| (Q_\alpha^\varepsilon)^{(j)} \right\|_{L^\infty(0,T)}.
\end{equation}
Since $q_\alpha(t)=e^{i\hat\omega_{\tau(\alpha)}t}Q_\alpha^\varepsilon(\varepsilon t)$ and $0<\varepsilon\le1$, this fixed-interval norm controls the actually sampled part of the envelope:
\begin{equation}
\sum_{\alpha=1}^N\sum_{j=0}^{m+1}\varepsilon^{-j}
 \left\|\partial_t^j\!\left(e^{-i\hat\omega_{\tau(\alpha)}t}q_\alpha(t)\right)\right\|_{L^\infty(0,T)}
 =\sum_{\alpha=1}^N\sum_{j=0}^{m+1}
 \left\|(Q_\alpha^\varepsilon)^{(j)}\right\|_{L^\infty(0,\varepsilon T)}
 \le \|\mathbf q\|_{\mathcal V_\varepsilon^m}.
\end{equation}
\end{definition}
}

\paragraph{Modal carrier data.}
Let $\mathfrak F_{I_M}^m$ consist of specified tuples
$\boldsymbol f=(f_1,\ldots,f_L)$ of modal packets
\[
 f_\ell(t)=e^{i\widehat\omega_\ell t}F_\ell^\varepsilon(\varepsilon t),
 \qquad F_\ell^\varepsilon\in
 C^{m+1}([0,T];\mathbb C^{N_M}).
\]
We require
$(F_\ell^\varepsilon)^{(j)}(0)=0$ for $0\le j\le m+1$, so that
their zero extensions to negative slow times are $C^{m+1}$.
The vanishing derivative of order $m+1$ ensures continuity of the highest envelope derivative across each delayed starting time. Thus the boundary-delay operator preserves the required regularity without a smoothing estimate.
Set
\[
 \|\boldsymbol f\|_{\mathfrak F_\varepsilon^m}
 :=\sum_{\ell=1}^L\sum_{j=0}^{m+1}
       \|(F_\ell^\varepsilon)^{(j)}\|_{L^\infty(0,T)}.
\]

\begin{proposition}
\label{prop:bank-right-inverse}
Under Assumption~\ref{ass:Minnaert-coupling}, set
\[
 K_{M,\ell}:=\BM L_M^{(\ell)},\qquad
 \LMell:=L_M^{(\ell)}\sum_{n=0}^{n_T}K_{M,\ell}^{\,n},
 \qquad n_T=\lfloor T/\tau_\partial\rfloor.
\]
Then $\CMws\LMell=I_{N_M}$ on causal histories on $[0,T]$, and
$\LMell$ takes values only in bank $\mathcal N_\ell$.
For $\boldsymbol f\in\mathfrak F_{I_M}^m$, the source
$\mathbf q=c_0^{-2}\sum_\ell\LMell f_\ell$ satisfies
\[
 \mathbf q\in V_{I_M}^m,\qquad
 c_0^2\CMws\mathbf q=\sum_{\ell=1}^L f_\ell,\qquad
 \|\mathbf q\|_{\mathcal V_\varepsilon^m}
 \le c_0^{-2}C_{T,\rm bank}\,
                     \|\boldsymbol f\|_{\mathfrak F_\varepsilon^m}.
\]
The constant depends on the fixed banks, their conditioning and the
horizon, but not on $\varepsilon$ or a common amplitude factor.
\end{proposition}
\begin{proof}
Since $C_ML_M^{(\ell)}=I$, the strict-delay argument of
Proposition~\ref{prop:ws-right-inverse} gives
$K_{M,\ell}^{n_T+1}=0$ and
\[
 \CMws\LMell=(I-K_{M,\ell})
               \sum_{n=0}^{n_T}K_{M,\ell}^n=I.
\]
Every term has $L_M^{(\ell)}$ as its leftmost factor, so its output is supported in bank $\mathcal N_\ell$. For every delay $a\ge0$,
\[
 e^{-i\widehat\omega_\ell t}f_\ell(t-a)
 =e^{-i\widehat\omega_\ell a}
              F_\ell^\varepsilon(\varepsilon t-\varepsilon a).
\]
Zero extension makes this identity valid also for $t<a$.
Translations, constant matrices and finite-measure delay integrals
preserve the carrier and the $C^{m+1}$ envelope regularity.
Each slow derivative has the same bound. In fixed compatible vector
norms the bound for bank $\ell$ is controlled, up to a fixed
dimensional factor, by
\[
 \|L_M^{(\ell)}\|\sum_{n=0}^{n_T}
    \bigl(\|\BM\|\,\|L_M^{(\ell)}\|\bigr)^n
       \sum_{j=0}^{m+1}\|(F_\ell^\varepsilon)^{(j)}\|_\infty.
\]
The initial jets vanish because the undelayed term has the stated
flatness and all other terms have a strictly positive delay.
Summing over the disjoint banks proves the result. The finite series includes the entire boundary correction, with no smallness requirement on $K_{M,\ell}$.
\par\noindent\hfill$\blacksquare$
\end{proof}

\begin{remark}[On the number of resonant clusters]
Assigning \(N_M\) clusters to each of the \(L\) carriers, hence using
\[
    N=L N_M
\]
clusters, gives a carrier-preserving right inverse for every admissible target in
\[
    H_M=\bigoplus_{\ell=1}^{L}E_\ell.
\]
A component at \(\widehat\omega_\ell\) is generated by clusters tuned to that carrier, with the components in the other selected eigenspaces cancelled by the same bank.

\medskip
\noindent
Fewer clusters may suffice for a restricted objective. For example, to control a single simple mode \(\phi_\ell\) at \(\widehat\omega_\ell\), without cancelling the other modal components, the restricted evaluation map needs only one cluster at a point satisfying
\[
    \phi_\ell(y)\neq 0.
\]
For an eigenspace \(E_\ell\) of multiplicity \(m_\ell\), full row rank of the restricted evaluation matrix suffices, requiring at least \(m_\ell\) suitably placed clusters. Cancelling the induced components in the other selected eigenspaces imposes further spatial conditions. The choice \(N=L N_M\) is a sufficient construction for the full space \(H_M\) (i.e. a sufficient condition), not a minimal resonator count.
\end{remark}

\paragraph{Amplitude normalization.}
For a family of nonzero physical profiles, the factor
$\beta(\varepsilon)\ne0$ may be chosen to compensate the flat start by
requiring
\begin{equation}\label{eq:amplitude-normalization}
 0<c_*\le
 \max_\alpha\|Q_\alpha^\varepsilon\|_{L^\infty(0,T)}
 \le C_*<\infty,
 \qquad Q_\alpha^\varepsilon(\tau)=\beta(\varepsilon)A_\alpha(\tau),
\end{equation}
together with
\begin{equation}\label{eq:normalized-envelope-shape}
 |\beta(\varepsilon)|^{-1}
 \|\mathbf q^\varepsilon\|_{\mathcal V_\varepsilon^m}
 =\sum_{\alpha=1}^N\sum_{j=0}^{m+1}
\|A_\alpha^{(j)}\|_{L^\infty(0,T)}\le C_A.
\end{equation}
These conditions normalize the slow-envelope family on the fixed interval $[0,T]$. The lower bound concerns the vector envelope amplitude, so some cluster components may vanish. The factor $\beta$ remains in the homogeneous cost and error estimates. Indeed, Taylor's integral formula gives
\begin{equation}\label{eq:flat-start-norm-comparison}
 \|\mathbf q^\varepsilon\|_{L^\infty(0,T)}
 \le \frac{T^{m+1}}{(m+1)!}\,
       \varepsilon^{m+1}\|\mathbf q^\varepsilon\|_{\mathcal V_\varepsilon^m}.
\end{equation}
On the fixed physical interval, only the segment $Q_\alpha^\varepsilon([0,\varepsilon T])$ is sampled. Consequently, an order-one \emph{physical} source amplitude requires an order-one lower bound on that sampled segment; under the flat-start conditions this forces the intrinsic norm to be at least of order $\varepsilon^{-(m+1)}$. Our estimates are uniform operator bounds relative to this norm. A vanishing relative error for such amplitude-normalized targets would require a further estimate; it does not follow from \eqref{eq:amplitude-normalization}--\eqref{eq:normalized-envelope-shape} alone. In particular, the factor $|\beta|$ remains in $\varepsilon^p|\beta|\sum_{\alpha,j}\|A_\alpha^{(j)}\|_\infty$.

\subsection{Quantitative Estimation of the Intra-Cluster Spectral Gap}
\noindent
We give two examples of the local gap used to separate the Perron modes from the complementary microscopic modes. Let cluster $\alpha$ contain $M_\alpha$ bubbles with distances $d_{ij}=\widetilde d_{ij}\varepsilon^p$, $p\in(0,2/3)$. For these examples, take a fixed rescaled geometry and equal scaled capacitance weights $\mathcal C_\alpha(\varepsilon)$ within the cluster. The conservative local matrix is
\begin{equation*}
    \widetilde C_\alpha^{(0)}(\varepsilon):=\left(\frac{(1-\delta_{ij})\mathcal C_\alpha(\varepsilon)}{4\pi\widetilde d_{ij}}\right)_{i,j\in\mathcal I_\alpha}=\frac{\mathcal C_\alpha(\varepsilon)}{4\pi}\mathcal M_\alpha,
\end{equation*}
where the suppressed $\varepsilon$ dependence of $\mathcal C_\alpha$ is understood in the formulas below, and $\mathcal M_\alpha$ is the dimensionless interaction matrix with entries:
\begin{equation*}
    [\mathcal{M}_\alpha]_{ij} = (1 - \delta_{ij})\widetilde{d}_{ij}^{-1}.
\end{equation*}
The Perron--Frobenius theorem makes $\mu_1(\alpha)=\max\sigma(\mathcal M_\alpha)$ simple \cite[Ch. 8]{la1}. Perturbation of this eigenvalue gives the isolated conservative resonance shift \cite{kato1995perturbation},
\begin{equation} \label{eq:fine_shift}
    \omega_{\alpha, 1}(\varepsilon) = \omega_{M,\alpha} - \frac{\omega_{M,\alpha}^3 \mathcal{C}_\alpha}{8\pi} \mu_1(\alpha) \varepsilon^{1-p} + \mathcal{O}(\varepsilon^{\min(1, 2-2p)}).
\end{equation}
The matrix $\mathcal M_\alpha$ has zero trace, and its Perron eigenvalue is separated from the second eigenvalue by $\Delta\mu_\alpha:=\mu_1(\alpha)-\mu_2(\alpha)>0$. The corresponding frequency gap is
\begin{equation}
    \Delta_{\text{gap}}^{(\alpha)}(\varepsilon) := \frac{\omega_{M,\alpha}^3 \mathcal{C}_\alpha}{8\pi} \Delta\mu_\alpha \varepsilon^{1-p}.
\end{equation}
These formulas concern a single conservative cluster. Section~\ref{proofprop} treats unequal positive weights and the complete group at a common carrier. For equal weights, $\Delta\mu_\alpha$ can be described as follows.
\begin{enumerate}[(i)]
    \item \textit{Equidistant Cluster:} For uniform dimensionless distances $\widetilde{d}_{ij}=1$ (physically limited to $M_\alpha \le 4$ for embeddings in $\mathbb{R}^3$, representing a regular tetrahedron), the interaction matrix simplifies to $\mathcal{M}_\alpha = \mathbf{J} - \mathbf{I}$. The spectrum is $\sigma(\mathcal{M}_\alpha) = \{M_\alpha-1, -1, \dots, -1\}$, yielding a discrete gap $\Delta\mu_\alpha = M_\alpha > 0$.
     \item \textit{1D Uniform Lattice:} For a linear chain with spacing $h_\alpha$, the distance is $\widetilde{d}_{ij} = h_\alpha |i-j|$. This reduces $\mathcal{M}_\alpha$ to a scaled symmetric Toeplitz matrix $h_\alpha^{-1}\mathcal{T}^{(\alpha)}$ with entries:
    \begin{equation*}
        [\mathcal{T}^{(\alpha)}]_{ij} = (1 - \delta_{ij})|i-j|^{-1}.
    \end{equation*}
    By standard row-sum bounds for non-negative matrices \cite{la1}, the spectral radius $\mu_1(\mathcal{T}^{(\alpha)})$ satisfies:
    \begin{equation*}
        H_{M_\alpha-1} \le \mu_1(\mathcal{T}^{(\alpha)}) \le 2 H_{\lfloor M_\alpha/2 \rfloor}, \quad H_n := \sum_{k=1}^n k^{-1}.
    \end{equation*}
    The positive off-diagonal entries make $\mathcal T^{(\alpha)}$ irreducible. The Perron--Frobenius theorem \cite[Ch. 8]{la1} therefore gives
    \begin{equation*}
        C_{\alpha} := \mu_1(\mathcal{T}^{(\alpha)}) - \mu_2(\mathcal{T}^{(\alpha)}) > 0.
    \end{equation*}
\end{enumerate}
Choose a bandwidth larger than $C\varepsilon$ and smaller than a sufficiently small fixed fraction of $\min_\alpha\Delta_{\rm gap}^{(\alpha)}(\varepsilon)$. It separates the complete tuned principal packet from the non-principal microscopic modes, while allowing the roots within the packet to coincide. Proposition~\ref{prop:collective-block-reduction} proves this separation and the associated complementary inverse.

\begin{figure}[H]
    \centering
    \begin{tikzpicture}[
        x={(1cm, 0.2cm)},
        y={(-0.3cm, 0.9cm)},
        z={(0cm, 0.4cm)},
        >=stealth,
        font=\small\sffamily
    ]

\tikzset{
    bubble/.style={ball color=cyan!30, opacity=0.9},
    conn/.style={gray, thin},
    dist/.style={gray, <->, font=\scriptsize}
}

\begin{scope}[shift={(0,0,0)}]
    
    \coordinate (A) at (0,0,0);
    \coordinate (B) at (1.5,0,0);
    \coordinate (C) at (0.75, 1.3, 0);
    \coordinate (D) at (0.75, 0.43, 1.2);

    \draw[conn] (A) -- (B) -- (C) -- (A);
    \draw[conn] (A) -- (D); \draw[conn] (B) -- (D); \draw[conn] (C) -- (D);

    \draw[dist] ($(A)-(0,0.2,0)$) -- ($(B)-(0,0.2,0)$) node[midway, below] {$\widetilde{d}_{ij}=1$};

    \foreach \p in {A,B,C,D} \shade[bubble] (\p) circle (0.25cm);
    
    \node[text width=4cm, align=center, below] at (0.75, -0.8, 0) {
        $\Delta\mu = M_\alpha$ \\ (Maximum Separation)};
\end{scope}

\begin{scope}[shift={(6, 0.65, 0)}]
    \def\h{0.9} 
    \foreach \i in {0,...,5} {
        \coordinate (P\i) at (\i*\h, 0, 0);
    }

    \draw[conn, bend left=40] (P2) to node[midway, above, font=\tiny] {$1/h_\alpha$} (P3);
    \draw[conn, bend left=50, dashed] (P2) to node[midway, above, font=\tiny] {$1/2h_\alpha$} (P4);

    \draw[dist] ($(P0)-(0,0.3,0)$) -- ($(P1)-(0,0.3,0)$) node[midway, below] {$h_\alpha$};

    \foreach \i in {0,...,5} \shade[bubble] (P\i) circle (0.25cm);

    \node[text width=4cm, align=center, below] at (2.2, -1.45, 0) {
        $\mu_1 \approx 2H_{M_\alpha/2}$ \\ (Toeplitz Structure)};
\end{scope}

\end{tikzpicture}
    \caption{Cluster topologies: (Left) physical equidistant cluster and (Right) 1D uniform lattice with Toeplitz interaction structure.}
    \label{fig:cluster_geometries}
\end{figure}

\subsection{Control Objective and Main Result}

\noindent
\textit{Admissible profiles and the tracking problem.}
Let the target pressure profiles be slowly time-modulated resonant modes,
\begin{equation}\label{eq:reference-profile-primary-class}
 p_r^{I_M}(x,t)=\sum_{j=1}^{N_M}(\bm p_r(t))_j\psi_{k_j}(x),
 \qquad
 \bm p_r(t)=\sum_{\ell=1}^{L}e^{i\widehat\omega_\ell t}
 P_\ell^\varepsilon(\varepsilon t),
\end{equation}
where $P_\ell^\varepsilon\in C^{m+3}([0,T];H_\ell)$ satisfy the flat initial conditions specified below. Since
$\Lambda_M P_\ell^\varepsilon=\widehat\omega_\ell^2P_\ell^\varepsilon$, the modal residual admits the carrier decomposition:
\begin{equation}\label{eq:reference-resonance-compatibility-primary}
 \ddot{\bm p}_r+\Lambda_M\bm p_r
 =\sum_{\ell=1}^{L}f_\ell,
 \qquad
 f_\ell(t)=e^{i\widehat\omega_\ell t}F_\ell^\varepsilon(\varepsilon t),
\end{equation}
with
\begin{equation}\label{eq:F-from-P-primary}
 F_\ell^\varepsilon(\tau)
 =2i\widehat\omega_\ell\varepsilon(P_\ell^\varepsilon)'(\tau)
  +\varepsilon^2(P_\ell^\varepsilon)''(\tau).
\end{equation}
The slow modulation depends on the scaled variable $\varepsilon t$, with $P_\ell^\varepsilon$ and $F_\ell^\varepsilon$ defined on the fixed interval $[0,T]$.

\medskip
\noindent
More generally, the following analysis requires only a reference trajectory $p_r^{I_M}\in C^{m+3}([0,T];H_M)$ whose modal vector satisfies $\bm p_r^{(j)}(0)=0$ for $0\le j\le m+2$ and whose residual admits the prescribed packet decomposition
\begin{equation}
\label{eq:reference-resonance-compatibility}
 \ddot{\bm p}_r+\Lambda_M\bm p_r
        =\sum_{\ell=1}^L f_\ell,
 \qquad f_\ell(t)=e^{i\widehat\omega_\ell t}
                      F_\ell^\varepsilon(\varepsilon t),
 \qquad \boldsymbol f\in\mathfrak F_{I_M}^m.
\end{equation}
The class \eqref{eq:reference-profile-primary-class} is the principal example and produces \eqref{eq:reference-resonance-compatibility} through \eqref{eq:F-from-P-primary}.

\medskip
\noindent
Given this target profile on $[0,T]$, the tracking problem is formulated in two steps:
\begin{enumerate}[(i)]
\item construct the bank-wise ideal source \eqref{eq:main-ideal-source}, including its causal boundary correction, to yield exact tracking in the target spectral subspace;
\item realize this ideal source, up to the stated asymptotic error, by exterior incident waves.
\end{enumerate}
The prescribed modal envelopes ensure temporal compatibility, and Proposition~\ref{prop:bank-right-inverse} yields source profiles with the corresponding carrier structure.

\bigskip

\begin{theorem}[Minnaert-resonant projected pressure tracking]
\label{thm:Minnaert-tracking_main}
Let $T>0$ be fixed and let $m\ge10$. Let $U_{\rm tr}^\varepsilon$ be
\eqref{eq:Utr-definition}, and let $\Xnine$ be the causal regularity
class in \eqref{eq:causal-C9-class}. Set
$V=C([0,T];\mathbb C^N)$, and let $V_{I_M}^m\subset V$ be as in
Definition~\ref{def:multi-band-source-space}. Let $H_M$ and $P_{I_M}$
be the fixed target subspace and its orthogonal projector.

\medskip
\noindent
Assume \eqref{eq:finite-propagation-localization}, the material scaling of Assumption~\ref{as111}, the separation conditions of Assumption~\ref{assum:dual-scale-geometry} with $0<p<2/3$, the local tuning in Assumption~\ref{assum:precise-tuning}, the bank-wise controllability of Assumption~\ref{ass:Minnaert-coupling} with $N=L N_M$ and $n_\ell=N_M$, and the transducer accessibility of Assumption~\ref{ass:transducer-accessibility}. The forward estimates
are used with the uniform admissibility margins stated in
Proposition~\ref{prop:MS-expansion}. Use the carrier-preserving causal inverses $\LMell$ of
Proposition~\ref{prop:bank-right-inverse}.
The block kernel is given by the exact factorization in Proposition~\ref{prop:collective-block-reduction}. Then the following assertions hold.
\begin{enumerate}[(i)]
\item \textit{Projected field decomposition.}
Let $p^\varepsilon=u^{\mathrm{sc},\varepsilon}$ and let
$p^{\mathrm{eff},\varepsilon}$ be the whole-space field of the realized
cluster sources. We project the exterior observations using the zero-extension operator $\OMext$ in \eqref{eq:exterior-observation-projection}. With $\mu>0$
from the Foldy--Lax estimate,
\begin{equation}\label{eq:f-eps-H-1-thm}
 \sup_{t\le T}
 \|\OMext p^\varepsilon(t)-P_{I_M}p^{\mathrm{eff},\varepsilon}(t)\|_{H^1(\Omega)}
 \le C_{T,M}(\varepsilon^\mu+\varepsilon^p)
                  \|\boldsymbol\lambda^\varepsilon\|_{\Xnine}.
\end{equation}

\item \textit{Macroscopic right inverse.}
$\operatorname{rank}C_M^{(\ell)}=N_M$ and
$\CMws\LMell=I_{N_M}$ for each $\ell$. These inverses preserve the
physical bank, including every boundary-delay correction; $C_M$
itself need not be block diagonal.

\item \textit{Spatial conditioning.}
The active input and output matrices satisfy
\begin{equation}\label{eq:Hext-right-inverse-bound-thm}
 \sup_\ell\bigl(\|(C_\ell^\varepsilon)^\dagger\|
 +\|(B_\ell^\varepsilon(i\widehat\omega_\ell))^\dagger\|\bigr)\le C.
\end{equation}
For common-frequency clusters, these are the spatial maps in
\eqref{eq:spatial-block-factors}. Their conditioning follows from the local Perron bases, the block graph estimates, and transducer accessibility.

\item \textit{Causal realization.}
The construction in Proposition~\ref{prop:surjectivity-IM} satisfies
\[
 \|R_{I_M}^\varepsilon\mathbf q\|_{\Xnine}
 +\|R_{I_M}^\varepsilon\mathbf q\|_{W^{2,\infty}}
 \le C\|\mathbf q\|_{\mathcal V_\varepsilon^m},
 \qquad
 \|\mathcal T^\varepsilon R_{I_M}^\varepsilon\mathbf q-\mathbf q\|_\infty
 \le C(\varepsilon+\varepsilon^p)\|\mathbf q\|_{\mathcal V_\varepsilon^m}.
\]
The constructed controls have a $C^9$ zero extension and satisfy the
physical rest compatibility conditions.
\end{enumerate}

Let $p_r^{I_M}=\sum_{k=1}^{N_M}(\bm p_r)_k\psi_k$ with
$\bm p_r\in C^{m+3}([0,T];\mathbb C^{N_M})$ and
$\bm p_r^{(j)}(0)=0$ for $0\le j\le m+2$. Assume the prescribed modal packet decomposition
\eqref{eq:reference-resonance-compatibility}. Define the bank-wise
ideal source by
\begin{equation}\label{eq:main-ideal-source}
 \mathbf q^{\mathrm{ideal}}
 =c_0^{-2}\sum_{\ell=1}^L\LMell f_\ell\in V_{I_M}^m.
\end{equation}
Proposition~\ref{prop:bank-right-inverse} gives the stated membership in $V_{I_M}^m$. The associated whole-space field has the prescribed projected pressure, since
\begin{equation}\label{ref-pressure}
 \ddot{\bm p}_r+\Lambda_M\bm p_r
 =c_0^2\CMws\mathbf q^{\mathrm{ideal}},
 \qquad \bm p_r(0)=\dot{\bm p}_r(0)=0.
\end{equation}
Set $\boldsymbol\lambda^\varepsilon
 =R_{I_M}^\varepsilon\mathbf q^{\mathrm{ideal}}$ and
$\gamma=\min\{\mu,p\}$. There are constants independent of
$\varepsilon$ and of the multiplicative amplitude factor $\beta$ such
that
\begin{equation}\label{eq:Minnaert-tracking-final}
 \sup_{t\le T}\|\OMext p^\varepsilon(t)-p_r^{I_M}(t)\|_{H^1(\Omega)}
 \le C_{T,M,m}\varepsilon^\gamma
                  \|\mathbf q^{\mathrm{ideal}}\|_{\mathcal V_\varepsilon^m},
\end{equation}
\[
 \|\boldsymbol\lambda^\varepsilon\|_{U_{\rm tr}^\varepsilon}
 +\|\boldsymbol\lambda^\varepsilon\|_{\Xnine}
 \le \kappa_{T,M,m}\|\mathbf q^{\mathrm{ideal}}\|_{\mathcal V_\varepsilon^m}.
\]
Moreover,
$\|\mathbf q^{\mathrm{ideal}}\|_{\mathcal V_\varepsilon^m}
\le c_0^{-2}C_{T,\rm bank}
\|\boldsymbol f\|_{\mathfrak F_\varepsilon^m}$, so the same error and
control bounds hold with the prescribed modal packet norm on the
right, after changing the constant. The constants are uniform for fixed $L,N_M,N,M,M_{\rm tr}$ and fixed geometric and conditioning margins. They are operator bounds relative to the indicated norm. For nonvanishing targets normalized by \eqref{eq:amplitude-normalization}--\eqref{eq:normalized-envelope-shape}, the norm dependence in \eqref{eq:flat-start-norm-comparison} must also be taken into account.
\end{theorem}



\subsection*{Organisation of the remaining parts of the paper}
\noindent
Section~\ref{sec:abstract-wave} proves exact projected tracking for ideal point sources. Section~\ref{sec:bubble-R3} derives the source representation from the bubble scattering problem and estimates the exterior pressure error. The local resonance calculation and the collective block reduction are given in Section~\ref{proofprop}. Section~\ref{sec:resonant-gain} constructs the causal exterior controls and compares resonant and non-resonant costs. The projected pressure estimates in Section~\ref{sec:proof-main} complete the proof of Theorem~\ref{thm:Minnaert-tracking_main}.
\section{Ideal Point-Actuator Tracking}
\label{sec:abstract-wave}
\noindent
We first treat the strengths of the point sources at $y_\alpha$ as independent controls. The causal inverse constructed above then gives exact tracking of the selected Dirichlet modes. 

\subsection{Abstract Formulation and Well-posedness}
\noindent
Consider the restriction to $\Omega\subset\mathbb R^3$ of the whole-space source field, 
\begin{align}\label{eq:wave-control}
    \begin{cases}
        \dfrac{1}{c_0^2}\,\partial_t^2 p - \Delta p = \displaystyle\sum_{\alpha=1}^N q_\alpha(t)\,\delta_{y_\alpha}
        & \text{in }\Omega\times({0},T),\\
        p=g_{\mathbf q} & \text{on }\partial\Omega\times({0},T),\\
        p(\cdot,{0})=0,\quad \partial_t p(\cdot,{0})=0.
    \end{cases}
\end{align}
As noted above, the Dirac sources prevent a direct formulation in $H^1_0(\Omega)\times L^2(\Omega)$. We instead project the equation onto the target space. Let $\{\psi_k\}_{k\ge1}$ be the Dirichlet Laplacian eigenfunctions, forming an orthonormal basis of $L^2(\Omega)$, and let $H_M := \mathrm{span}\{\psi_1,\dots,\psi_{N_M}\}$ denote the $N_M$-dimensional target subspace corresponding to the macroscopic Minnaert composite band $I_M$, with orthogonal projection $P_{I_M} : L^2(\Omega)\to H_M$. Given the cluster-level source vector $\mathbf q=(q_1,\dots,q_N)^\top\in {W^{1,\infty}(0,T;\mathbb R^N)}$, we seek
\[
p_M(t)\in H_M,\qquad t\in[{0},T],
\]
such that, for a.e.\ $t\in({0},T)$,
\begin{equation}
\label{eq:galerkin-weak}
 \frac{1}{c_0^2}\langle \partial_t^2 p_M(\cdot,t), \varphi\rangle
 + \langle \nabla p_M(\cdot,t), \nabla \varphi\rangle
 = \sum_{\alpha=1}^N q_\alpha(t)\,\varphi(y_\alpha)-\int_{\partial\Omega}g_{\mathbf q}(x,t)\partial_\nu\varphi(x)\,dS_x,
 \qquad \forall \varphi\in H_M,
\end{equation}
with zero initial data
\begin{equation}
\label{eq:IC-galerkin}
 p_M(\cdot,{0})=0,\qquad \partial_t p_M(\cdot,{0})=0.
\end{equation}
Since $H_M\subset H^2(\Omega)$ is spanned by smooth eigenfunctions, both the point evaluations and the boundary integrals in \eqref{eq:galerkin-weak} are well-defined.

\begin{proposition}
\label{prop:well-posed}
Let $T>0$ and $\mathbf q\in L^2({0},T;\mathbb R^N)$. Then the Galerkin problem
\eqref{eq:galerkin-weak}--\eqref{eq:IC-galerkin} admits a unique solution
\[
 p_M\in H^2({0},T;H_M).
\]
Writing $p_M(x,t)=\sum_{k=1}^{N_M} p_k(t)\psi_k(x)$, the coefficients satisfy the decoupled forced oscillator system
\begin{equation}
\label{eq:mode-ODE-N}
 \ddot p_k(t)+\omega_k^2 p_k(t)
 = c_0^2 \sum_{\alpha=1}^N q_\alpha(t)\psi_k(y_\alpha)-c_0^2(\BM\mathbf q)_k(t),
 \qquad k=1,\dots,N_M,
\end{equation}
with initial data $p_k({0})=0$ and $\dot p_k({0})=0$.
Moreover, for a constant $C_{T,M}>0$ depending on $T$ and $H_M$, but not on $\mathbf q$, we have
\begin{equation}
\label{eq:galerkin-energy-est}
 \sup_{t\in[{0},T]}
 \Bigl(
  \norm{\nabla p_M(\cdot,t)}_{L^2(\Omega)}
  + \norm{\partial_t p_M(\cdot,t)}_{L^2(\Omega)}
 \Bigr)
 \le C_{T,M} \norm{\mathbf q}_{L^2({0},T;\mathbb R^N)}.
\end{equation}
\end{proposition}

\begin{proof}
Insert $p_M(x,t)=\sum\limits_{k=1}^{N_M}p_k(t)\psi_k(x)$ into \eqref{eq:galerkin-weak} and test with $\varphi=\psi_k$. This gives
\[
\frac{1}{c_0^2}\ddot p_k(t) + \lambda_k p_k(t)
= \sum_{\alpha=1}^N q_\alpha(t)\psi_k(y_\alpha)-(\BM\mathbf q)_k(t).
\]
Multiplication by $c_0^2$ gives \eqref{eq:mode-ODE-N}, since $\omega_k^2=c_0^2\lambda_k$. With
\[
 \bm p^M(t):=(p_1(t),\ldots,p_{N_M}(t))^\top,
 \qquad \Lambda_M:=\operatorname{diag}(\omega_1^2,\ldots,\omega_{N_M}^2),
\]
and $C_M(\mathbf y)\in\mathbb R^{N_M\times N}$, $[C_M(\mathbf y)]_{k\alpha}:=\psi_k(y_\alpha)$, the system reads
\begin{equation}\label{eq:modal-matrix}
\ddot{\bm p}^M(t)+\Lambda_M\bm p^M(t)=c_0^2(\CMws\mathbf q)(t).
\end{equation}
For prescribed $\mathbf q\in L^2(0,T)$, the forcing
$\mathbf F=c_0^2\CMws\mathbf q$ is in $L^2(0,T)$ by
\eqref{eq:boundary-delay-operator}. The oscillator system has a unique
$H^2$ modal solution with zero initial data. Define
\[
 E_M(t)=\tfrac12\bigl(|\Lambda_M^{1/2}\bm p^M(t)|^2
                         +|\dot{\bm p}^M(t)|^2\bigr).
\]
Then $E_M'=\dot{\bm p}^M\cdot\mathbf F$ and the integrated energy
inequality gives
\[
 \sup_{t\le T}\sqrt{E_M(t)}\le C\sqrt T\,\|\mathbf F\|_{L^2(0,T)}
 \le C_{T,M}\|\mathbf q\|_{L^2(0,T)}.
\]
Using $\omega_k^2=c_0^2\lambda_k$, we obtain \eqref{eq:galerkin-energy-est}.

\par\noindent\hfill$\blacksquare$
\end{proof}

\subsection{Tracking on a Finite Spectral Subspace}
\noindent
Let $T>0$ and $\bm{p}_r\in C^{m+3}([{0},T];\R^{N_M})$ be a target reference trajectory for the modes in the Minnaert band. Define the corresponding reference field
\[
 p_r^{I_M}(x,t):=\sum_{k=1}^{N_M} (\bm{p}_r(t))_k\psi_k(x).
\]
Compatibility with the zero initial data requires
\begin{equation}
\label{eq:IC-match}
 (\bm{p}_r({0}))_k=0,\qquad
 (\dot{\bm{p}}_r({0}))_k= 0,\qquad k=1,\dots,N_M.
\end{equation}
\noindent
Choose $N\ge N_M$ and a configuration for which $C_M(\mathbf y)\in\mathbb R^{N_M\times N}$ has full row rank, as holds generically by Proposition~\ref{prop:generic-rank-CM}. Fix a right inverse $L_M\in\mathbb R^{N\times N_M}$ satisfying
\[
    C_M(\mathbf y)L_M=I_{N_M}.
\]
One may take the Moore--Penrose inverse in Assumption~\ref{ass:Minnaert-coupling}. Proposition~\ref{prop:ws-right-inverse} then gives $\CMws\LMws=I_{N_M}$. No block structure of $C_M$ is needed for this ideal result. For physical realization, the same argument uses the bank-preserving inverses of Proposition~\ref{prop:bank-right-inverse}.

\begin{theorem}[Exact tracking on the target subspace]
\label{thm:finite-tracking}
Let $N \ge N_M$ and let the macroscopic centers $\mathbf y$ satisfy the generic rank condition. Let $T>0$ and $\bm{p}_r\in C^{m+3}([{0},T];\R^{N_M})$ satisfy \eqref{eq:IC-match}. Fix any right inverse $L_M$ satisfying $C_M(\mathbf y)L_M=I_{N_M}$ and define the ideal actuator control vector:
\begin{equation}
\label{eq:control-ideal}
 \mathbf{q}^{\mathrm{ideal}}(t):=c_0^{-2}\LMws\bigl(\ddot{\bm p}_r+\Lambda_M\bm p_r\bigr)(t),\qquad t\in[{0},T].
\end{equation}
Let $\bm{p}^M$ be the unique solution of \eqref{eq:modal-matrix} driven by $\mathbf{q}=\mathbf{q}^{\mathrm{ideal}}$ with zero initial data. Then
\[
 \bm p^M(t)=\bm p_r(t),\qquad t\in[0,T],
\]
and consequently
\[
 P_{I_M}p^{\mathrm{ws}}[\mathbf q^{\mathrm{ideal}}](\cdot,t)
 =p_r^{I_M}(\cdot,t),\qquad t\in[0,T].
\]
\end{theorem}

\begin{proof}
The regularity of $\bm p_r\in C^{m+3}([{0},T];\R^{N_M})$ and boundedness of $\LMws$ make \eqref{eq:control-ideal} an admissible $L^2$ source. Proposition~\ref{prop:well-posed} therefore gives a unique modal solution $\bm p^M\in H^2({0},T;\R^{N_M})$.

\medskip
\noindent
Using \eqref{eq:control-ideal} and
$\CMws\LMws=I_{N_M}$ from Proposition~\ref{prop:ws-right-inverse},
\[
 c_0^2\CMws\mathbf q^{\mathrm{ideal}}
 =\ddot{\bm p}_r+\Lambda_M\bm p_r.
\]
Thus \eqref{eq:modal-matrix} becomes
\[
 \ddot{\bm{p}}^M(t)+\Lambda_M\bm{p}^M(t)
 =\ddot{\bm{p}}_r(t)+\Lambda_M\bm{p}_r(t),\qquad t\in[{0},T].
\]
Define the tracking error $\bm{e}(t):=\bm{p}^M(t)-\bm{p}_r(t)$. Then $\bm e$ satisfies
\[
 \ddot{\bm{e}}(t)+\Lambda_M\bm{e}(t)=0,\qquad t\in[{0},T],
\]
with initial conditions
\[
 \bm{e}({0})=0,\qquad \dot{\bm{e}}({0})=0,
\]
by \eqref{eq:IC-match}. Uniqueness for the homogeneous oscillator gives $\bm e\equiv0$ on $[{0},T]$ and proves the result.
\par\noindent\hfill$\blacksquare$
\end{proof}
\section{Bubble Dynamics and the Transfer Operator}\label{sec:bubble-transfer}
\label{sec:bubble-R3}

\noindent
The preceding section prescribes the ideal source strengths. We now identify the source strengths produced by the bubbles and express their dependence on the exterior incident wave. This relation is the transfer operator to be inverted in the control construction.

\subsection{Scattered Field as a Superposition of Retarded Monopoles}
\label{subsec:time-domain-asymptotics}
\noindent
For the transmission problem \eqref{eq:MS-physical}, transducers at $x_m^{\mathrm{tr}}\in\R^3\setminus\overline\Omega$ with signals $\lambda\in U_{\rm tr}^\varepsilon$ generate the incident wave
\begin{equation}
\label{eq:incident-field-R3}
 u^{\mathrm{in}}(x,t)
 =
 \sum_{m=1}^{M_{\rm tr}}
 \frac{\rho_c}{4\pi|x-x_m^{\mathrm{tr}}|}
 \,\lambda_m\bigl(t-c_0^{-1}|x-x_m^{\mathrm{tr}}|\bigr).
\end{equation}
The time-domain analysis in \cite{MukherjeeSiniJEE,MukherjeeSiniSIAM} gives the following expansion of $u^{\mathrm{sc}}:=u-u^{\mathrm{in}}$.

\begin{proposition}
\label{prop:MS-expansion}
let Assumption \ref{as111} and Assumption \ref{assum:dual-scale-geometry} be satisfied. Let $\boldsymbol\lambda\in\Xnine$, and suppose that the smallness conditions of the cited forward theorem hold with uniform margins. Then there exist
\begin{enumerate}
  \item constants $\mathcal{C}_i(\varepsilon)=\varepsilon \tilde{\mathcal{C}}_i$, where $\widetilde C_i=\widetilde C_i(\varepsilon)>0$ are uniformly bounded above and below and depend on the geometry and the admissible material parameters;
  \item scalar functions $Y_i\in C^2([{0},T])$, $i=1,\dots,M$, solving a finite-dimensional delayed system (see \eqref{eq:Y-tilde-system-vector} below);
  \item a remainder $R^\varepsilon\in C\big([{0},T];L^2_{\mathrm{loc}}(\R^3\setminus \overline{D})\big)$;
\end{enumerate}
such that for $x$ in bounded regions away from the cluster,
\begin{equation}
\label{eq:MS-expansion-again}
 u^{\mathrm{sc}}(x,t)
 =
 -\sum_{i=1}^M C_i(\varepsilon)\,
 \frac{1}{4\pi|x-z_i|}
 \,Y_i\bigl(t-c_0^{-1}|x-z_i|\bigr)
 + R^\varepsilon(x,t),
 \qquad t\in[{0},T],
\end{equation}
and, for the chosen bounded localization domain $\Omega$ containing all bubbles, there exist $\mu>0$ and $C_{T,\Omega}>0$ independent of $\varepsilon$ such that
\begin{equation}
\label{eq:R-eps-estimate-again}
 \sup_{t\in[{0},T]}\norm{R^\varepsilon(\cdot,t)}_{L^2\big(\Omega\setminus \overline{D}\big)}
 \le C_{T,\Omega}\varepsilon^\mu \|\boldsymbol\lambda\|_{\Xnine}.
\end{equation}
The amplitudes $\big(\mathrm{Y}_i\big)_{i=1}^M$ solve the delayed system
\begin{align}
\label{eq:Y-tilde-system}
    \begin{cases}
        \omega_{M,i}^{-2} \dfrac{\mathrm{d}^2}{\mathrm{d} t^2} \mathrm{Y}_i(t) + \mathrm{Y}_i(t) + \sum\limits_{\substack{j=1 \\ j \neq i}}^M q_{ij} \dfrac{\mathrm{d}^2}{\mathrm{d} t^2} \mathrm{Y}_j\bigl(t - \tau_{ij}\bigr) = \dfrac{\partial^2}{\partial t^2} u^\textit{in}(z_i,t), &\textit{in} \; ({0}, T), \\
        \mathrm{Y}_i({0}) = \dfrac{\mathrm{d}}{\mathrm{d} t} \mathrm{Y}_i({0}) = 0,
    \end{cases}
\end{align}
where $\omega_{M,i}^{-2} = \dfrac{\rho_c}{2 \overline{\kappa}_{\mathrm{b},i}} \Lambda_{\partial B_i}$ represents the inverse square of the Minnaert frequency $\omega_{M,\alpha}$ of the bubble $D_i$. For each $i=1,\dots,M$, the parameters satisfy
\begin{align*}
    \omega_{M,i}^{-2} > 0,
\qquad \tau_{ij} = c_0^{-1}|z_i - z_j| \ge 0 \ (j=1,\dots,M),
\qquad q_{ij} = \frac{C_j(\varepsilon)}{4\pi |z_i - z_j|} \in \mathbb R,
\end{align*}
with the convention $q_{ii} = 0$ and $\tilde{C}_j := \operatorname{vol}(B_j)\frac{\rho_c}{\overline{\kappa}_{\mathrm{b},j}}.$ Also, $\displaystyle\Lambda_{\partial B_i} := \frac{1}{|\partial B_i|} \int_{\partial B_i} \int_{\partial B_i}
 \frac{(\mathrm{x} - \mathrm{y}) \cdot \nu_\mathrm{x}}{|\mathrm{x} - \mathrm{y}|}
 \, d\sigma_\mathrm{x} \, d\sigma_\mathrm{y}$ is a geometric constant. The well-posedness of the system of differential equations \eqref{eq:Y-tilde-system} is discussed in \cite[Section 2.4]{MukherjeeSiniJEE}. In vector form, \eqref{eq:Y-tilde-system} can be written as
\begin{equation}
\label{eq:Y-tilde-system-vector}
 \mathcal{D}\,\ddot{\mathbf{Y}}(t)
 + \mathrm{Y}(t)
 + \mathcal{Q}\bigl[\ddot{\mathbf{Y}}\bigr](t)
 = \boldsymbol{\mathcal{F}}(t),\qquad
 \mathrm{Y}({0})=\dot{\mathrm{Y}}({0})=0,
\end{equation}
where (i) $\mathcal{D}:=\mathrm{diag}(d_1,\dots,d_M)$, (ii) $\mathcal{Q}[\ddot{\mathrm{Y}}](t)$ has components
\[
 \bigl(\mathcal{Q}[\ddot{\mathrm{Y}}](t)\bigr)_i
 :=\sum_{j\neq i}q_{ij}\,\ddot{\mathrm{Y}}_j\bigl(t-\tau_{ij}\bigr),
 \quad
 \tau_{ij}:=c_0^{-1}|z_i-z_j|,
\]
and (iii) $\boldsymbol{\mathcal{F}}(t)\in\R^M$ has $i$-th component $\dfrac{\partial^2}{\partial t^2} u^\textit{in}(z_i,t)$.
\end{proposition}

\begin{remark}[Temporal regularity inherited from the forward theory]
Theorem~1.1 of \cite{MukherjeeSiniJEE} assumes a causal signal
$\lambda\in C^9(\mathbb R)$. Section~2.3 there records the higher
solution and trace regularity used in the asymptotic expansion.
We use this regularity in the physical realization theorem and keep its dependence in the $\Xnine$ norm. The choice $m\ge10$ ensures that the controls constructed below have the required $C^9$ zero extension; arbitrary $W^{2,\infty}$ signals are not covered by this forward expansion.
\end{remark}

\noindent
Define
\[
 q_i^\veps(t):=-C_i(\veps)\,Y_i(t),\qquad i=1,\dots,M,
\]
and $q^\veps=(q_1^\veps,\dots,q_M^\veps)$. Then \eqref{eq:MS-expansion-again} becomes
\begin{equation}
\label{eq:retarded-actuators}
 u^{\mathrm{sc}}(x,t)
 =
 \sum_{i=1}^M
 \frac{1}{4\pi|x-z_i|}\,q_i^\veps\bigl(t-c_0^{-1}|x-z_i|\bigr)
 + R^\veps(x,t).
\end{equation}
Introduce the retarded fundamental solution of the background wave operator,
\begin{equation}
\label{eq:fundamental-wave}
 \mathcal{G}_0(x,t;y)
 :=
 \frac{1}{4\pi|x-y|}\,\delta\bigl(t-c_0^{-1}|x-y|\bigr),
 \qquad x\neq y,\ t>0.
\end{equation}
Given a time-dependent source $f\in\mathscr{S}'(\R)$, the retarded field generated by a point source at $y\in\R^3$ is
\begin{equation}
\label{eq:def-W}
\mathcal{W}[f;y](x,t) := (\mathcal{G}_0(\cdot - y,\cdot) *_{t} f)(x,t)
= \frac{1}{4\pi |x-y|} \, f\!\left(t - c_0^{-1}|x-y|\right),
\end{equation}
where $f$ is extended by zero to negative times. This kernel is normalized for $c_0^{-2}\partial_t^2-\Delta$; the incident-source kernel in \eqref{eq:incident-traces-transducers} contains the additional factor $\rho_c$ because the physical operator is $\rho_c^{-1}(c_0^{-2}\partial_t^2-\Delta)$.
The effective scattered pressure is therefore
\[
 p^{\mathrm{eff},\veps}(x,t)
 :=
 \sum_{i=1}^M \mathcal{W}[q_i^\veps;\,z_i](x,t).
\]
On the observation domain $\Omega$ containing $D$, the physical and effective fields satisfy
\[
 u^{\mathrm{sc}}(x,t)=p^{\mathrm{eff},\veps}(x,t)+\operatorname{err}_{\veps}(x,t),
 \qquad (x,t)\in (\Omega\setminus\overline D)\times({0},T),
\]
where the error is estimated outside the bubbles. The effective field satisfies the distributional point-source equation
\begin{equation}
\label{eq:effective-wave-R3}
 \frac{1}{c_0^2}\partial_t^2 p^{\mathrm{eff},\veps} - \Delta p^{\mathrm{eff},\veps}
 = \sum_{i=1}^M q_i^\veps(t)\,\delta_{z_i}
 \quad\text{in }\mathscr{D}'(\R^3\times({0},T)).
\end{equation}

\subsection{Effective Cluster Sources and Exterior Observations}
\label{subsec:cluster-reduction-localization}
\noindent
The restriction of the microscopic effective field to $\Omega$ satisfies
\begin{equation}
\label{eq:effective-wave-Omega}
 \frac{1}{c_0^2}\partial_t^2 p^{\mathrm{eff},\veps}_{M} - \Delta p^{\mathrm{eff},\veps}_{M}
 = \sum_{i=1}^M q_i^\veps(t)\,\delta_{z_i}
 \quad\text{in }\Omega\times({0},T),
\end{equation}
and is projected with its actual boundary trace as in \eqref{eq:ws-modal-identity}. Having determined $q_i^\veps$ from the microscopic system, we group their fields around the cluster centers. Let
\[
 r_\veps:=\max_{1\le\alpha\le N}\max_{i\in\mathcal I_\alpha}|z_i-y_\alpha|
\]
measure the cluster size. On observation sets separated from the bubbles, expansion of the retarded field in the source variable gives, for $i\in\mathcal I_\alpha$,
\[
 \mathcal W[q_i^\veps;z_i]
 =\mathcal W[q_i^\veps;y_\alpha]+\mathcal O(r_\veps),
\]
with the error measured in the same away-from-the-bubbles norms as in \eqref{eq:MS-expansion-again}, provided the amplitudes are controlled in the corresponding time norms. Summing inside each cluster yields
\begin{equation}
\label{eq:cluster-retarded-field}
 p^{\mathrm{eff},\veps}_{M}(x,t)
 =
 \sum_{\alpha=1}^N \mathcal W[Q_\alpha^\veps;y_\alpha](x,t)
 +\mathcal R_{\rm cl}^\veps(x,t),
 \qquad
 Q_\alpha^\veps(t):=(B_{\rm out}\mathbf q_{\rm mic}^\veps(t))_\alpha
 =\sum_{i\in\mathcal I_\alpha}q_i^\veps(t).
\end{equation}


\medskip
\noindent
The cluster-center field consequently satisfies
\begin{equation}
\label{eq:cluster-effective-wave-Omega}
 \frac{1}{c_0^2}\partial_t^2 p^{\mathrm{eff},\veps}_{\rm cl} - \Delta p^{\mathrm{eff},\veps}_{\rm cl}
 = \sum_{\alpha=1}^N Q_\alpha^\veps(t)\delta_{y_\alpha}
 \quad\text{in }\Omega\times({0},T).
\end{equation}
Its projected equation includes the boundary operator $\BM$. Below, $q_\alpha^\veps$ may also denote the cluster strength $Q_\alpha^\veps$, whereas $q_i^\veps$ refers to a microscopic bubble amplitude. With $p^{\mathrm{eff},\veps}=p^{\mathrm{eff},\veps}_{\rm cl}$, we have
\begin{equation}
\label{eq:phys-eff-decomposition-Omega}
 p^\veps=p^{\mathrm{eff},\veps}+\operatorname{err}_{\veps}
 \quad\text{on }(\Omega\setminus\overline D)\times({0},T),
\end{equation}

\paragraph{Projection of the exterior pressure.}
We estimate this difference using only values outside the bubbles. For $v$ defined on $\Omega\setminus\overline D$, denote its zero extension to $\Omega$ by $E_\varepsilon v$ and set
\begin{equation}\label{eq:exterior-observation-projection}
 \OMext v:=P_{I_M}E_\varepsilon v
 =\sum_{k=1}^{N_M}\left(\int_{\Omega\setminus\overline D}
                 v(x)\psi_k(x)\,dx\right)\psi_k.
\end{equation}
This map takes exterior observations into $H_M$; it is distinct from an orthogonal projector on $L^2(\Omega\setminus\overline D)$. Finite dimensionality gives
\[
 \|\OMext R^\varepsilon(t)\|_{H^1(\Omega)}
 \le (1+\lambda_{\max,M})^{1/2}
             \|R^\varepsilon(t)\|_{L^2(\Omega\setminus\overline D)}
\]
without taking a time derivative of the remainder. Besides the Foldy--Lax error $R_{\rm FL}^\varepsilon=R^\varepsilon$ in \eqref{eq:MS-expansion-again}, the replacement of bubble centers by cluster centers contributes
\[
 R_{\rm ctr}^\varepsilon
 =\sum_\alpha\sum_{i\in\mathcal I_\alpha}
   \bigl(\mathcal W[q_i^\varepsilon;z_i]
                  -\mathcal W[q_i^\varepsilon;y_\alpha]\bigr).
\]
For a fixed smooth mode, differentiating the retarded kernel with
respect to its source point gives the integrable majorant
\[
 C\left(|x-z|^{-2}\|q\|_\infty
       +|x-z|^{-1}\|\dot q\|_\infty\right).
\]
Its integral over the bounded three-dimensional domain is uniformly
bounded for the source points considered here. The mean-value formula
along the segment from $z_i$ to $y_\alpha$ therefore yields
\begin{equation}\label{eq:projected-center-remainder}
 \sup_{t\le T}\|\OMext R_{\rm ctr}^\varepsilon(t)\|_{H^1(\Omega)}
 \le C_{T,M}\varepsilon^p
                   \sum_i\|q_i^\varepsilon\|_{W^{1,\infty}(0,T)}.
\end{equation}
The derivative in this estimate acts on the source amplitude, not on $R_{\rm FL}^\varepsilon$.

\medskip
\noindent
The zero extension also omits the effective field inside $D$. Since $M$ is fixed, $|D|=O(\varepsilon^3)$, and $\displaystyle\sup_y\int_D|x-y|^{-1}dx\le C|D|^{2/3}$, this contribution satisfies
\begin{equation}\label{eq:observation-hole-remainder}
 \sup_{t\le T}\|P_{I_M}(\mathbf1_D p^{\mathrm{eff},\varepsilon})(t)\|_{H^1(\Omega)}
 \le C_M\varepsilon^2\sum_\alpha\|Q_\alpha^\varepsilon\|_\infty.
\end{equation}
The integral bound follows by splitting at radius $|D|^{1/3}$. Adding the Foldy--Lax error, the center-replacement error, and the omitted-volume term gives
\begin{equation}\label{eq:f-eps-H-1}
 \begin{split}
 &\sup_{t\le T}
 \|\OMext p^\varepsilon(t)-P_{I_M}p^{\mathrm{eff},\varepsilon}(t)\|_{H^1(\Omega)}\\
 &\quad\le C_{T,M}\left[
       \varepsilon^\mu\|\boldsymbol\lambda\|_{\Xnine}
       +(\varepsilon^p+\varepsilon^2)
                    \sum_i\|q_i^\varepsilon\|_{W^{1,\infty}}\right]\\
 &\quad\le C_{T,M}(\varepsilon^\mu+\varepsilon^p)
                    \|\boldsymbol\lambda\|_{\Xnine}.
 \end{split}
\end{equation}
The last line uses Lemma~\ref{lem:causal-transfer-fixed-time} and the
bounded transducer trace map. For the controls constructed below,
\eqref{eq:surjectivity-control-cost} supplies the corresponding bound by
$\|\mathbf q^{\mathrm{ideal}}\|_{\mathcal V_\varepsilon^m}$.


\subsection{Laplace-domain Transfer Function of the Bubble Subsystem}
\noindent
It remains to determine how the exterior signals generate the source strengths. We take the Laplace transform of the bubble system, using
\[
q_i^\varepsilon(t) := -C_i(\varepsilon)\,Y_i(t),\qquad q^\varepsilon(t) := (q_1^\varepsilon(t),\dots,q_M^\varepsilon(t))^\top = \mathcal{A}_b\,\bm Y(t),
\]
where $\mathcal{A}_b:=-\operatorname{diag}(C_i(\varepsilon))_{i=1}^M$.

\medskip
\noindent
The incident traces form the vector $\mathbf u^{\mathrm{in}}(t)\in\mathbb R^M$, which depends linearly on $\lambda=(\lambda_1,\dots,\lambda_{M_{\mathrm{tr}}})$. The zero initial conditions and the identity
\[
\mathcal{L}[\ddot Y_j(t-\tau_{ij})](s) = e^{-s\tau_{ij}} s^2 \widehat Y_j(s),
\]
give the transformed system
\[
(I+\mathcal{T}(s))\widehat{\bm Y}(s) = \mathcal{V}(s), \qquad \Re s > \sigma_0,
\]
where we define the diagonal unperturbed operator and the interaction matrix as
\[
\mathcal{D}(s) = \operatorname{diag}(\omega_{M,i}^{-2}s^2+1)_{i=1}^M, \qquad \mathcal{Q}(s) = \bigl(q_{ij}e^{-s\tau_{ij}}\bigr)_{i,j=1}^M,
\]
and we set
\[
\mathcal{T}(s) := \mathcal{D}(s)^{-1}s^2 \mathcal{Q}(s), \qquad \mathcal{V}(s) := \mathcal{D}(s)^{-1}s^2\widehat{\mathbf u}^{\mathrm{in}}(s).
\]
The transformed forcing is $\widehat{\boldsymbol{\mathcal F}}=s^2\widehat{\mathbf u}^{\mathrm{in}}$. Thus $B^{\mathrm{in}}=I$ when the input is the undifferentiated incident trace. The function $\mathcal V(s)$ is analytic for $\Re s>\sigma_0$, and the parameters satisfy
\[
\omega_{M,i}>0,\qquad \tau_{ij} = c_0^{-1}|z_i - z_j| \ge 0,\qquad q_{ij} = \frac{C_j(\varepsilon)}{4\pi |z_i - z_j|} \in \mathbb R,
\]
with the convention $q_{ii}=0$. By defining the dynamic interaction kernel
\[
\mathcal{K}_{\mathrm{dyn}}(s) := I + \mathcal{T}(s),
\]
we obtain
\[
\mathcal{K}_{\mathrm{dyn}}(s)\widehat{\bm Y}(s) = \mathcal{D}(s)^{-1}s^2\widehat{\mathbf u}^{\mathrm{in}}(s).
\]
For $\Re s > 0$ such that $\mathcal{K}_{\mathrm{dyn}}(s)$ is invertible, we have
\[
\widehat{\bm Y}(s) = \mathcal{K}_{\mathrm{dyn}}(s)^{-1}\mathcal{D}(s)^{-1}s^2\widehat{\mathbf u}^{\mathrm{in}}(s).
\]
More generally, an incident-input parametrization $B^{\textit{in}}(s)$ gives
\[
\widehat{\boldsymbol{\mathcal{F}}}(s) = s^2 B^{\textit{in}}(s)\,\widehat u^{\textit{in}}(s).
\]
Defining the composite input matrix $B(s) := \mathcal{D}(s)^{-1}s^2 B^{\textit{in}}(s)$, we obtain
\[
\widehat{\mathbf{Y}}(s) = \mathcal{K}_{\mathrm{dyn}}(s)^{-1} B(s)\,\widehat u^{\textit{in}}(s).
\]
Since $\widehat{q}^\varepsilon(s) = \mathcal{A}_b \widehat{\bm Y}(s)$, it follows that
\[
H_b(s) := \mathcal{A}_b \mathcal{K}_{\mathrm{dyn}}(s)^{-1} B(s),
\]
and hence
\[
\widehat q^\varepsilon(s) = H_b(s)\,\widehat u^{\textit{in}}(s), \qquad \Re s > 0.
\]
The characteristic determinant governing the bubble resonances is
\begin{equation}
    \mathcal{Z}(s) := \det\left( \mathcal{D}(s) + s^2 \mathcal{Q}(s) \right) = 0.
\end{equation}
For a single bubble, $M=1$ and $\mathcal Q(s)=0$, so the characteristic equation reduces to
\begin{equation}
    \omega_M^{-2} s^2 + 1 = 0.
\end{equation}
Its roots $s=\pm i\omega_M$ are the bare Minnaert resonances. For $M>1$, the delayed interaction matrix $\mathcal Q(s)$ perturbs them into resonance packets. We analyze these packets next.
\section{Local Resonances and Collective Block Reduction}\label{proofprop}

\noindent
This section derives the spectral and transfer-operator results required to realize the ideal sources by exterior transducers. We establish the conclusions of Proposition~\ref{proposition} and derive the exact collective normal form used in Section~\ref{sec:resonant-gain}. First, we prove the meromorphic structure of Proposition~\ref{proposition}(i). The isolated-group calculation then identifies the principal Perron branch, its $\mathcal{O}(h)$ conservative shift, its $\mathcal{O}(\varepsilon)$ radiative correction, and the $\mathcal{O}(h)$ gap to the complementary microscopic modes; this formalizes the local tuning and provides the scalar reference case. For clusters sharing a common carrier, we prove Proposition~\ref{prop:collective-block-reduction} and Proposition~\ref{proposition}(ii)--(iii) without diagonalizing the collective roots. This yields the block generator $S_\ell^\varepsilon=i\widehat\omega_\ell I+\varepsilon A_\ell^\varepsilon$ together with the input factor $R_\ell^\varepsilon=\varepsilon\widetilde R_\ell^\varepsilon$ and the uniformly bounded input/output coordinates. These components govern the causal inversion in Section~\ref{sec:resonant-gain}.

\subsection{Meromorphic Structure of the Transfer Operator}
\begin{proof}

We first prove Proposition~\ref{proposition}(1). Substituting $B(s)$ into $H_b(s)$ gives
\[
    H_b(s, \varepsilon) = \mathcal{A}_b \mathcal{K}_{\mathrm{dyn}}(s, \varepsilon)^{-1} \mathcal{D}(s)^{-1} s^2 B^{\mathrm{in}}(s).
\]
Set $P(s, \varepsilon) := \mathcal{D}(s)\mathcal{K}_{\mathrm{dyn}}(s, \varepsilon) = \mathcal{D}(s) + s^2 \mathcal{Q}(s, \varepsilon)$. Then
\begin{equation} \label{eq:Hb_pencil_compact}
    H_b(s, \varepsilon) = \mathcal{A}_b P(s, \varepsilon)^{-1} s^2 B^{\mathrm{in}}(s).
\end{equation}
For each fixed $\varepsilon>0$, the pencil $P(s,\varepsilon)$ is entire in $s$: $\mathcal D$ is polynomial, while $\mathcal{Q}_{ij}(s, \varepsilon) = q_{ij}(\varepsilon)e^{-s\tau_{ij}(\varepsilon)}$ is entire.

\medskip
\noindent
Let $f(s,\varepsilon)=\det P(s,\varepsilon)$. To show that $f$ is not identically zero, take $s=x>0$ and write
\[
    P(x, \varepsilon) = \mathcal{D}(x) \left( I + \mathcal{D}(x)^{-1} x^2 \mathcal{Q}(x, \varepsilon) \right).
\]
Since $\mathcal{D}(x)=\operatorname{diag}(\omega_{M,i}^{-2}x^2+1)_{i=1}^M$, we have $\mathcal D(x)^{-1}x^2\to\operatorname{diag}(\omega_{M,i}^2)_{i=1}^M$ as $x\to\infty$. The separation condition gives $d_{ij}\ge\min(c_1\varepsilon^p,D_{\textit{min}})>0$ for $i\ne j$. Hence all delays $\tau_{ij}(\varepsilon)=d_{ij}/v$ are positive for fixed $v$ and $\varepsilon$. 

\medskip
\noindent
It follows that $\|\mathcal D(x)^{-1}x^2\mathcal Q(x,\varepsilon)\|\to0$ exponentially. Therefore
\[
    f(x, \varepsilon) = \det \mathcal{D}(x) \cdot \det\left(I + \mathcal{D}(x)^{-1} x^2 \mathcal{Q}(x, \varepsilon)\right) \sim \prod_{i=1}^M (\omega_{M,i}^{-2} x^2 + 1) \ne 0, \quad \text{as } x \to \infty.
\]
The identity theorem \cite[Ch.~4]{complex1} now implies that $Z(\varepsilon)=\{s\in\mathbb C:f(s,\varepsilon)=0\}$ is discrete. Outside this set, Cramer's rule gives
\[
    P(s, \varepsilon)^{-1} = \frac{\operatorname{adj}(P(s, \varepsilon))}{f(s, \varepsilon)}.
\]
The adjugate is entire, so $P(s,\varepsilon)^{-1}$ is meromorphic; see \cite[Ch.~3]{complex2}. Formula~\eqref{eq:Hb_pencil_compact} gives the same property for $H_b(s,\varepsilon)$, since its remaining factors are analytic. Its poles are finite in number on every compact subset of $\mathbb C$.
\par\noindent\hfill$\blacksquare$
\end{proof}

\subsection{Discrete Poles and Asymptotic Cluster Separation}\label{Poles-asymptotics}
\begin{proof}
We first consider an isolated frequency group containing one cluster. This calculation identifies its principal pole and damping. When several clusters share a carrier, their principal modes must instead be treated together, as in Section~\ref{subsec:collective-block-reduction}. In particular, the scalar damping calculation below is not a spectral assumption on the collective system.
Let $\mathbb{C}_{\ge 0} = \{s \in \mathbb{C} : \Re(s) \ge 0\}$. The poles of the transfer matrix $H_b(s)$ in the complex plane coincide with the roots of the characteristic equation $F(s, \varepsilon) = 0$, where
\[
    F(s, \varepsilon) := \det\left( \mathcal{D}(s) + s^2 \mathcal{Q}(s, \varepsilon) \right).
\]
Here $\mathcal{D}(s)=\operatorname{diag}(\omega_{M,i}^{-2}s^2+1)_{i=1}^M$ is the uncoupled pencil, and $\mathcal{Q}_{ij}(s,\varepsilon)=\varepsilon\frac{\overline{\mathcal C}_je^{-s\tau_{ij}}}{4\pi d_{ij}}$ for $i\ne j$ describes the retarded interactions.

\begin{enumerate}
    \item \textit{Claim 1: Localization near $i\omega_{M,\alpha}$.}

\noindent
With the material coefficients fixed at their values for the given $\varepsilon$, the uncoupled determinant is
\[
    F_0(s;\varepsilon) := \det(\mathcal{D}(s,\varepsilon)) = \prod_{\gamma=1}^N \left(\omega_{M,\gamma}^{-2} s^2 + 1\right)^{M_\gamma}.
\]
Its upper-half-plane roots are the bare frequencies. The active root $i\omega_{M,\alpha}$ has multiplicity $M_\alpha$ and is uniformly separated from the other groups. Choose $\rho_\alpha>0$ so that $U_\alpha:=\{s\in\mathbb C:|s-i\omega_{M,\alpha}|<\rho_\alpha\}$ contains no other root in its closure.

Set $m_\alpha:=\min_{s\in\partial U_\alpha}|F_0(s)|>0$. The two separation scales give the following bounds for the entries of $\mathcal Q$. 
\begin{enumerate}
    \item \textbf{Intra-cluster terms} ($i,j \in \mathcal{I}_\gamma$): $d_{ij} \ge c_1 \varepsilon^p \implies |\mathcal{Q}_{ij}| = \mathcal{O}(\varepsilon^{1-p})$.
    \item \textbf{Inter-cluster terms} ($i \in \mathcal{I}_\gamma, j \in \mathcal{I}_\beta$): $d_{ij} \ge D_{\min} \implies |\mathcal{Q}_{ij}| = \mathcal{O}(\varepsilon)$.
\end{enumerate}
Thus $\|\mathcal Q(s,\varepsilon)\|=\mathcal O(\varepsilon^{1-p})$ uniformly on $\partial U_\alpha$. Since $p<1$, 
\[
    \lim_{\varepsilon \to 0} \sup_{s \in \partial U_\alpha} \|s^2 \mathcal{Q}(s, \varepsilon)\| = 0.
\]
By continuity of the determinant, for all sufficiently small $\varepsilon$,
\[
    \max_{s \in \partial U_\alpha} |F(s, \varepsilon) - F_0(s)| < m_\alpha \le \min_{s \in \partial U_\alpha} |F_0(s)|.
\]
Rouché's theorem gives exactly $M_\alpha$ roots of $F(s,\varepsilon)$ in $U_\alpha$, counted with multiplicity.

\item \textit{Claim 2: The distinguished cluster pole.}

\medskip
\noindent
For the isolated active carrier, the fixed detuning from the other groups permits a scalar reduction of the pencil $P:\mathbb C\times\mathbb R_{>0}\to\mathbb C^{M\times M}$,
$$P(s, \varepsilon) := \mathcal{D}(s) + s^2 \mathcal{Q}(s, \varepsilon),$$ 
with $\mathcal{D}(s) := \operatorname{diag}(\omega_{M,i}^{-2} s^2 + 1)_{i=1}^M$ and the interaction matrix elements evaluated near the local resonance $s = i\omega_{M,\alpha}$ given by $\mathcal{Q}_{ij}(i\omega_{M,\alpha}, \varepsilon) := \mathcal{C}_j G(x_i, x_j, i\omega_{M,\alpha})$ for all $i \neq j$. The leading block-diagonal part has a distinguished simple pole associated with the principal mode of the active cluster:
$$s_\alpha^\varepsilon = -\eta_{\alpha} + i\omega_{\alpha}(\varepsilon), \quad \eta_{\alpha}, \omega_{\alpha}(\varepsilon) > 0,$$
where $\omega_\alpha(\varepsilon)$ is the frequency after the intra-cluster interaction shift.

\begin{enumerate}
\item \textit{Step 1: Two-scale expansion.}

\medskip
\noindent
For $N$ clusters containing $M=\sum_{\gamma=1}^N M_\gamma$ bubbles, the distances satisfy
\begin{enumerate}
    \item \textbf{Intra-cluster} ($i, j \in \mathcal{I}_\gamma$): $d_{ij} = \widetilde{d}_{ij}\varepsilon^p$ for $p \in (0,2/3)$.
    \item \textbf{Inter-cluster} ($i \in \mathcal{I}_\gamma, j \in \mathcal{I}_\beta$ with $\gamma \neq \beta$): $d_{ij} = D_{ij} \sim \mathcal{O}(1)$.
\end{enumerate}
Expanding the free-space Green's function at $i\omega_{M,\alpha}$ gives
$$ \mathbf{Q}(i\omega_{M,\alpha}, \varepsilon) = \varepsilon^{1-p} \widetilde{\mathbf{C}}^{(0)} - i\varepsilon \mathbf{\Gamma}^{(1)} + \mathcal{O}(\varepsilon^{1+p}), $$
where a permutation matrix $\bm{\mathrm P}\in\{0,1\}^{M\times M}$ puts the leading matrix in block-diagonal form:
\begin{align}
    \bm{\mathrm{P}}^\top \widetilde{\mathbf{C}}^{(0)} \bm{\mathrm{P}} = \bigoplus_{\gamma=1}^N \widetilde{C}_\gamma^{(0)} = \operatorname{diag}(\widetilde{C}_1^{(0)}, \widetilde{C}_2^{(0)}, \dots, \widetilde{C}_N^{(0)}).
\end{align}
Here, each $\widetilde{C}_\gamma^{(0)} \in \mathbb{R}^{M_\gamma\times M_\gamma}$ is an irreducible Metzler sub-block.

\medskip
\noindent
The perturbation matrix $\mathbf{\Gamma}^{(1)}$ contains both the intra-cluster radiation matrices (in the diagonal blocks) and the macroscopic inter-cluster wave interactions (in the off-diagonal blocks), evaluated at $i\omega_{M,\alpha}$:
$$ [\mathbf{\Gamma}^{(1)}]_{\gamma\gamma} = C_\gamma^{(1)}, \quad [\mathbf{\Gamma}^{(1)}]_{\gamma\beta} = i G_{\gamma\beta}(i\omega_{M,\alpha}) \quad (\text{for } \gamma \neq \beta), $$
with $G_{\gamma\beta, ij} = \overline{\mathcal{C}}_j \frac{e^{-ik_{M,\alpha} D_{ij}}}{4\pi D_{ij}}$ where the wavenumber is $k_{M,\alpha} = \omega_{M,\alpha}/{c_0}$.

\item \textit{Step 2: The Perron mode and detuning.}

\medskip
\noindent
The spectrum of $\widetilde{\mathbf C}^{(0)}$ is the union of the cluster spectra. Each block $\widetilde C_\alpha^{(0)}$ is non-negative and irreducible. Its principal eigenvalue $\widetilde\mu_{\alpha,1}^{(0)}$ is therefore simple by the Perron--Frobenius theorem \cite[Ch.~8, Th.~8.4.4]{la1}. 

At $s^{(0)}=i\omega_{M,\alpha}$, Assumption~\ref{assum:dual-scale-geometry}(i) gives the off-resonant blocks $\mathbf P_{\beta\beta}^{(0)}=(1-\omega_{M,\alpha}^2/\omega_{M,\beta}^2)\mathbf I_{M_\beta}$. They are uniformly invertible for the isolated group considered here. The corresponding unperturbed principal vectors are
$$\mathbf{v}_\alpha^{(0)} = \begin{bmatrix} \mathbf{0} \\ v_{\alpha, 1}^{(0)} \\ \mathbf{0} \end{bmatrix}, \quad \mathbf{w}_\alpha^{(0)\top} = \begin{bmatrix} \mathbf{0}^\top & w_{\alpha, 1}^{(0)\top} & \mathbf{0}^\top \end{bmatrix} \in \mathbb{R}_{\ge 0}^M,$$
where $v_{\alpha,1}^{(0)}$ and $w_{\alpha,1}^{(0)}$ are positive Perron vectors of the active block. In particular, $\mathbf w_\alpha^{(0)\top}\mathbf v_\alpha^{(0)}>0$.

\item \textit{Step 3: Perturbation of the principal pole.}

\medskip
\noindent
Analytic perturbation in the scaled spectral variable applies to the simple Perron root; see \cite[Chap.~II, Sec.~2.3]{kato1995perturbation}. The material family itself need not depend analytically on $\varepsilon$. Write the perturbed pole and vector as $s_\alpha^\varepsilon=i\omega_{M,\alpha}+\delta(\varepsilon)$ and $\mathbf v_\alpha(\varepsilon)=\mathbf v_\alpha^{(0)}+\tilde{\mathbf v}(\varepsilon)$.

\medskip
\noindent
At $s=s_\alpha^\varepsilon$, the pencil $\mathbf P(s,\varepsilon)=\mathcal D(s)+s^2\mathbf Q(s,\varepsilon)$ satisfies
\begin{equation}
    \mathbf{P}(s_\alpha^\varepsilon, \varepsilon) = \mathbf{D}_\alpha^{(0)} + (2i\omega_{M,\alpha}\delta + \delta^2)\mathbf{\Omega}_M^{-2} + \left(-\omega_{M,\alpha}^2 + 2i\omega_{M,\alpha}\delta + \delta^2\right) \mathbf{Q}(s_\alpha^\varepsilon, \varepsilon),
\end{equation}
where $\mathbf{\Omega}_M^{-2} := \operatorname{diag}(\omega_{M,1}^{-2}\mathbf{I}_{M_1}, \dots, \omega_{M,N}^{-2}\mathbf{I}_{M_N})$ and the detuning matrix is \\
$\mathbf{D}_\alpha^{(0)} := \operatorname{diag}\left((1 - \omega_{M,\alpha}^2/\omega_{M,\beta}^2)\mathbf{I}_{M_\beta}\right)_{\beta=1}^N$.

For $i,j\in\mathcal I_\alpha$, use $d_{ij}=\varepsilon^p\tilde d_{ij}$ and expand the Green's function to obtain
\begin{equation}
    Q_{ij}(s_\alpha^\varepsilon, \varepsilon) = \varepsilon^{1-p} \frac{\tilde{C}_j}{4\pi \tilde{d}_{ij}} - i\varepsilon \frac{\tilde{C}_j \omega_{M,\alpha}}{4\pi c_0} - \varepsilon \frac{\tilde{C}_j}{4\pi c_0} \delta + \mathcal{O}(\varepsilon^{1+p}).
\end{equation}
For inter-cluster bubbles ($i \in \mathcal{I}_\alpha, j \in \mathcal{I}_\beta, \alpha \neq \beta$), the macroscopic separation $D_{ij} = \mathcal{O}(1)$ gives radiation-order scaling
\begin{equation}
    Q_{ij}(s_\alpha^\varepsilon, \varepsilon) = \varepsilon \frac{\tilde{C}_j}{4\pi D_{ij}} e^{-i\omega_{M,\alpha} D_{ij}/c_0} \left[ 1 - \frac{\delta D_{ij}}{c_0} + \mathcal{O}(\delta^2) \right] = \mathcal{O}(\varepsilon) + \mathcal{O}(\delta \varepsilon).
\end{equation}
Combining the two expansions, we obtain
\begin{equation}
    \mathbf{Q}(s_\alpha^\varepsilon, \varepsilon) = \varepsilon^{1-p}\widetilde{\mathbf{C}}^{(0)} - i\varepsilon \mathbf{\Gamma}^{(1)} + \mathcal{R}_Q,
\end{equation}
where $\widetilde{\mathbf{C}}^{(0)} := \operatorname{diag}(\widetilde{\mathbf{C}}_1^{(0)}, \dots, \widetilde{\mathbf{C}}_N^{(0)})$ and $[\mathbf{\Gamma}_{\alpha\beta}^{(1)}]_{ij}$ encodes both intra- and inter-cluster leading-order radiation. The remainder is bounded by $\|\mathcal{R}_Q\| = \mathcal{O}(\delta \varepsilon) + \mathcal{O}(\varepsilon^{1+p})$.

Multiplication by $(s_\alpha^\varepsilon)^2$ gives
\begin{equation}
    \left(-\omega_{M,\alpha}^2 + 2i\omega_{M,\alpha}\delta + \delta^2\right) \mathbf{Q} = -\omega_{M,\alpha}^2 \varepsilon^{1-p}\widetilde{\mathbf{C}}^{(0)} + i\omega_{M,\alpha}^2 \varepsilon \mathbf{\Gamma}^{(1)} + 2i\omega_{M,\alpha}\delta \varepsilon^{1-p}\widetilde{\mathbf{C}}^{(0)} + \mathcal{R}_{\mathrm{cross}}.
\end{equation}
The remainder satisfies $\|\mathcal{R}_{\mathrm{cross}}\| = \mathcal{O}(\varepsilon^{1+p}) + \mathcal{O}(\delta \varepsilon) + \mathcal{O}(\delta^2 \varepsilon^{1-p})$.

Multiply $\mathbf P(s_\alpha^\varepsilon,\varepsilon)(\mathbf v_\alpha^{(0)}+\tilde{\mathbf v})=0$ by $\mathbf w_\alpha^{(0)\top}$, which is supported on $\mathcal I_\alpha$. The resulting terms are as follows.
\begin{enumerate}
    \item $\mathbf{w}_\alpha^{(0)\top} \mathbf{D}_\alpha^{(0)} \mathbf{v}_\alpha^{(0)} = 0$.
    \item $\mathbf{w}_\alpha^{(0)\top}(2i\omega_{M,\alpha}\delta\,\mathbf{\Omega}_M^{-2} + \delta^2\,\mathbf{\Omega}_M^{-2})\mathbf{v}_\alpha^{(0)} = (2i\omega_{M,\alpha}^{-1}\delta + \omega_{M,\alpha}^{-2}\delta^2)\langle w_{\alpha,1}^{(0)}, v_{\alpha,1}^{(0)}\rangle$.
    \item Imposing Kato's orthogonality condition (see \cite[Chap. II, Sec. 3.3]{kato1995perturbation}), $\mathbf{w}_\alpha^{(0)\top} \tilde{\mathbf{v}} = 0$, eliminates the leading-order cross terms
    \begin{equation*}
        \mathbf{w}_\alpha^{(0)\top} \left[ 2i\omega_{M,\alpha}^{-1} \delta \mathbf{I} - \omega_{M,\alpha}^2 \varepsilon^{1-p} \widetilde{\mathbf{C}}^{(0)} \right] \tilde{\mathbf{v}} = 0.
    \end{equation*}
\end{enumerate}
Dividing by $\langle w_{\alpha,1}^{(0)}, v_{\alpha,1}^{(0)}\rangle > 0$ yields the scalar characteristic polynomial
\begin{equation}
    2i\omega_{M,\alpha}^{-1} \delta - \omega_{M,\alpha}^2 \widetilde{\mu}_{\alpha, 1}^{(0)} \varepsilon^{1-p} + i\omega_{M,\alpha}^2 \mathbf{M}_{11}^{(\alpha)} \varepsilon + \mathcal{O}(\delta^2) + \widetilde{\mathcal{R}}_{\mathrm{cross}} = 0,
\end{equation}
where $\mathbf{M}_{11}^{(\alpha)} := \frac{\mathbf{w}_\alpha^{(0)\top} \mathbf{\Gamma}^{(1)} \mathbf{v}_\alpha^{(0)}}{\mathbf{w}_\alpha^{(0)\top} \mathbf{v}_\alpha^{(0)}} \in \mathbb{R}_{>0}$ is the local radiation projection, and $\widetilde{\mathcal{R}}_{\mathrm{cross}} = \mathcal{O}(\varepsilon^{1+p}) + \mathcal{O}(\delta \varepsilon) + \mathcal{O}(\delta^2 \varepsilon^{1-p}) + \mathcal{O}(\varepsilon\|\tilde{\mathbf{v}}\|)$.

To determine the size of $\delta$, put $s=i\omega_{M,\alpha}+hz$, where $h=\varepsilon^{1-p}$. On the active block the leading scaled pencil is
$(2i/\omega_{M,\alpha})zI-\omega_{M,\alpha}^2
\widetilde C_\alpha^{(0)}$. Its Perron root is simple, so analytic perturbation, equivalently the implicit-function theorem at this root, gives $\delta=O(h)$.
The intra-cluster eigenvector correction caused by the $O(\varepsilon)$
retarded perturbation is $O(\varepsilon/h)=O(\varepsilon^p)$.

For frozen admissible material coefficients the Perron
branch is exactly
\[
 s_{c,\alpha}
 =i\omega_{M,\alpha}
  (1+\omega_{M,\alpha}^2\widetilde\mu_{\alpha,1}^{(0)}h)^{-1/2}.
\]
Its expansion, together with the radiation perturbation, gives
\begin{equation}\label{eq:corrected-conservative-shift}
 \begin{split}
 s_\alpha^\varepsilon
 &=i\omega_{M,\alpha}
 -\frac{i\omega_{M,\alpha}^3}{2}\widetilde\mu_{\alpha,1}^{(0)}h
 +\frac{3i\omega_{M,\alpha}^5}{8}
       (\widetilde\mu_{\alpha,1}^{(0)})^2h^2\\
 &\quad-\frac{\omega_{M,\alpha}^3}{2}\mathbf M_{11}^{(\alpha)}\varepsilon
 +O\bigl(h^3+\varepsilon h+\varepsilon^{1+p}\bigr).
 \end{split}
\end{equation}
The coefficient $3i\omega_{M,\alpha}^5\mu^2/8$ includes the cross term
$2i\omega_{M,\alpha}\delta\mu h$ in the second-order balance.
When $p=1/2$, the $h^2$ and $\varepsilon$ terms are retained together;
the former is conservative and the latter dissipative. Because
$0<p<2/3$, the displayed remainder is $o(\varepsilon)$.
Here the superscript $(0)$ denotes the conservative rescaled kernel. Its weights and Perron data are evaluated at the admissible material parameters for each $\varepsilon$, including the value of $\overline\kappa_b(\varepsilon)$.

The isolated-mode damping is consequently
$\eta_\alpha=\varepsilon\omega_{M,\alpha}^3
\mathbf M_{11}^{(\alpha)}/2+o(\varepsilon)$.
The coefficient is positive by Step~5. For the isolated group, the other clusters contribute through off-resonant elimination. At a common carrier, their coupling has the same order as the radiation term and must be included in the principal block.

\item \textit{Step 4: Simplicity and the rank-one residue.}

\medskip
\noindent
The simplicity established in Step~2 and the preceding perturbation argument imply that $s_\alpha^\varepsilon$ is a simple root of $\det\mathbf P(s,\varepsilon)=0$. The kernel of $\mathbf P(s_\alpha^\varepsilon,\varepsilon)$ is therefore one-dimensional and spanned by $\mathbf v_\alpha(\varepsilon)$. 

Taylor expansion of the pencil gives
\begin{equation}
    \mathbf{P}(s, \varepsilon) = \mathbf{P}(s_\alpha^\varepsilon, \varepsilon) + (s - s_\alpha^\varepsilon) \left. \frac{\partial \mathbf{P}}{\partial s} \right|_{s_\alpha^\varepsilon} + \mathcal{O}((s - s_\alpha^\varepsilon)^2).
\end{equation}

The resolvent has the rank-one residue
\begin{equation}
    \mathcal{R}_{\alpha}(\varepsilon) = \lim_{s \to s_\alpha^\varepsilon} (s - s_\alpha^\varepsilon) \mathbf{P}(s, \varepsilon)^{-1} = \frac{\mathbf{v}_\alpha(\varepsilon) \mathbf{w}_\alpha(\varepsilon)^\top}{\mathbf{w}_\alpha(\varepsilon)^\top \left. \frac{\partial \mathbf{P}}{\partial s} \right|_{s_\alpha^\varepsilon} \mathbf{v}_\alpha(\varepsilon)}.
\end{equation}

Denote the denominator by $\bm J_\alpha(\varepsilon)$. Since $\mathbf P_{\alpha\alpha}^{(0)}(s)=(1+s^2/\omega_{M,\alpha}^2)\mathbf I$, its derivative is $(2s/\omega_{M,\alpha}^2)\mathbf I$. Using $s_\alpha^\varepsilon=i\omega_{M,\alpha}+\mathcal O(\varepsilon^{1-p})$ and $\mathbf v_\alpha(\varepsilon)=\mathbf v_\alpha^{(0)}+O(\varepsilon^p)$, we find
\begin{equation}
    \mathbf J_\alpha(\varepsilon)=\frac{2i}{\omega_{M,\alpha}}\langle w_{\alpha,1}^{(0)},v_{\alpha,1}^{(0)}\rangle+O(\varepsilon^{1-p}+\varepsilon^p).
\end{equation}

Both unperturbed vectors vanish outside $\mathcal I_\alpha$. Hence the leading residue is supported on the $\alpha\times\alpha$ block:
\begin{equation}
    \mathcal{R}_{\alpha}(\varepsilon) \approx \frac{\omega_{M,\alpha}}{2i \langle w_{\alpha, 1}^{(0)}, v_{\alpha, 1}^{(0)} \rangle} \begin{bmatrix} \mathbf{0} & \mathbf{0} & \mathbf{0} \\ \mathbf{0} & v_{\alpha, 1}^{(0)} w_{\alpha, 1}^{(0)\top} & \mathbf{0} \\ \mathbf{0} & \mathbf{0} & \mathbf{0} \end{bmatrix}.
\end{equation}

This localization holds to leading order for the isolated group; the perturbed vector has components outside the active cluster. For a common-frequency group, localization concerns the whole principal subspace rather than its individual columns. 

The half-power bandwidth is determined by $\Re s_\alpha^\varepsilon=-\eta_\alpha$. Step~3 gives
\begin{equation}
    \Gamma_\alpha = 2|\Re(s_\alpha^\varepsilon)| = 2\eta_{\alpha} = \omega_{M,\alpha}^3 \mathbf{M}_{11}^{(\alpha)} \varepsilon + \mathcal{O}(\varepsilon^{>1}).
\end{equation}

\item \textit{Step 5: Positive damping for the isolated cluster.}

\medskip
\noindent
By Step~3, the damping satisfies
\begin{equation}
    \eta_{\alpha} = \varepsilon \frac{\omega_{M,\alpha}^3}{2} \mathbf{M}_{11}^{(\alpha)} + \mathcal{O}(\varepsilon^{>1}), \quad \mathbf{M}_{11}^{(\alpha)} = \frac{\mathbf{w}_\alpha^{(0)\top} \mathbf{\Gamma}_{\alpha\alpha}^{(1)} \mathbf{v}_\alpha^{(0)}}{\mathbf{w}_\alpha^{(0)\top} \mathbf{v}_\alpha^{(0)}},
\end{equation}
where $\mathbf{\Gamma}_{\alpha\alpha}^{(1)}$ has zero diagonal ($\Gamma_{ii}^{(1)} = 0$) and constant off-diagonal entries $\Gamma_{ij}^{(1)} = \frac{\omega_{M,\alpha}}{4\pi c_0} \tilde{C}_j$.

Write $\widetilde{\mathbf C}_\alpha^{(0)}=\mathbf D_\alpha\mathbf C_\alpha$, where $\mathbf C_\alpha=\operatorname{diag}(\tilde C_j)$ and $\mathbf D_\alpha=\mathbf D_\alpha^\top$. We may choose $\mathbf w_\alpha^{(0)}=\mathbf C_\alpha\mathbf v_\alpha^{(0)}$. Thus $\mathbf w_\alpha^{(0)\top}\mathbf v_\alpha^{(0)}=\sum_{j\in\mathcal I_\alpha}\tilde C_j(v_{\alpha,j}^{(0)})^2>0$.

For the numerator, set $u_j:=\tilde C_jv_{\alpha,j}^{(0)}$. Then
\begin{equation}
\begin{aligned}
\mathbf{w}_\alpha^{(0)\top} \mathbf{\Gamma}_{\alpha\alpha}^{(1)} \mathbf{v}_\alpha^{(0)} &= \sum_{\substack{i, j \in \mathcal{I}_\alpha \\ i \neq j}} (\tilde{C}_i v_{\alpha, i}^{(0)}) \left( \frac{\omega_{M,\alpha}}{4\pi c_0} \tilde{C}_j \right) v_{\alpha, j}^{(0)} \\
&= \frac{\omega_{M,\alpha}}{4\pi c_0} \sum_{\substack{i, j \in \mathcal{I}_\alpha \\ i \neq j}} u_i u_j = \frac{\omega_{M,\alpha}}{4\pi c_0} \left[ \left( \sum_{j \in \mathcal{I}_\alpha} u_j \right)^2 - \sum_{j \in \mathcal{I}_\alpha} u_j^2 \right].
\end{aligned}
\end{equation}
Each $u_j$ is positive by the Perron--Frobenius theorem. Since $M_\alpha\ge2$, the off-diagonal sum is non-empty and $(\sum u_j)^2>\sum u_j^2$. 

Consequently $\mathbf M_{11}^{(\alpha)}>0$. After normalization of the Perron vectors, this lower bound is uniform over the compact admissible families of scaled distances and weights, with cluster sizes fixed. We obtain
\begin{equation}
    \eta_{\alpha} = \frac{\omega_{M,\alpha}^3}{2} \mathbf{M}_{11}^{(\alpha)} \varepsilon + \mathcal{O}(\varepsilon^{>1}) > 0
\end{equation}
for sufficiently small $\varepsilon>0$. Thus the isolated principal pole lies in the left half-plane. The collective analysis below does not require this sign property.

\item \textit{Step 6: Separation from the complementary spectrum.}

\medskip
\noindent
Let $\mathcal{P}(\varepsilon) \subset \mathbb{C}$ denote the discrete poles of the global resolvent $\mathbf{H}_b(s, \varepsilon)$. We define the $\varepsilon$-dependent neighborhood $U_\alpha(\varepsilon) := \{s \in \mathbb{C} : |s - s_\alpha^\varepsilon| < c\varepsilon^{1-p}\}$ for a sufficiently small constant $c>0$.

For a common-frequency bank, let
$\mathcal P_{\mathrm{prin},\ell}:=\sigma(S_\ell^\varepsilon)$ be its
complete collective principal packet, counted with algebraic
multiplicity. Proposition~\ref{prop:collective-block-reduction} gives the separation of this whole packet from the remaining spectrum:
\begin{equation}
 \operatorname{dist}\bigl(\mathcal P_{\mathrm{prin},\ell},
        \mathcal P(\varepsilon)\setminus\mathcal P_{\mathrm{prin},\ell}\bigr)
       \ge g_0\varepsilon^{1-p}>0.
\end{equation}
For two distinct carriers, the $O(\varepsilon)$ localization of their
packets gives, for sufficiently small $\varepsilon$,
\begin{equation}
 \operatorname{dist}\bigl(\mathcal P_{\mathrm{prin},\ell},
                  \mathcal P_{\mathrm{prin},r}\bigr)
 \ge |\widehat\omega_\ell-\widehat\omega_r|-C\varepsilon
 \ge\Delta_\omega>0,\qquad \ell\ne r,
\end{equation}
where $\Delta_\omega:=\tfrac12\min_{\ell\ne r}
|\widehat\omega_\ell-\widehat\omega_r|$ for $L>1$; this comparison is
omitted when $L=1$. This fixed gap concerns distinct carriers; within a bank, the bare frequencies may differ by $O(\varepsilon^{1-p})$.

For the non-principal modes of the active cluster, Perron simplicity gives
$$\Delta_{\text{gap}} := \operatorname{dist}\left(\widetilde{\mu}_{\alpha, 1}^{(0)}, \sigma(\widetilde{\mathbf{C}}_\alpha^{(0)}) \setminus \{\widetilde{\mu}_{\alpha, 1}^{(0)}\}\right) > 0.$$

By Step~3, choosing $c$ sufficiently small excludes the non-principal modes from $U_\alpha(\varepsilon)$ for all small $\varepsilon>0$. The isolated principal pole is therefore separated from the complementary spectrum on the $\varepsilon^{1-p}$ scale.

\end{enumerate}
\end{enumerate}
\par\noindent\hfill$\blacksquare$
\end{proof}

\subsection{Collective Reduction without a Spectral Splitting Hypothesis}
\label{subsec:collective-block-reduction}
\noindent
For clusters sharing a carrier, the local Perron vectors generate a principal subspace of dimension $n_\ell$. We reduce the full pencil to this subspace without separating its collective roots. Local tuning and the two separation scales suffice for the reduction. Throughout this subsection $B^{\rm in}=I$, as in the physical system. For the chosen carrier banks, $n_\ell=N_M$ and $N=LN_M$, independently of the Dirichlet multiplicity $m_\ell$. The cluster count is fixed as $\varepsilon\to0$; it enters the constants but not the powers of $\varepsilon$.

\begin{proposition}
\label{prop:collective-block-reduction}
Assume the material, geometric, and local tuning conditions above,
with a fixed number of bubbles and fixed $T$. Set $s_{0,\ell}=
 i\widehat\omega_\ell$. There are uniformly bounded Perron bases
$V_{0,\ell},W_{0,\ell}$, normalized by
$W_{0,\ell}^{\top}V_{0,\ell}=I_{n_\ell}$, and a fixed $r>0$ such that:
\begin{enumerate}[(a)]
\item On $|s-s_{0,\ell}|\le 4rh$, the complementary pencil is
invertible with norm at most $C/h$. The exact Schur complement
$F_\ell(s,\varepsilon)$ is given by
\eqref{eq:collective-schur-pencil} below; its Schur correction is
$O(\varepsilon^2/h)=O(\varepsilon^{1+p})$.
\item There is a bounded matrix family $A_\ell^\varepsilon$, with
no spectral sign or simplicity requirement, such that
\begin{equation}\label{eq:block-generator}
 S_\ell^\varepsilon=s_{0,\ell}I+\varepsilon A_\ell^\varepsilon,
 \qquad \|A_\ell^\varepsilon\|\le C,
\end{equation}
and an analytic invertible $U_\ell(s,\varepsilon)$ on a smaller
$rh$ neighborhood for which
\begin{equation}\label{eq:exact-block-factorization}
 F_\ell(s,\varepsilon)
       =U_\ell(s,\varepsilon)(sI-S_\ell^\varepsilon).
\end{equation}
Thus the packet contains $n_\ell$ characteristic roots counted with
algebraic multiplicity, all in $|s-s_{0,\ell}|\le C\varepsilon$.
Repeated roots and Jordan blocks are allowed.
\item On that neighborhood the microscopic transfer has the exact
meromorphic decomposition
\begin{equation}\label{eq:block-principal-transfer}
 H_b(s)=\widetilde{\mathcal A}_b V_\ell^\varepsilon
        (sI-S_\ell^\varepsilon)^{-1}R_\ell^\varepsilon
        (W_\ell^\varepsilon)^{\top}
          +H_{{\rm reg},\ell}(s),
\end{equation}
where $H_{{\rm reg},\ell}$ is analytic,
\begin{equation}\label{eq:block-input-normalization}
 \begin{split}
 V_\ell^\varepsilon&=V_{0,\ell}+O(\varepsilon^p),\\
 (W_\ell^\varepsilon)^{\top}&=W_{0,\ell}^{\top}+O(\varepsilon^p),\\
 R_\ell^\varepsilon&=\varepsilon\widetilde R_\ell^\varepsilon,
 \qquad \widetilde R_\ell^\varepsilon
          \longrightarrow \frac{i\widehat\omega_\ell^3}{2}I,
 \qquad \|(\widetilde R_\ell^\varepsilon)^{-1}\|\le C.
 \end{split}
\end{equation}
Here $V_\ell^\varepsilon$ and $W_\ell^\varepsilon$ give coordinates for the principal block, and $R_\ell^\varepsilon$ is its input factor. No choice of collective eigenvectors is required.
\end{enumerate}
In particular, on the fixed interval $[0,T]$,
\begin{equation}\label{eq:block-semigroup-fixed-time}
 \sup_{0\le t\le T}\|e^{S_\ell^\varepsilon t}\|
       \le e^{C\varepsilon T},\qquad
 \int_0^T\|e^{S_\ell^\varepsilon t}R_\ell^\varepsilon\|\,dt
       \le C\varepsilon T e^{C\varepsilon T}.
\end{equation}
\end{proposition}

\begin{proof}
\noindent\textit{Step 1: Uniform local gap and the complementary inverse.}
Write the conservative pencil as
\[
 P^c(s)=\operatorname{diag}_\alpha
       [I+s^2(d_\alpha I+h\widetilde C_\alpha^{(0)})],
 \qquad P(s)=P^c(s)+E(s).
\]
The superscript $c$ on $P$ denotes its conservative part. On every fixed bounded complex frequency
set, $\|E(s)\|+\|E_s(s)\|\le C\varepsilon$. Indeed, within a cluster
\[
 \frac{\varepsilon\widetilde C_j}{4\pi d_{ij}}
       (e^{-s d_{ij}/c_0}-1)=O(\varepsilon),
\]
whereas every inter-cluster entry and its first derivative are
$O(\varepsilon)$.

\medskip
\noindent
The local matrix is symmetrizable:
$\widetilde C_\alpha^{(0)}=D_\alpha C_\alpha$, with
$D_\alpha=D_\alpha^{\top}$ and
$C_\alpha=\operatorname{diag}(\widetilde C_j)>0$.
It is similar to $C_\alpha^{1/2}D_\alpha C_\alpha^{1/2}$.
Thus it has a real spectrum and a uniformly conditioned eigenbasis.
Its positive off-diagonal entries and the Perron theorem imply a simple
largest eigenvalue. The gap is uniformly positive by compactness of the admissible weights and scaled distances, with cluster sizes fixed. Normalize its positive vectors
by $w_\alpha^{\top}v_\alpha=1$ and assemble the disjointly supported
bases $V_{0,\ell},W_{0,\ell}$. Set
$\Pi_\ell=V_{0,\ell}W_{0,\ell}^{\top}$ and
$\Pi_\ell^c=I-\Pi_\ell$.

\medskip
\noindent
For a non-principal eigenvalue $\mu_{\alpha,j}$ in this group,
\[
 1-\widehat\omega_\ell^2(d_\alpha+h\mu_{\alpha,j})
 =\underbrace{1-\widehat\omega_\ell^2
       (d_\alpha+h\mu_{\alpha,1})}_{O(\varepsilon)}
       +\widehat\omega_\ell^2h(\mu_{\alpha,1}-\mu_{\alpha,j}).
\]
Its modulus is at least $ch$. The other frequency groups are separated
by a fixed positive gap and give order-one bounds. For sufficiently
small fixed $r$, these bounds persist on $|s-s_{0,\ell}|\le4rh$.
Since $\varepsilon/h=\varepsilon^p\to0$, a Neumann series incorporates $E$ and gives
\[
 D_{\ell,c}(s):=\Pi_\ell^cP(s)\Pi_\ell^c
             |_{\operatorname{Ran}\Pi_\ell^c},\qquad
 \|D_{\ell,c}(s)^{-1}\|\le C/h.
\]
The pencil preserves both subspaces, so both
principal--complement couplings of $P$ are $O(\varepsilon)$.

\medskip
\noindent\textit{Step 2: Schur elimination and the scaled expansion.}
Decompose $v=V_{0,\ell}z+v_c$. The complementary equation gives
$v_c=-D_{\ell,c}^{-1}\Pi_\ell^cPV_{0,\ell}z$.
Substituting into the principal equation gives
\begin{equation}\label{eq:collective-schur-pencil}
 \begin{split}
 F_\ell(s,\varepsilon)
 &=W_{0,\ell}^{\top}P(s)V_{0,\ell}\\
 &\quad-W_{0,\ell}^{\top}P(s)\Pi_\ell^cD_{\ell,c}(s)^{-1}
                         \Pi_\ell^cP(s)V_{0,\ell}.
 \end{split}
\end{equation}
In an adapted basis, $\det P=\det D_{\ell,c}\det F_\ell$.
The correction has norm $O(\varepsilon^2/h)$ and the graph
$v_c$ has size $O(\varepsilon^p)\|z\|$. Define
\[
 F_{\ell,0}=F_\ell(s_{0,\ell},\varepsilon),\qquad
 F_{\ell,1}=\partial_sF_\ell(s_{0,\ell},\varepsilon).
\]
Local tuning and the above bounds give
\[
 \|F_{\ell,0}\|\le C\varepsilon,\qquad
 F_{\ell,1}=\frac{2i}{\widehat\omega_\ell}I+o(1).
\]
For the derivative of the Schur term one uses
$\partial_sD_{\ell,c}^{-1}
=-D_{\ell,c}^{-1}(\partial_sD_{\ell,c})D_{\ell,c}^{-1}$;
its largest bound is $C\varepsilon^2/h^2=o(1)$.
Cauchy estimates on the larger $h$-scale disc give the Taylor expansion
\[
 F_\ell(s_{0,\ell}+w)=\sum_{k\ge0}F_{\ell,k}w^k,
 \qquad \|F_{\ell,k}\|\le C h(2rh)^{-k}\quad(k\ge2).
\]
In particular, uniformly for bounded $z$, including one $z$ derivative,
\begin{equation}\label{eq:collective-scaled-pencil}
 \varepsilon^{-1}F_\ell(s_{0,\ell}+\varepsilon z,\varepsilon)
 =F_{\ell,1}(zI-A_\ell^\varepsilon)+O(\varepsilon^p).
\end{equation}
We construct $A_\ell^\varepsilon$ below. Only its uniform boundedness is needed for the control construction; convergence of the scaled coefficients is not required.

\medskip
\noindent\textit{Step 3: Analytic factorization of the principal block.}
Suppress $\ell$ temporarily. Solve the matrix equation
\[
 F_0+F_1X+\sum_{k\ge2}F_kX^k=0
\]
with the coefficients acting on the left. On the ball
$\|X\|\le K\varepsilon$, the fixed-point map
\[
 X\longmapsto-F_1^{-1}F_0-F_1^{-1}\sum_{k\ge2}F_kX^k
\]
maps the ball into itself for a sufficiently large fixed $K$ and small
$\varepsilon$. Its nonlinear part is $O(\varepsilon^2/h)$ and its
Lipschitz constant is $O(\varepsilon/h)$; the latter follows from
$\|X^k-Y^k\|\le k\max(\|X\|,\|Y\|)^{k-1}\|X-Y\|$.
Contraction therefore gives a unique small solution with
\[
 X=-F_1^{-1}F_0+O(\varepsilon^2/h),\qquad
 S=s_0I+X,\qquad A^\varepsilon=\varepsilon^{-1}X.
\]
This also proves \eqref{eq:collective-scaled-pencil}.
The convergent analytic series
\[
 U(s_0+w)=\sum_{k\ge1}F_k
                \sum_{j=0}^{k-1}w^{k-1-j}X^j
\]
satisfies $F(s_0+w)=U(s_0+w)(wI-X)$, by the identity
$w^kI-X^k=(\sum_{j=0}^{k-1}w^{k-1-j}X^j)(wI-X)$.
On a sufficiently small fixed subdisc of the $h$-neighborhood,
$U$ is a uniformly small perturbation of $F_1$ and is invertible.
Shrinking $r$ once, if necessary, proves
\eqref{eq:exact-block-factorization}. Taking determinants counts all roots in the packet, with no simplicity assumption. The packet is separated from the remaining spectrum by at least $ch$, while distinct carrier packets have order-one separation. For $n_\ell=1$, this argument also gives the isolated-branch comparison following Assumption~\ref{assum:precise-tuning}.

\medskip
\noindent\textit{Step 4: Constant input and output factors for the exact packet.}
Define the analytic graph factors
\[
 \begin{split}
 V(s)&=V_0-\Pi^cD_c(s)^{-1}\Pi^cP(s)V_0,\\
 W(s)&=W_0^{\top}-W_0^{\top}P(s)\Pi^cD_c(s)^{-1}\Pi^c.
 \end{split}
\]
The full inverse is
$P^{-1}=V(s)F(s)^{-1}W(s)+\Pi^cD_c(s)^{-1}\Pi^c$.
Here $V(s)=V_0+O(\varepsilon^p)$ and
$W(s)=W_0^{\top}+O(\varepsilon^p)$ uniformly on a suitable $h$-disc.
Put $L(s)=U(s)^{-1}W(s)s^2$ and expand
$V(s_0+w)=\sum V_kw^k$, $L(s_0+w)=\sum L_kw^k$.
Keeping the Taylor coefficients on their respective sides, set
\[
 V^\varepsilon=\sum_{k\ge0}V_kX^k,\qquad
 L^\varepsilon=\sum_{k\ge0}X^kL_k.
\]
These series converge because $\|X\|=O(\varepsilon)=o(h)$.
Using polynomial division on the two sides of the resolvent shows
\[
 V(s)(sI-S)^{-1}L(s)
    =V^\varepsilon(sI-S)^{-1}L^\varepsilon
                       +\text{an analytic matrix}.
\]
For example, $w^k(wI-X)^{-1}-X^k(wI-X)^{-1}$ is a matrix polynomial;
the same identity applies on the right. Applying these identities to the convergent series gives the stated decomposition directly.

\medskip
\noindent
Set
\[
 \widetilde R^\varepsilon=L^\varepsilon V_0,\qquad
 (W^\varepsilon)^{\top}=(\widetilde R^\varepsilon)^{-1}L^\varepsilon,
 \qquad R^\varepsilon=\varepsilon\widetilde R^\varepsilon.
\]
Since $F_1\to2iI/\widehat\omega_\ell$ and
$s_0^2=-\widehat\omega_\ell^2$, one has
$\widetilde R^\varepsilon\to i\widehat\omega_\ell^3I/2$.
Moreover $L^\varepsilon\Pi^c=O(\varepsilon^p)$; together with
$(W^\varepsilon)^{\top}V_0=I$ this gives
$(W^\varepsilon)^{\top}=W_0^{\top}+O(\varepsilon^p)$.
The graph estimate and the convergent series similarly give
$V^\varepsilon=V_0+O(\varepsilon^p)$.
Multiplication by $\mathcal A_b=\varepsilon\widetilde{\mathcal A}_b$
now proves \eqref{eq:block-principal-transfer} and
\eqref{eq:block-input-normalization}.

\medskip
\noindent
Finally, $s_0$ is purely imaginary, so the exponential series gives
$\|e^{St}\|=\|e^{\varepsilon A^\varepsilon t}\|
\le e^{\varepsilon\|A^\varepsilon\|t}$.
Multiplying by $\|R^\varepsilon\|\le C\varepsilon$ and integrating
proves \eqref{eq:block-semigroup-fixed-time}. These estimates depend only on $\|A^\varepsilon\|$, not on its invertibility or diagonalizability.
\par\noindent\hfill$\blacksquare$
\end{proof}

\subsection{Scalar Residue Normalization for an Isolated Group}

\begin{proof}
Normalize the left and right vectors by $w_\alpha^\top v_\alpha=1$ and put
$\mathcal J_\alpha=w_\alpha^\top P_s(s_\alpha^\varepsilon,\varepsilon)v_\alpha$.
For an isolated active group ($n_\ell=1$),
$\mathcal J_\alpha=2i/\widehat\omega_{\tau(\alpha)}+o(1)$.
The exact microscopic matrix residue is
\[
 \mathcal R_\alpha^\varepsilon
 =\frac{(s_\alpha^\varepsilon)^2}{\mathcal J_\alpha}
        [\mathcal A_bv_\alpha]
        [w_\alpha^\top B^{\mathrm{in}}(s_\alpha^\varepsilon)].
\]
Set $\widetilde{\mathcal A}_b=\varepsilon^{-1}\mathcal A_b$ and define
\begin{equation}\label{eq:corrected-scalar-residue}
 r_\alpha^\varepsilon
 =\varepsilon\frac{(s_\alpha^\varepsilon)^2}{\mathcal J_\alpha},\quad
 c_\alpha^\varepsilon=B_{\rm out}\widetilde{\mathcal A}_bv_\alpha,\quad
 b_\alpha^\varepsilon(s)^*
 =w_\alpha^\top B^{\mathrm{in}}(s_\alpha^\varepsilon)G_{\rm tr}(s).
\end{equation}
Then
\[
 \operatorname*{Res}_{s=s_\alpha^\varepsilon}\mathcal H_{\rm ext}^\varepsilon(s)
 =r_\alpha^\varepsilon c_\alpha^\varepsilon
                       b_\alpha^\varepsilon(s_\alpha^\varepsilon)^*.
\]
The matrix dimensions are respectively $M\times M$ for
$\mathcal R_\alpha^\varepsilon$ and $N\times M_{\rm tr}$ for the exterior
residue. Since $r_\alpha^\varepsilon/\varepsilon
\to i\widehat\omega_{\tau(\alpha)}^3/2\ne0$,
$c\varepsilon\le|r_\alpha^\varepsilon|\le C\varepsilon$.
The normalized output and input factors determine whether this mode can be observed and excited. Removing the pole packet leaves an analytic remainder, which proves the stated Laurent decomposition.
\par\noindent\hfill$\blacksquare$
\end{proof}

\begin{remark}
The bands $J$ and $I_M$ in Section~\ref{sec:Minnaert} refer to the Fourier frequency $\omega$. In the Laplace domain, the corresponding estimates are taken on $s=\sigma+i\omega$, with $\sigma>0$ fixed and $\omega$ in the given band.
\end{remark}
\section{Finite-Time Realization by Collective Resonant Blocks}
\label{sec:resonant-gain}
\noindent
We now use the collective transfer matrix to construct exterior incident waves that realize the prescribed cluster sources. The construction is causal and uses the exact matrix kernel on a fixed time interval. For slowly modulated packets at a Minnaert carrier, the differential inverse cancels the factor $\varepsilon^{-1}$ coming from the bubble capacitance. Its cost is bounded in the intrinsic source norm. Away from the Minnaert frequencies, harmonic inversion retains this factor. We first establish the latter comparison, then derive the causal realization and its error. No stationary inverse at the carrier or all-time decay estimate is needed.

\subsection{The Exterior-to-source Map}
\noindent
Let $\{x_m^{\mathrm{tr}}\}_{m=1}^{M_{\mathrm{tr}}}\subset\mathbb{R}^3\setminus \overline{\Omega}$ denote the exterior transducer locations. For a fixed active tracking horizon $T>0$, we define the physical transducer input space and the macroscopic cluster-source space as
\[
\begin{aligned}
 U_{\mathrm{tr}}^\varepsilon
   &:=\{\boldsymbol\lambda\in W^{2,\infty}(0,T;\mathbb R^{M_{\rm tr}}):
                    \boldsymbol\lambda(0)=\dot{\boldsymbol\lambda}(0)=0\},\\
 V&:=C([0,T];\mathbb R^N).
\end{aligned}
\]
For inputs in $U_{\mathrm{tr}}^\varepsilon$, the delayed incident traces belong to $H^2(0,T)$, so the bubble equations have $L^2$ acceleration forcing. The source profiles to be realized belong to the smoother class $V_{I_M}^m$. Their additional regularity is part of the target specification, not a smoothing property of the actuator map. For the physical forward expansion, we use controls in $\Xnine$. 

\medskip
\noindent
The actuator map $\mathcal T^\varepsilon\in\mathcal L(U_{\mathrm{tr}}^\varepsilon,V)$ sends the transducer signals to the cluster-source strengths. It is the composition
\[
 \boldsymbol{\lambda} \overset{G_{\mathrm{tr}}}{\longmapsto} \mathbf{u}^{\mathrm{in}}\big|_{\{z_i\}_{i=1}^M}
 \overset{H_b}{\longmapsto} \mathbf{q}
 \overset{B_{\mathrm{out}}}{\longmapsto} \mathbf{q}_{\mathrm{cl}}^\varepsilon.
\]
Its Laplace symbol is therefore
\begin{equation}
    \mathcal{H}_{\mathrm{ext}}^\varepsilon(s) = B_{\mathrm{out}} H_b(s) G_{\mathrm{tr}}(s) \in \mathbb{C}^{N \times M_{\mathrm{tr}}}.
\end{equation}

\subsection{Asymptotic Divergence and Control Constraints on Non-Resonant Bands}
\label{subsec:gain-estimates}
\noindent
The effective sources satisfy $q_i^\veps=-C_i(\veps)Y_i$, with $C_i(\veps)=\veps\tilde C_i$. Away from resonance, the factor $\veps$ in the capacitance reduces the source amplitude by one power of the bubble size. The following estimate quantifies the resulting input cost.

\medskip
\noindent
Let $J\subset(0,\infty)$ be a compact interval and define the spectral contour $\Gamma_J := \{i\omega : \omega \in J\} \subset i\mathbb{R}$. Let $d_J$ denote the minimal distance from $\Gamma_J$ to the pole set $\mathcal{P} = \{s_j^\ast\}$ of the transfer matrix $H_b$:
\[
 d_J := \inf_{\omega \in J, \, s_j^\ast \in \mathcal{P}} |i\omega - s_j^\ast|>0.
\]

\begin{proposition}
\label{prop:nonresonant-gain}
Let $J \subset (0,\infty)$ be a compact non-resonant frequency band a fixed positive distance from the limiting bare Minnaert frequencies such that its spectral distance to the perturbed pole set $\mathcal{P}$ satisfies $d_J := \inf_{\omega\in J,\, s_j^*\in \mathcal{P}} |i\omega - s_j^*| > 0$. Under the asymptotic scaling $C_i(\varepsilon) = \varepsilon \tilde{C}_i$, there exists a constant $C_J>0$ independent of $\varepsilon$ such that the transfer matrix satisfies
\begin{equation}
\label{eq:nonresonant-bound}
 \sup_{\omega\in J} \|H_b(i\omega)\|_{\mathcal{L}(\mathbb{C}^M,\mathbb{C}^M)} \le \varepsilon \frac{C_J}{d_J}.
\end{equation}
Every measurable right inverse $H_b(i\omega)^\#$, when it exists, satisfies
\[
\operatorname*{ess\,sup}_{\omega\in J}\|H_b(i\omega)^{\#}\|\ge \frac{1}{\sup_{\omega\in J}\|H_b(i\omega)\|} \ge \frac{d_J}{\varepsilon C_J}.
\]
\end{proposition}

\begin{proof}
By Proposition~\ref{proposition}, $H_b(s,\varepsilon)=\mathcal A_b(\varepsilon)P(s,\varepsilon)^{-1}s^2B^{\text{in}}(s)$, where $P=\mathcal D+s^2\mathcal Q$. Factoring $\mathcal A_b(\varepsilon)=\varepsilon\tilde{\mathcal A}_b$, we have
\[
H_b(s, \varepsilon) = \varepsilon \tilde{H}_b(s, \varepsilon), \quad \text{where} \quad \tilde{H}_b(s, \varepsilon) := \tilde{\mathcal{A}}_b P(s, \varepsilon)^{-1} s^2 B^{\text{in}}(s).
\]
It remains to bound $\tilde H_b$ uniformly on $J$.
\newline
The separation of $J$ from the bare Minnaert frequencies gives a uniform bound for $\mathcal D(i\omega)^{-1}$. The two-scale geometry gives $\sup_{\omega\in J}\|\mathcal Q(i\omega,\varepsilon)\|=\mathcal O(\varepsilon^{1-p})$.

\medskip
\noindent
For small $\varepsilon$, a Neumann series yields, uniformly on $J$,
\begin{align*}
P(i\omega, \varepsilon)^{-1} &= \left(I + \mathcal{D}(i\omega)^{-1}(i\omega)^2 \mathcal{Q}(i\omega, \varepsilon)\right)^{-1} \mathcal{D}(i\omega)^{-1} \\
&= \mathcal{D}(i\omega)^{-1} - \mathcal{D}(i\omega)^{-1}(i\omega)^2 \mathcal{Q}(i\omega, \varepsilon)\mathcal{D}(i\omega)^{-1} + \mathcal{O}(\varepsilon^{2-2p}) \\
&= \mathcal{D}(i\omega)^{-1} + \mathcal{O}(\varepsilon^{1-p}).
\end{align*}
The remaining factors are uniformly bounded on $J$, so $\sup_{\omega\in J}\|\tilde H_b(i\omega,\varepsilon)\|\le C_J/d_J$. This proves \eqref{eq:nonresonant-bound}. The inverse estimate follows from $\|I\|\le\|H_b(i\omega)\|\,\|H_b(i\omega)^\#\|$.
\par\noindent\hfill$\blacksquare$
\end{proof}

\noindent
Thus harmonic inversion on a fixed non-resonant band has an $\varepsilon^{-1}$ cost. We next consider causal realization on $[0,T]$ for sources modulated at the assigned carriers.

\subsection{Causal Realization on the Minnaert Band}\label{sec:Minnaert}\label{appsur}
\noindent
The source class is centered at the assigned Minnaert carriers,
with $I_M=\bigcup_\alpha I_\alpha(\varepsilon)$ and
$I_\alpha(\varepsilon)=[\widehat\omega_{\tau(\alpha)}-\delta(\varepsilon),
\widehat\omega_{\tau(\alpha)}+\delta(\varepsilon)]$.
For each carrier, we use the whole principal packet of dimension $n_\ell=N_M$. The index $\alpha$ labels a physical cluster, the microscopic system has dimension $M$, and its cluster output has dimension $N=LN_M$. The reduction is made in block coordinates, without diagonalizing $S_\ell^\varepsilon$.

\begin{proposition}
\label{prop:Minnaert-gain}
Choose $r>0$ so that the $rh$ neighborhoods of
$s_{0,\ell}=i\widehat\omega_\ell$ contain the complete principal
packets of Proposition~\ref{prop:collective-block-reduction} and no
complementary characteristic roots. Then the exact block expansion
\eqref{eq:block-principal-transfer} satisfies
\begin{equation}\label{eq:local-block-remainder}
 \|H_{{\rm reg},\ell}(s)\|\le C\varepsilon^p,
       \qquad |s-s_{0,\ell}|\le rh/2.
\end{equation}
Define
\begin{equation}\label{eq:spatial-block-factors}
 \begin{split}
 C_\ell^\varepsilon&=B_{\rm out}\widetilde{\mathcal A}_b
                                  V_\ell^\varepsilon,\\
 B_\ell^\varepsilon(s)&=(W_\ell^\varepsilon)^{\top}G_{\rm tr}(s),\\
 D_\ell^\varepsilon(s)&=(sI-S_\ell^\varepsilon)^{-1}R_\ell^\varepsilon.
 \end{split}
\end{equation}
Thus the local exterior transfer is
$\mathcal H_{\rm ext}^\varepsilon(s)=
 C_\ell^\varepsilon D_\ell^\varepsilon(s)B_\ell^\varepsilon(s)
 +\mathcal H_{{\rm reg},\ell}^\varepsilon(s)$, with analytic
$\mathcal H_{{\rm reg},\ell}^\varepsilon=O(\varepsilon^p)$ on the
smaller disc. At points where the inverse exists, we have
\begin{equation}\label{eq:block-linewidth-transfer}
 D_\ell^\varepsilon(s_{0,\ell}+\varepsilon z)
       =(zI-A_\ell^\varepsilon)^{-1}\widetilde R_\ell^\varepsilon.
\end{equation}
The assumptions do not require $zI-A_\ell^\varepsilon$ to be uniformly invertible on $i\mathbb R$.
\end{proposition}
\begin{proof}
Writing $w=s-s_{0,\ell}$, local tuning and the Schur estimate give
on a small $h$-scale circle
\[
 F_\ell(s)=\frac{2i}{\widehat\omega_\ell}wI
  +O(h|w|+|w|^2+\varepsilon+\varepsilon^2/h)
  =\frac{2i}{\widehat\omega_\ell}wI+o(h).
\]
Equivalently one may use the exact factorization
$F_\ell=U_\ell(s)(sI-S_\ell^\varepsilon)$, see Proposition \ref{prop:collective-block-reduction}, where
$\|S_\ell^\varepsilon-s_{0,\ell}I\|=O(\varepsilon)=o(h)$.
In either case $\|F_\ell^{-1}\|\le C/h$ on the circle.
The full Schur inverse, including its graph factors and complementary
term, gives $\|P(s)^{-1}\|\le C/h$. Consequently
\[
 \sup_{|s-s_{0,\ell}|=rh}\|H_b(s)\|
       \le C\varepsilon/h=C\varepsilon^p.
\]
The rational block in \eqref{eq:block-principal-transfer} contains all singular terms, including higher-order terms at defective eigenvalues. Its subtraction leaves an analytic matrix, and Cauchy's formula gives
\[
 H_{{\rm reg},\ell}(s)
   =\frac1{2\pi i}\int_{|\zeta-s_{0,\ell}|=rh}
                 \frac{H_b(\zeta)}{\zeta-s}\,d\zeta,
        \qquad |s-s_{0,\ell}|<rh.
\]
This proves the bound on $|s-s_{0,\ell}|\le rh/2$. Multiplication by the bounded exterior trace and output maps gives the exterior estimate. Formula~\eqref{eq:block-linewidth-transfer} follows by cancelling the factors of $\varepsilon$ in the resolvent and the input matrix. 
\bigskip

\noindent A uniform harmonic bound would require a further condition on the collective resolvent; the causal construction below uses only the matrix exponential.
\par\noindent\hfill$\blacksquare$
\end{proof}

\begin{proposition}
\label{prop:right-inverse-band}
Under the geometric, local tuning, and transducer accessibility
assumptions, the matrices in \eqref{eq:spatial-block-factors} satisfy
uniformly in $\varepsilon$ and the finite set of groups
\[
 \sigma_{\min}(C_\ell^\varepsilon)\ge c_{\rm out}>0,
 \qquad
 B_\ell^\varepsilon(i\widehat\omega_\ell)
 B_\ell^\varepsilon(i\widehat\omega_\ell)^*
      \succeq c_{\rm in}^2I_{n_\ell}.
\]
The output has full column rank and the input full row rank, and both
Moore--Penrose inverses are uniformly bounded. If $J_\ell$ embeds the
physical cluster coordinates of this group into $\mathbb C^N$, then
$J_\ell^*C_\ell^\varepsilon$ is uniformly invertible and
\[
 \|(I-C_\ell^\varepsilon(C_\ell^\varepsilon)^\dagger)J_\ell\|
       \le C\varepsilon^p.
\]
Here $C_\ell^\varepsilon$ has size $(LN_M)\times N_M$ and $B_\ell^\varepsilon$ has size $N_M\times M_{\rm tr}$. These bounds concern excitation and observation of the cluster modes. The separate bound on $C_M^{(\ell)}$ concerns their action on the selected pressure modes. Both enter the control construction, without a condition on collective eigenvectors.
\end{proposition}
\begin{proof}
By the block construction,
\[
 V_\ell^\varepsilon=V_{0,\ell}+E_{V,\ell},\qquad
 (W_\ell^\varepsilon)^{\top}=W_{0,\ell}^{\top}+E_{W,\ell},\qquad
 \|E_{V,\ell}\|+\|E_{W,\ell}\|\le C\varepsilon^p.
\]
The columns of the reference bases are positive on their own clusters
and vanish on the others, with $w_\alpha^{\top}v_\alpha=1$.

\medskip
\noindent\textit{(i) Input sampling.}
The transducers and bubble region have a fixed positive separation.
For $i\in\mathcal I_\alpha$, Taylor's formula at a fixed carrier gives
\[
 G_{i\widehat\omega_\ell}(z_i,x_m^{\rm tr})
 =G_{i\widehat\omega_\ell}(y_\alpha,x_m^{\rm tr})+\mathfrak e_{im},
 \qquad |\mathfrak e_{im}|\le C|z_i-y_\alpha|\le C\varepsilon^p.
\]
Therefore
\[
 [W_{0,\ell}^{\top}G_{\rm tr}(i\widehat\omega_\ell)]_{\alpha m}
 =\left(\sum_{i\in\mathcal I_\alpha}w_{\alpha,i}^{(0)}\right)
       [\widetilde G_{{\rm tr},\ell}(i\widehat\omega_\ell)]_{\alpha m}
        +O(\varepsilon^p).
\]
Here $\widetilde G_{{\rm tr},\ell}=J_\ell^*\widetilde G_{\rm tr}$.
Let $D_{{\rm in},\ell}$ be the diagonal matrix of these positive
Perron row sums. The sums have uniformly positive lower bounds and
finite upper bounds on the fixed admissible family.

\medskip
\noindent\textit{(ii) Output weights.}
With the factor $\varepsilon$ included in $R_\ell^\varepsilon$, the output factor is $\widetilde{\mathcal A}_b$. For each reference column,
\[
 B_{\rm out}\widetilde{\mathcal A}_b v_\alpha^{(0)}
   =-\left(\sum_{i\in\mathcal I_\alpha}
                \widetilde C_i v_{\alpha,i}^{(0)}\right)\mathbf e_\alpha.
\]
Let $D_{{\rm out},\ell}$ be the diagonal matrix of these negative
sums; their moduli are bounded away from zero. Hence
\begin{equation}\label{eq:B_matrix_identity}
 \begin{split}
 C_\ell^\varepsilon
    &=J_\ell D_{{\rm out},\ell}+\mathcal E_{C,\ell},\\
 B_\ell^\varepsilon(i\widehat\omega_\ell)
    &=D_{{\rm in},\ell}\widetilde G_{{\rm tr},\ell}
                           (i\widehat\omega_\ell)+\mathcal E_{B,\ell},\\
 \|\mathcal E_{C,\ell}\|+\|\mathcal E_{B,\ell}\|
    &\le C\varepsilon^p.
 \end{split}
\end{equation}

\medskip
\noindent\textit{(iii) Singular values and spatial inverses.}
The full-row accessibility condition implies
$\widetilde G_{{\rm tr},\ell}\widetilde G_{{\rm tr},\ell}^*
\succeq c_{\rm tr}^2I_{n_\ell}$ by restriction to these rows.
The leading input Gram matrix therefore has a uniform positive lower
bound. The leading output Gram matrix is
$D_{{\rm out},\ell}^*D_{{\rm out},\ell}\succeq c^2I$.
Singular-value perturbation by $O(\varepsilon^p)$ proves the asserted
rank and lower bounds, also for
$J_\ell^*C_\ell^\varepsilon=D_{{\rm out},\ell}+O(\varepsilon^p)$.
In particular,
\begin{equation}\label{eq:spatial-pseudoinverse-bounds}
 \|(C_\ell^\varepsilon)^\dagger\|
 +\|(B_\ell^\varepsilon(i\widehat\omega_\ell))^\dagger\|
 +\|(J_\ell^*C_\ell^\varepsilon)^{-1}\|\le C.
\end{equation}

\medskip
\noindent\textit{(iv) Output projection defect.}
Let $\mathcal P_\ell^\varepsilon=C_\ell^\varepsilon
(C_\ell^\varepsilon)^\dagger$. This is the orthogonal projector onto
$\operatorname{Ran}C_\ell^\varepsilon$.
Multiplying the first line of \eqref{eq:B_matrix_identity} by
$I-\mathcal P_\ell^\varepsilon$ gives the exact identity
\[
 (I-\mathcal P_\ell^\varepsilon)J_\ell
 =-(I-\mathcal P_\ell^\varepsilon)
              \mathcal E_{C,\ell}D_{{\rm out},\ell}^{-1}.
\]
The right-hand side is $O(\varepsilon^p)$, which gives the distance between the output subspace and the physical cluster-coordinate subspace.
\par\noindent\hfill$\blacksquare$
\end{proof}

\medskip

\noindent
We now pass to the time domain. A fixed-time estimate for the full delayed system will control the part not represented by the selected blocks. We then invert each principal block and estimate the errors due to output projection, travel times, and coupling to the other carriers. All incident waves are extended by zero for $t<0$.

\begin{lemma}
\label{lem:causal-transfer-fixed-time}
\leavevmode\par\noindent
For the physical trace forcing in \eqref{eq:Y-tilde-system}
($B^{\mathrm{in}}=I$), a fixed number of bubbles, and fixed $T$,
\begin{equation}\label{eq:fixed-time-bubble-estimate}
 \|\mathbf q_{\rm mic}^\varepsilon\|_{C^r([0,T])}
 \le C_{T,r}\varepsilon\|\mathbf u^{\mathrm{in}}\|_{C^{r+2}([0,T])},
 \qquad r=0,1,
\end{equation}
for inputs whose zero extensions have the indicated regularity. Let $H_{\rm prin}$ be the sum of the complete selected packet terms
in \eqref{eq:block-principal-transfer}, and set
$H_{\rm rem}=H_b-H_{\rm prin}$. Then
\begin{equation}\label{eq:fixed-time-remainder-estimate}
 \|(H_{\rm rem}*\mathbf u)(\cdot)\|_{L^\infty(0,T)}
 \le C_T\varepsilon\|\mathbf u\|_{C^2([0,T])}.
\end{equation}
Here $H*\mathbf u$ denotes the causal operator with Laplace symbol $H$;
its kernel need not be an ordinary $L^1$ function. The constants are
uniform in sufficiently small $\varepsilon$ on the fixed horizon.
\end{lemma}
\begin{proof}
\noindent\textit{(i) Inverting the delayed operator.}
Write $D_0=\operatorname{diag}(\omega_{M,i}^{-2})$ and denote the delayed
interaction operator by $\mathcal Q_\varepsilon$. In the component
maximum norm, on every causal prefix $[0,t]\subset[0,T]$,
\[
 \|D_0^{-1}\mathcal Q_\varepsilon\|_{L^\infty(0,t)\to L^\infty(0,t)}
 \le \max_i\omega_{M,i}^2\sum_{j\ne i}|q_{ij}|
 \le C(\varepsilon^{1-p}+\varepsilon)=:\theta_\varepsilon.
\]
A delay cannot increase the supremum on a causal prefix. The second inequality follows from the two separation scales and the bounded material coefficients. Since the bubble count is fixed, the row sums remain $O(\varepsilon^{1-p}+\varepsilon)$ for $N=LN_M$ clusters. For small $\varepsilon$, with the threshold depending on this count, $\theta_\varepsilon<1/2$. Hence the causal inverse
\[
 \mathcal L_\varepsilon
  =(D_0+\mathcal Q_\varepsilon)^{-1}
  =\sum_{n=0}^\infty(-D_0^{-1}\mathcal Q_\varepsilon)^nD_0^{-1}
\]
satisfies $\|\mathcal L_\varepsilon\|\le2\|D_0^{-1}\|$ on every
prefix. This inverse uses the small interaction norm. By contrast, the boundary inverse in \eqref{eq:ws-right-inverse} is a finite series because its delays are strictly positive.

\medskip
\noindent\textit{(ii) The finite-time integral estimate.}
The equation is now
$\ddot{\mathbf Y}=\mathcal L_\varepsilon
(\ddot{\mathbf u}^{\mathrm{in}}-\mathbf Y)$.
Since $\mathbf Y(0)=\dot{\mathbf Y}(0)=0$, put
$Z(t)=\|\ddot{\mathbf Y}\|_{L^\infty(0,t)}$ and obtain
\[
 \|\mathbf Y\|_{L^\infty(0,t)}
 \le\int_0^t(t-s)Z(s)\,ds,
 \qquad
 Z(t)\le C\|\ddot{\mathbf u}^{\mathrm{in}}\|_{L^\infty(0,T)}
            +CT\int_0^tZ(s)\,ds.
\]
Gronwall's inequality gives
$Z(T)\le C e^{CT^2}\|\ddot{\mathbf u}^{\mathrm{in}}\|_\infty$.
Integrating once and twice also bounds $\dot{\mathbf Y}$ and
$\mathbf Y$, respectively. Thus
\[
 \|\mathbf Y\|_{C^2([0,T])}
 \le C_T\|\mathbf u^{\mathrm{in}}\|_{C^2([0,T])}.
\]
Equivalently, the integral equation obtained by substituting
$\displaystyle\mathbf Y(t)=\int_0^t(t-s)\ddot{\mathbf Y}(s)\,ds$ is a causal
Volterra equation. It can be solved on short successive intervals with
uniformly bounded $\mathcal L_\varepsilon$, and the preceding estimate
continues its solution to $T$.
For $r=1$, differentiate the equation once. Fixed delays commute with
this derivative, and the causal jets imply
$\dot{\mathbf Y}(0)=\ddot{\mathbf Y}(0)=0$; the same estimate applies
to $\dot{\mathbf Y}$ with input $\dot{\mathbf u}^{\mathrm{in}}$.
Finally $\mathbf q_{\rm mic}^\varepsilon
=\mathcal A_b\mathbf Y=\varepsilon\widetilde{\mathcal A}_b\mathbf Y$,
which proves \eqref{eq:fixed-time-bubble-estimate}.

\medskip
\noindent\textit{(iii) The remaining transfer components.}
The selected kernel is
\[
 h_{\rm prin}(t)=\mathbf1_{t\ge0}\sum_\ell
   \widetilde{\mathcal A}_bV_\ell^\varepsilon
      e^{S_\ell^\varepsilon t}R_\ell^\varepsilon
                                  (W_\ell^\varepsilon)^\top.
\]
The number and dimensions of the blocks are fixed. Their spatial
factors are bounded, $R_\ell^\varepsilon=O(\varepsilon)$, and
\eqref{eq:block-semigroup-fixed-time} gives
$\|h_{\rm prin}\|_{L^1(0,T)}\le C_T\varepsilon$.
Subtracting these finitely many kernels from the full causal operator in (ii) proves \eqref{eq:fixed-time-remainder-estimate}. This controls the negative-frequency poles and all unselected delay poles together, without assumptions on their multiplicity or spectral sign. In particular, for fixed
$\varepsilon$ and real $s\to+\infty$,
\[
 H_b(s)\longrightarrow\mathcal A_bD_0^{-1},
\]
so the impulse response may contain a Dirac mass. 
\par\noindent\hfill$\blacksquare$
\end{proof}

\begin{proposition}
\label{prop:surjectivity-IM}
Let $T$ be fixed, let $m\ge10$, and assume the material, geometric,
local tuning, and accessibility hypotheses of the main theorem. There is a causal linear differential
construction $R_{I_M}^\varepsilon:V_{I_M}^m\to
U_{{\rm tr},\mathbb C}^\varepsilon\cap\Xnine$ such that
\begin{equation}\label{eq:surjectivity-control-cost}
 \|R_{I_M}^\varepsilon\mathbf q\|_{W^{2,\infty}(0,T)}
 +\|R_{I_M}^\varepsilon\mathbf q\|_{\Xnine}
 \le \kappa_{T,M,m}\|\mathbf q\|_{\mathcal V_\varepsilon^m},
\end{equation}
and
\begin{equation}\label{eq:approx-right-inverse-J}
 \|\mathcal T^\varepsilon R_{I_M}^\varepsilon\mathbf q-\mathbf q\|_{L^\infty(0,T)}
 \le C_{T,M,m}(\varepsilon+\varepsilon^p)
               \|\mathbf q\|_{\mathcal V_\varepsilon^m}.
\end{equation}
The source norm in these estimates includes the amplitude and initial-flatness dependence described in \eqref{eq:flat-start-norm-comparison}.
\end{proposition}
\begin{proof}
We first express the desired sources in block coordinates, invert the principal dynamics, and then determine the transducer signals. The physical travel times are included in the realization error; they are not inverted by a time advance.

\medskip
\noindent\textit{Step 1: Spatial coordinates and the principal inverse.}
Let $J_\ell:\mathbb C^{n_\ell}\to\mathbb C^N$ be the coordinate
inclusion for the physical clusters in $\mathcal N_\ell$, and put
$\mathbf q_\ell=J_\ell J_\ell^*\mathbf q$, so
$\mathbf q=\sum_\ell\mathbf q_\ell$. For this group use the matrices of
Proposition~\ref{prop:collective-block-reduction}:
\[
 S_\ell^\varepsilon=i\widehat\omega_\ell I
                    +\varepsilon A_\ell^\varepsilon,
 \quad R_\ell^\varepsilon=\varepsilon\widetilde R_\ell^\varepsilon,
 \quad D_\ell^\varepsilon(s)
          =(sI-S_\ell^\varepsilon)^{-1}R_\ell^\varepsilon.
\]
This resolvent represents the complete packet. The output matrix relates its block coordinates to the physical source strengths. Accordingly, define
\begin{equation}\label{eq:u_req_def}
 \begin{aligned}
 \mathbf z_\ell&=(C_\ell^\varepsilon)^\dagger\mathbf q_\ell,\\
 \mathbf u_{{\rm req},\ell}
   &=(R_\ell^\varepsilon)^{-1}
       (\partial_t-S_\ell^\varepsilon)\mathbf z_\ell,\\
 \boldsymbol\lambda_\ell
   &=\bigl(B_\ell^\varepsilon(i\widehat\omega_\ell)\bigr)^\dagger
       \mathbf u_{{\rm req},\ell},
 \qquad \boldsymbol\lambda=\sum_\ell\boldsymbol\lambda_\ell.
 \end{aligned}
\end{equation}
Here $\mathbf z_\ell,\mathbf u_{{\rm req},\ell}\in\mathbb C^{n_\ell}$ and $\boldsymbol\lambda_\ell\in\mathbb C^{M_{\rm tr}}$. The order of these factors is fixed: the output inverse precedes the differential inverse.

\medskip
\noindent
For $\mathbf q\in V_{I_M}^m$ (and, in particular, for the
bank-constructed ideal source), write
$\mathbf q_\ell(t)=e^{i\widehat\omega_\ell t}J_\ell\mathbf Q_\ell(\varepsilon t)$
and set $Z_\ell=(C_\ell^\varepsilon)^\dagger J_\ell\mathbf Q_\ell$.
Then the derivative and carrier cancellation give
\begin{align*}
 (\partial_t-S_\ell^\varepsilon)\mathbf z_\ell(t)
 &=e^{i\widehat\omega_\ell t}
       \bigl[\varepsilon Z_\ell'(\varepsilon t)
             +(i\widehat\omega_\ell I-S_\ell^\varepsilon)
                         Z_\ell(\varepsilon t)\bigr],\\
 \mathbf u_{{\rm req},\ell}(t)
 &=e^{i\widehat\omega_\ell t}(\widetilde R_\ell^\varepsilon)^{-1}
       [Z_\ell'(\varepsilon t)+\Theta_\ell^\varepsilon Z_\ell(\varepsilon t)],
\end{align*}
where $\widetilde R_\ell^\varepsilon=\varepsilon^{-1}R_\ell^\varepsilon$
and $\Theta_\ell^\varepsilon
=\varepsilon^{-1}(i\widehat\omega_\ell I-S_\ell^\varepsilon)$ are
uniformly bounded, with $(\widetilde R_\ell^\varepsilon)^{-1}$ uniformly
bounded as well. Since $\Theta_\ell^\varepsilon=-A_\ell^\varepsilon$, boundedness follows from the $O(\varepsilon)$ tuning. Thus, including for defective blocks,
\[
 u_{{\rm req},\ell}(t)
   =e^{i\widehat\omega_\ell t}(\widetilde R_\ell^\varepsilon)^{-1}
       [Z_\ell'(\varepsilon t)-A_\ell^\varepsilon Z_\ell(\varepsilon t)].
\]
Only the input factor $\widetilde R_\ell^\varepsilon$ is inverted. The carrier cancellation removes the explicit $\varepsilon^{-1}$ factor, while the amplitude dependence remains in the slow-envelope norm.

\medskip
\noindent\textit{Step 2: Control regularity, cost, and initialization.}
Put $\mathcal N_\varepsilon=\|\mathbf q\|_{\mathcal V_\varepsilon^m}$.
Writing $\boldsymbol\lambda_\ell(t)
=e^{i\widehat\omega_\ell t}\Lambda_\ell(\varepsilon t)$, the explicit
slow envelope is
\[
 \Lambda_\ell(\tau)
 =\bigl(B_\ell^\varepsilon(i\widehat\omega_\ell)\bigr)^\dagger
    (\widetilde R_\ell^\varepsilon)^{-1}
       [Z_\ell'(\tau)+\Theta_\ell^\varepsilon Z_\ell(\tau)].
\]
Proposition~\ref{prop:right-inverse-band} and the preceding bounds imply
$\sum_\ell\|\Lambda_\ell^{(j)}\|_\infty\le C\mathcal N_\varepsilon$
for $0\le j\le m$. More explicitly, for $0\le k\le9$,
\[
 \partial_t^k\boldsymbol\lambda_\ell(t)
 =e^{i\widehat\omega_\ell t}
   \sum_{j=0}^k\binom{k}{j}(i\widehat\omega_\ell)^{k-j}
                      \varepsilon^j\Lambda_\ell^{(j)}(\varepsilon t).
\]
The carriers and the number of groups are fixed, so summing these
estimates proves \eqref{eq:surjectivity-control-cost}. The flat jets of
$\mathbf q$ through order $m$ imply those of $Z_\ell$ through order
$m$, and hence those of $\Lambda_\ell$ through order $m-1$. Thus
$\boldsymbol\lambda^{(k)}(0)=0$ for $0\le k\le m-1$; for $m\ge10$
the zero extension is $C^9$. A $C^9$ extension beyond $T$ can be
chosen without affecting the causal solution on $[0,T]$. The differential formula is causal on these smooth histories. The propagation delays remain in the physical input map and are estimated in Step~4.

\medskip
\noindent\textit{Step 3: Exact principal convolution and output defect.}
Let $k_\ell(t)=\mathbf1_{t\ge0}e^{S_\ell^\varepsilon t}R_\ell^\varepsilon$.
The frozen input matrix has full row rank, and therefore
$B_\ell^\varepsilon(i\widehat\omega_\ell)\boldsymbol\lambda_\ell
=\mathbf u_{{\rm req},\ell}$. The principal convolution is exactly
\begin{align*}
 (k_\ell*\mathbf u_{{\rm req},\ell})(t)
 &=\int_0^t e^{S_\ell^\varepsilon(t-s)}
                (\partial_s-S_\ell^\varepsilon)\mathbf z_\ell(s)\,ds\\
 &=\int_0^t\partial_s
       \bigl(e^{S_\ell^\varepsilon(t-s)}\mathbf z_\ell(s)\bigr)\,ds\\
 &=\mathbf z_\ell(t)-e^{S_\ell^\varepsilon t}\mathbf z_\ell(0)
  =\mathbf z_\ell(t).
\end{align*}
The resulting output is $C_\ell^\varepsilon(C_\ell^\varepsilon)^\dagger\mathbf q_\ell$. Its difference from $\mathbf q_\ell$ is controlled by Proposition~\ref{prop:right-inverse-band}:
\begin{equation}\label{eq:output-projection-defect}
 \|(I-C_\ell^\varepsilon(C_\ell^\varepsilon)^\dagger)J_\ell\|
 \le C\varepsilon^p.
\end{equation}
Consequently the sum of these output defects has norm at most
$C\varepsilon^p\mathcal N_\varepsilon$.
The bank construction cancels the unwanted ideal modal components exactly. The present defect concerns the physical output map and contributes $O(\varepsilon^p)$ to the tracking error.

\medskip
\noindent\textit{Step 4: Physical propagation delays.}
Set $a_{im}=\rho_c/(4\pi|z_i-x_m^{\rm tr}|)$ and
$\tau_{im}=|z_i-x_m^{\rm tr}|/c_0$. These coefficients and delays are
uniformly bounded for the fixed transducer geometry. The true trace
of the $\ell$th carrier at bubble $i$ is
\[
 (G_{\rm tr}*\boldsymbol\lambda_\ell)_i(t)
 =e^{i\widehat\omega_\ell t}
   \sum_m a_{im}e^{-i\widehat\omega_\ell\tau_{im}}
         \Lambda_{\ell,m}(\varepsilon t-\varepsilon\tau_{im}).
\]
Projecting this expression by the constant block input rows gives
$B_\ell^\varepsilon(\partial_t)\boldsymbol\lambda_\ell$.
The difference from the frozen value is consequently
\[
 \begin{split}
 e_{{\rm delay},\ell,\alpha}(t)
 =e^{i\widehat\omega_\ell t}\sum_{i,m}
  w_{\alpha,i}^\varepsilon a_{im}e^{-i\widehat\omega_\ell\tau_{im}}
  \bigl[\Lambda_{\ell,m}(\varepsilon t-\varepsilon\tau_{im})
                   -\Lambda_{\ell,m}(\varepsilon t)\bigr].
 \end{split}
\]
The weights $w_{\alpha,i}^\varepsilon$ are the entries of $(W_\ell^\varepsilon)^\top$ in \eqref{eq:block-principal-transfer}.
The initial flatness allows us to apply the fundamental theorem of calculus to the zero-extended envelopes, including across $t=0$:
\[
 \Lambda_{\ell,m}(\varepsilon t-\varepsilon\tau_{im})
       -\Lambda_{\ell,m}(\varepsilon t)
 =-\varepsilon\tau_{im}\int_0^1
     \Lambda_{\ell,m}'(\varepsilon t-\theta\varepsilon\tau_{im})\,d\theta.
\]
It follows that $\|\mathbf e_{{\rm delay},\ell}\|_\infty
\le C\varepsilon\mathcal N_\varepsilon$ on all of $[0,T]$.
More specifically, before an individual arrival $t<\tau_{im}$ that
summand is zero on the delayed side, while flatness gives
$|\Lambda_{\ell,m}(\varepsilon t)|
\le(\varepsilon t)^m\|\Lambda_{\ell,m}^{(m)}\|_\infty/m!$.
Thus the initial defect is small in the intrinsic norm even though the transducer paths have different arrival times.
Since $\|k_\ell\|_{L^1(0,T)}\le C_T\varepsilon$, the corresponding
principal output error is at most
$C_T\varepsilon^2\mathcal N_\varepsilon$.

\medskip
\noindent\textit{Step 5: The full realization error.}
Using the packet decomposition with the actual input delays, we obtain
\[
 \begin{split}
 \mathcal T^\varepsilon\boldsymbol\lambda-\mathbf q
  ={}&\underbrace{\sum_\ell
      (C_\ell^\varepsilon(C_\ell^\varepsilon)^\dagger-I)
                                      \mathbf q_\ell}_{\mathcal E_{\rm out}}\\
   &+\underbrace{\sum_\ell C_\ell^\varepsilon
                   (k_\ell*\mathbf e_{{\rm delay},\ell})}_{\mathcal E_{\rm delay}}\\
   &+\underbrace{\sum_{\ell\ne k}C_\ell^\varepsilon k_\ell*
               (B_\ell^\varepsilon(\partial_t)\boldsymbol\lambda_k)}_{\mathcal E_{\rm cross}}
    +\underbrace{B_{\rm out}H_{\rm rem}*
                 (G_{\rm tr}*\boldsymbol\lambda)}_{\mathcal E_{\rm rem}}.
 \end{split}
\]
Every row operator $B_\ell^\varepsilon(\partial_t)$ is a bounded
finite sum of delays. Therefore
$\|\mathcal E_{\rm cross}\|_\infty
\le C_T\varepsilon\sum_k\|\boldsymbol\lambda_k\|_\infty$.
Also $G_{\rm tr}$ maps causal $C^2$ inputs boundedly into causal
$C^2$ bubble traces, so Lemma~\ref{lem:causal-transfer-fixed-time} gives
$\|\mathcal E_{\rm rem}\|_\infty
\le C_T\varepsilon\|\boldsymbol\lambda\|_{C^2}$.
Combining the four terms proves
\begin{equation}\label{eq:final_error_bound}
 \begin{split}
 \|\mathcal T^\varepsilon\boldsymbol\lambda-\mathbf q\|_\infty
 &\le\|\mathcal E_{\rm out}\|_\infty+\|\mathcal E_{\rm delay}\|_\infty
       +\|\mathcal E_{\rm cross}\|_\infty+\|\mathcal E_{\rm rem}\|_\infty\\
 &\le C_{T,M,m}(\varepsilon^p+\varepsilon^2+\varepsilon)
                                      \mathcal N_\varepsilon,
 \end{split}
\end{equation}
which proves \eqref{eq:approx-right-inverse-J}.
\par\noindent\hfill$\blacksquare$
\end{proof}

\paragraph{The contribution of a separated complementary mode.}
The causal estimate above controls the unselected transfer as a whole. The role of the local spectral gap can also be seen directly for a simple complementary pole $s_j$ satisfying $|s_j|\le C$, with residue $\|R_j\|\le C\varepsilon$ and $|i\widehat\omega_\ell-s_j|\ge g_0\varepsilon^{1-p}$. For a causal
slow input $\mathcal U(\varepsilon t)e^{i\widehat\omega_\ell t}$,
with bounded derivatives through order $m$ and flat jets through
$m-1$, its contribution is
\begin{equation}\label{eq:q_trans_integral}
 y_j(t)=R_je^{s_jt}\int_0^t e^{(i\widehat\omega_\ell-s_j)s}
                                      \mathcal U(\varepsilon s)\,ds.
\end{equation}
Put $\Delta_j=i\widehat\omega_\ell-s_j$ and
$\displaystyle I_k(t)=\int_0^t e^{\Delta_js}\mathcal U^{(k)}(\varepsilon s)\,ds$.
Integration by parts gives the explicit recurrence
\[
 I_k(t)=\frac{e^{\Delta_jt}\mathcal U^{(k)}(\varepsilon t)
                     -\mathcal U^{(k)}(0)}{\Delta_j}
                     -\frac{\varepsilon}{\Delta_j}I_{k+1}(t).
\]
After $m$ iterations, the initial terms vanish and
\[
 \begin{split}
 y_j(t)&=e^{i\widehat\omega_\ell t}
       \sum_{k=0}^{m-1}\frac{(-\varepsilon)^k R_j}{\Delta_j^{k+1}}
                                      \mathcal U^{(k)}(\varepsilon t)
                                      +\mathfrak r_{j,m}(t),\\
 \mathfrak r_{j,m}(t)
 &=\frac{(-\varepsilon)^mR_j}{\Delta_j^m}
       \int_0^t e^{s_j(t-s)}e^{i\widehat\omega_\ell s}
                                  \mathcal U^{(m)}(\varepsilon s)\,ds.
 \end{split}
\]
For these bounded-frequency poles, $\Re s_j$ is uniformly bounded
above on a fixed horizon. The $k$th displayed term is at most
$C\varepsilon^{(k+1)p}\|\mathcal U^{(k)}\|_\infty$ and
\[
 \|\mathfrak r_{j,m}\|_\infty
 \le C_T\varepsilon^{1+mp}\|\mathcal U^{(m)}\|_\infty.
\]
This individual contribution is therefore $O(\varepsilon^p)$ in the slow-profile norm. Repeated complementary modes and the remaining delay poles are covered by Lemma~\ref{lem:causal-transfer-fixed-time}, without a corresponding scalar decomposition.


\subsection{Resonant-band Gain and Finite-time Actuation Cost}
\label{subsec:finite-time-resonant-cost}
\noindent
We make explicit the cancellation responsible for the resonant input bound. It relates the Minnaert frequency of the cluster modes to the cost of realizing a source profile on $[0,T]$.

\medskip
\noindent
Fix a carrier $\widehat\omega_\ell$ and write the exact principal generator and input factor as
\[
 S_\ell^\varepsilon=i\widehat\omega_\ell I+\varepsilon A_\ell^\varepsilon,
 \qquad
 R_\ell^\varepsilon=\varepsilon\widetilde R_\ell^\varepsilon,
\]
with $\|A_\ell^\varepsilon\|+\|\widetilde R_\ell^\varepsilon\|+
\|(\widetilde R_\ell^\varepsilon)^{-1}\|\le C$ on the active block.  For a slowly modulated resonant packet
\begin{equation}\label{eq:resonant-packet-cost}
 z_\ell^{\rm res}(t)=e^{i\widehat\omega_\ell t}Z_\ell(\varepsilon t),
\end{equation}
the blockwise differential inverse gives
\begin{equation}\label{eq:resonant-required-input}
 u_{{\rm req},\ell}^{\rm res}(t)
 =e^{i\widehat\omega_\ell t}(\widetilde R_\ell^\varepsilon)^{-1}
 \big[Z_\ell'(\varepsilon t)-A_\ell^\varepsilon Z_\ell(\varepsilon t)\big].
\end{equation}
Hence, on every fixed interval $[0,T]$,
\begin{equation}\label{eq:resonant-cost-O1}
 \|u_{{\rm req},\ell}^{\rm res}\|_{L^\infty(0,T)}
 \le C_T\big(\|Z_\ell\|_{L^\infty}+\|Z_\ell'\|_{L^\infty}\big),
\end{equation}
and Proposition~\ref{prop:surjectivity-IM} gives the corresponding $O(1)$ bound for the transducer control in the intrinsic norm. The cancellation in \eqref{eq:resonant-required-input} depends on matching the temporal carrier to the cluster frequency. The modal inverse for each bank preserves this carrier through the boundary-delay corrections, so the matching remains valid for superpositions.

\medskip
\noindent
For comparison, let $\omega_*$ remain a fixed positive distance from every assigned Minnaert carrier and consider an otherwise comparable packet
\[
 z^{\rm off}(t)=e^{i\omega_*t}Z(\varepsilon t).
\]
If it is realized through the same small-bubble input channel, then
\[
 (\partial_t-S_\ell^\varepsilon)z^{\rm off}
 =e^{i\omega_*t}\Big(i(\omega_*-\widehat\omega_\ell)Z(\varepsilon t)
 +O(\varepsilon)(|Z|+|Z'|)\Big).
\]
The detuning term does not cancel. Since $(R_\ell^\varepsilon)^{-1}=\varepsilon^{-1}(\widetilde R_\ell^\varepsilon)^{-1}$, it produces an input of order $\varepsilon^{-1}$ for a nonzero order-one profile. Correspondingly, Proposition~\ref{prop:nonresonant-gain} gives, on every compact band $J$ separated from the Minnaert frequencies,
\begin{equation}\label{eq:off-resonant-cost-lower}
 \sup_{\omega\in J}\|H_b(i\omega)\|=O(\varepsilon),
 \qquad
 \operatorname*{ess\,sup}_{\omega\in J}\|H_b(i\omega)^\#\|
 \ge c_J\varepsilon^{-1}
\end{equation}
for every available harmonic right inverse.

\begin{proposition}\label{prop:finite-time-cost-advantage}
For resonance-compatible packets of the form \eqref{eq:resonant-packet-cost}, normalized in the intrinsic slow-envelope norm, the incident controls constructed in Proposition~\ref{prop:surjectivity-IM} satisfy
\[
 \|\boldsymbol\lambda_{\rm res}^\varepsilon\|_{\Xnine}\le C_T.
\]
On a fixed band $J$ separated from the Minnaert set, the bubble transfer has norm $O(\varepsilon)$, and every harmonic right inverse has norm at least $c_J\varepsilon^{-1}$ when it exists. Thus the resonant construction avoids the explicit $\varepsilon^{-1}$ loss of off-resonant harmonic inversion. The comparison is between the intrinsic-norm finite-time bound above and the harmonic operator bound on $J$.
\end{proposition}


\section{Energy Estimates and Proof of the Main Result}
\label{sec:proof-main}

\noindent
We complete the proof of Theorem~\ref{thm:Minnaert-tracking_main}. The modal inverse first gives an ideal source profile with the required projected pressure. The causal construction of Section~\ref{sec:resonant-gain} realizes this profile through the bubble clusters. We then estimate the projected pressure error and use the forward scattering expansion to pass from the effective field to the physical field. The incident-field bounds remain relative to the intrinsic source norm.

\medskip
\noindent
Let $I_M = \bigcup_{\alpha=1}^N I_\alpha(\varepsilon)$ be the composite Minnaert band and let $H_M = \operatorname{span}\{\psi_k\}_{k=1}^{N_M}$ be the associated target spectral subspace. Let $T>0$ and let $\bm{p}_r \in C^{m+3}([0,T]; \mathbb{R}^{N_M})$ be a prescribed reference trajectory vector satisfying the initial rest compatibility conditions 
$$ \bm{p}_r^{(k)}(0) = 0, \quad \forall k \in \{0, \dots, m+2\}. $$ 
We define the associated target spatial field
$p_r^{I_M}\in C^{m+3}([0,T];H_M)$ by
\[
 p_r^{I_M}(x,t):=\sum_{k=1}^{N_M}(\bm p_r(t))_k\psi_k(x).
\]
Use the prescribed modal packets in
\eqref{eq:reference-resonance-compatibility} and define
\[
 \mathbf q^{\mathrm{ideal}}
       :=c_0^{-2}\sum_{\ell=1}^L\LMell f_\ell.
\]
Proposition~\ref{prop:bank-right-inverse} gives
$c_0^2\CMws\mathbf q^{\mathrm{ideal}}
=\ddot{\bm p}_r+\Lambda_M\bm p_r$.
The whole-space field generated by this source has the prescribed projection. We first verify that the source belongs to the class required for causal realization.

\begin{proposition}
\label{prop:reference-envelope-consistency}
Let $m\ge2$ and let the reference satisfy
\eqref{eq:reference-resonance-compatibility}, with
$\boldsymbol f\in\mathfrak F_{I_M}^m$ and the stated initial rest
conditions. Define $\mathbf q^{\mathrm{ideal}}$ by
\eqref{eq:main-ideal-source}. For every bank there is a $C^{m+1}$
slow envelope with vanishing initial jets through order $m$ such that
\begin{equation}\label{eq:general-LM-resonance-compatibility}
 J_\ell^*\mathbf q^{\mathrm{ideal}}(t)
 =c_0^{-2}J_\ell^*\LMell f_\ell(t)
 =e^{i\widehat\omega_\ell t}Q_\ell^\varepsilon(\varepsilon t),
 \qquad \ell=1,\ldots,L.
\end{equation}
Consequently $\mathbf q^{\mathrm{ideal}}\in V_{I_M}^m$ and
$P_{I_M}p^{\mathrm{ws}}[\mathbf q^{\mathrm{ideal}}]=p_r^{I_M}$.
Multiplication of all reference packets by $\beta(\varepsilon)$
multiplies the source and its envelopes by the same factor.
\end{proposition}
\begin{proof}
The range of $\LMell$ lies in $\operatorname{Ran}J_\ell$, and the
banks are disjoint. Proposition~\ref{prop:bank-right-inverse} therefore
gives \eqref{eq:general-LM-resonance-compatibility}, its envelope
regularity, and the intrinsic-norm bound. It also gives
\[
 c_0^2\CMws\mathbf q^{\mathrm{ideal}}
 =\sum_\ell f_\ell=\ddot{\bm p}_r+\Lambda_M\bm p_r.
\]
Subtracting the reference equation from \eqref{eq:ws-modal-identity} gives a homogeneous finite-dimensional oscillator with zero initial data. Uniqueness yields the tracking identity, as in Theorem~\ref{thm:finite-tracking}, and linearity gives the amplitude statement. A boundary correction can have components in every selected eigenspace, but it preserves the temporal carrier and is inverted within the corresponding full bank. It is the complete operator $\CMws$, rather than $C_M$ alone, that cancels the unwanted modal components.
\par\noindent\hfill$\blacksquare$
\end{proof}

\paragraph{A sufficient class of reference trajectories.}
Identify each $H_\ell$ with its modal coordinate subspace. For
$P_\ell^\varepsilon\in C^{m+3}([0,T];H_\ell)$, let
\[
 \bm p_r(t)=\sum_{\ell=1}^L
       e^{i\widehat\omega_\ell t}P_\ell^\varepsilon(\varepsilon t).
\]
Assume the zero extensions of these envelopes are $C^{m+3}$,
equivalently $(P_\ell^\varepsilon)^{(j)}(0)=0$ for
$0\le j\le m+3$.
Since $\Lambda_MP_\ell^\varepsilon=\widehat\omega_\ell^2
P_\ell^\varepsilon$, the residual packets are explicitly
\[
 F_\ell^\varepsilon
 =2i\widehat\omega_\ell\varepsilon
                   (P_\ell^\varepsilon)'
       +\varepsilon^2(P_\ell^\varepsilon)'',
\]
and belong to the modal packet class above. More generally, $F_\ell^\varepsilon$ may take values in the whole target space. The source norm includes its amplitude and slow derivatives, with the dependence on initial flatness given by \eqref{eq:flat-start-norm-comparison}.

\medskip
\noindent
The ideal tracking theorem allows general time-dependent references. For physical realization, the packet condition \eqref{eq:reference-resonance-compatibility} and the bank construction give $\mathbf q^{\mathrm{ideal}}\in V_{I_M}^m$ by Proposition~\ref{prop:reference-envelope-consistency}. In particular, $\mathbf q^{\mathrm{ideal}}(0)=\dot{\mathbf q}^{\mathrm{ideal}}(0)=0$. We may therefore apply the causal realization theorem using the collective blocks and their input and output bounds.

\noindent
The resulting source error is uniform on $[0,T]$.

{
\begin{proposition}
\label{prop:surjectivity-to-qdiff-J}
Let $\mathbf q^{\mathrm{ideal}}\in V_{I_M}^m$ with $m\ge10$ and the hypotheses of Proposition~\ref{prop:surjectivity-IM} and set $\bm\lambda^\veps:=R_{I_M}^\veps\mathbf q^{\mathrm{ideal}}$ and $\mathbf q^\veps:=\mathcal T^\veps\bm\lambda^\veps$. There exists $C_{T,m}>0$, independent of $\veps$, such that
\begin{equation}
\label{eq:q-diff-Linfty-J}
 \sup_{t\in[0,T]} \sum_{\alpha=1}^N\abs{q_\alpha^\veps(t)-q_\alpha^{\mathrm{ideal}}(t)} \le C_{T,m}(\veps + \veps^p) \|\mathbf{q}^{\mathrm{ideal}}\|_{\mathcal{V}_\varepsilon^m}.
\end{equation}
\end{proposition}

\begin{proof}
Proposition~\ref{prop:surjectivity-IM} gives $$\norm{\mathcal T^\veps R_{I_M}^\veps \mathbf{q}^{\mathrm{ideal}}-\mathbf{q}^{\mathrm{ideal}}}_{L^\infty(0,T)} \le C_{T,m}(\veps + \veps^p)\|\mathbf{q}^{\mathrm{ideal}}\|_{\mathcal{V}_\varepsilon^m}.$$ Summation over the fixed $N=LN_M$ clusters changes only the constant.
\par\noindent\hfill$\blacksquare$
\end{proof}
}

\bigskip

\noindent
We now use this source estimate in the wave equation and derive the projected pressure bound.

\medskip
\noindent
\textit{Step 1: The effective error equation.}
Let $p^{\mathrm{ideal}}=p^{\mathrm{ws}}[\mathbf q^{\mathrm{ideal}}]$
and $p^{\mathrm{eff},\varepsilon}=p^{\mathrm{ws}}[\mathbf q^\varepsilon]$.
Define
\[
 e_{\mathrm{eff}}^\varepsilon
     =p^{\mathrm{eff},\varepsilon}-p^{\mathrm{ideal}},
 \qquad
 \Delta\mathbf q=\mathbf q^\varepsilon-\mathbf q^{\mathrm{ideal}}.
\]
Subtracting their whole-space point-source equations and restricting to
$\Omega$ gives
\begin{equation}\label{eq:error-equation}
 \frac1{c_0^2}\partial_t^2 e_{\rm eff}^\varepsilon-\Delta e_{\rm eff}^\varepsilon
 =\sum_\alpha\Delta q_\alpha(t)\delta_{y_\alpha}
 \quad\hbox{in }\mathscr D'(\Omega\times(0,T)),
 \qquad e_{\rm eff}^\varepsilon|_{\partial\Omega}=g_{\Delta\mathbf q}.
\end{equation}
The initial data are zero, and the boundary value is the trace of the whole-space difference field. The Foldy--Lax and center-replacement errors enter only when we compare this effective field with the physical field in Step~4.

\medskip
\noindent
For each retained mode set
$e_{k,\mathrm{eff}}^\varepsilon(t)
=\langle e_{\mathrm{eff}}^\varepsilon(\cdot,t),\psi_k\rangle_{L^2(\Omega)}$.
The point-source fields have well-defined $L^2$ coefficients, although
they need not be global $H^1$ fields. Green's identity, in the form
\eqref{eq:ws-modal-identity}, yields the finite-dimensional system
\begin{equation}\label{eq:band-error-vector}
 \ddot{\mathbf e}_{M,\rm eff}^\varepsilon
       +\Lambda_M\mathbf e_{M,\rm eff}^\varepsilon
 =\mathbf F_M^\varepsilon,
 \qquad
 \mathbf F_M^\varepsilon=c_0^2\CMws\Delta\mathbf q,
 \qquad
 \mathbf e_{M,\rm eff}^\varepsilon(0)
       =\dot{\mathbf e}_{M,\rm eff}^\varepsilon(0)=0.
\end{equation}
In particular,
$\mathbf F_M^\varepsilon=c_0^2(C_M\Delta\mathbf q-\BM\Delta\mathbf q)$.
Both terms are bounded causal operators applied to the same source difference.
Proposition~\ref{prop:surjectivity-to-qdiff-J} therefore implies
\begin{equation}\label{eq:forcing-band}
 \|\mathbf F_M^\varepsilon\|_{L^\infty(0,T)}
 \le C_{T,M,m}(\varepsilon+\varepsilon^p)
                   \|\mathbf q^{\mathrm{ideal}}\|_{\mathcal V_\varepsilon^m}.
\end{equation}
Here one uses
$\|\CMws\|\le\|C_M\|+\|\BM\|$ on $L^\infty(0,T)$; the same
bound, with its corresponding norm, holds on $L^2(0,T)$.

\medskip
\noindent\textit{Step 2: Modal energy estimate.}
For the projected effective error, set
\[
 E_M(t)=\frac12\left(
       |\Lambda_M^{1/2}\mathbf e_{M,\rm eff}^\varepsilon(t)|^2
       +|\dot{\mathbf e}_{M,\rm eff}^\varepsilon(t)|^2\right).
\]
For complexified fields, differentiation and
\eqref{eq:band-error-vector} give
\[
 E_M'(t)=\operatorname{Re}
       \langle\dot{\mathbf e}_{M,\rm eff}^\varepsilon(t),
                                      \mathbf F_M^\varepsilon(t)\rangle
 \le\sqrt{2E_M(t)}\,|\mathbf F_M^\varepsilon(t)|.
\]
Applying this inequality first to $\sqrt{E_M+\delta^2}$ and then
letting $\delta\downarrow0$, with $E_M(0)=0$, yields
\[
 \begin{split}
 \sup_{t\le T}\sqrt{E_M(t)}
 &\le\frac1{\sqrt2}\int_0^T|\mathbf F_M^\varepsilon(s)|\,ds
 \le\sqrt{\frac T2}\,\|\mathbf F_M^\varepsilon\|_{L^2(0,T)}\\
 &\le\frac{T}{\sqrt2}\|\mathbf F_M^\varepsilon\|_{L^\infty(0,T)}
 \le C_{T,M,m}(\varepsilon+\varepsilon^p)
                  \|\mathbf q^{\mathrm{ideal}}\|_{\mathcal V_\varepsilon^m}.
 \end{split}
\]
This also follows from the Duhamel representation
\[
 \mathbf e_{M,\rm eff}^\varepsilon(t)
 =\int_0^t\Lambda_M^{-1/2}\sin((t-s)\Lambda_M^{1/2})
                                           \mathbf F_M^\varepsilon(s)\,ds.
\]
The forcing includes the delayed boundary contribution. Both estimates use $\Delta\mathbf q$ without differentiating it.

\medskip
\noindent\textit{Step 3: The projected spatial norm.}
Orthonormality of the Dirichlet eigenfunctions gives
\[
 \|\nabla P_{I_M}e_{\rm eff}^\varepsilon(t)\|_{L^2}^2
    =c_0^{-2}|\Lambda_M^{1/2}\mathbf e_{M,\rm eff}^\varepsilon(t)|^2,
 \qquad
 \|P_{I_M}\partial_t e_{\rm eff}^\varepsilon(t)\|_{L^2}
    =|\dot{\mathbf e}_{M,\rm eff}^\varepsilon(t)|.
\]
Since the retained eigenvalues are positive and the set of modes is
fixed, Poincar\'e's inequality on their span controls the $L^2$ part
of the spatial $H^1$ norm as well. Thus
\begin{equation}\label{eq:approx-ideal-estimate}
 \sup_{t\le T}\left(
       \|P_{I_M}e_{\rm eff}^\varepsilon(t)\|_{H^1(\Omega)}
       +\|P_{I_M}\partial_t e_{\rm eff}^\varepsilon(t)\|_{L^2(\Omega)}
                \right)
 \le C_{T,M,m}(\varepsilon+\varepsilon^p)
                   \|\mathbf q^{\mathrm{ideal}}\|_{\mathcal V_\varepsilon^m}.
\end{equation}
This estimate concerns the effective point-source field. For the physical tracking result, we use only its pressure component.

\medskip
\noindent\textit{Step 4: Comparison with the physical pressure.}
By Proposition~\ref{prop:reference-envelope-consistency},
$P_{I_M}p^{\mathrm{ideal}}=p_r^{I_M}$ for the bank-constructed source. On the exterior of the bubbles,
\[
 p^\varepsilon
 =p^{\mathrm{eff},\varepsilon}
       +R_{\rm ctr}^\varepsilon+R_{\rm FL}^\varepsilon.
\]
Using the observation convention \eqref{eq:exterior-observation-projection}
therefore gives the exact projected identity
\[
 \begin{split}
 \OMext p^\varepsilon-p_r^{I_M}
  =P_{I_M}e_{\rm eff}^\varepsilon
       +\OMext R_{\rm FL}^\varepsilon
       +\OMext R_{\rm ctr}^\varepsilon
       -P_{I_M}(\mathbf1_Dp^{\mathrm{eff},\varepsilon}).
 \end{split}
\]
The last term accounts for the bubble volumes excluded from the observation. The exterior expansion is used only on $\Omega\setminus\overline D$. For its Foldy--Lax remainder,
\[
 \|\OMext R_{\rm FL}^\varepsilon(t)\|_{H^1(\Omega)}
 \le(1+\lambda_{\max,M})^{1/2}
         \|R_{\rm FL}^\varepsilon(t)\|_{L^2(\Omega\setminus\overline D)}.
\]
The center-replacement and missing-volume terms have already been
estimated separately in \eqref{eq:projected-center-remainder} and
\eqref{eq:observation-hole-remainder}. Consequently,
\[
 \begin{split}
 \sup_{t\le T}\|\OMext p^\varepsilon(t)-p_r^{I_M}(t)\|_{H^1}
 &\le C_{T,M,m}(\varepsilon+\varepsilon^p)
                         \|\mathbf q^{\mathrm{ideal}}\|_{\mathcal V_\varepsilon^m}\\
 &\quad+C_{T,M}\varepsilon^\mu\|\boldsymbol\lambda^\varepsilon\|_{\Xnine}
       +C_{T,M}(\varepsilon^p+\varepsilon^2)
                  \sum_i\|q_i^\varepsilon\|_{W^{1,\infty}(0,T)}.
 \end{split}
\]
Lemma~\ref{lem:causal-transfer-fixed-time}, composed with the transducer
trace map, bounds the last source norm by
$C_T\varepsilon\|\boldsymbol\lambda^\varepsilon\|_{C^3}$.
Together with \eqref{eq:surjectivity-control-cost} this proves
\[
 \sup_{t\le T}\|\OMext p^\varepsilon(t)-p_r^{I_M}(t)\|_{H^1}
 \le C_{T,M,m}(\varepsilon^\mu+\varepsilon^p+\varepsilon)
                       \|\mathbf q^{\mathrm{ideal}}\|_{\mathcal V_\varepsilon^m}.
\]
For $0<\varepsilon\le1$ and $\gamma=\min\{\mu,p\}$, the right-hand
side is at most $C_{T,M,m}\varepsilon^\gamma
\|\mathbf q^{\mathrm{ideal}}\|_{\mathcal V_\varepsilon^m}$, as stated
in \eqref{eq:Minnaert-tracking-final}.

\medskip
\noindent
Finally the explicit construction gives
\[
 \|\boldsymbol\lambda^\varepsilon\|_{U_{\rm tr}^\varepsilon}
   +\|\boldsymbol\lambda^\varepsilon\|_{\Xnine}
 \le\kappa_{T,M,m}\|\mathbf q^{\mathrm{ideal}}\|_{\mathcal V_\varepsilon^m}.
\]
The bank inverse estimate of
Proposition~\ref{prop:bank-right-inverse} also bounds every occurrence
of $\|\mathbf q^{\mathrm{ideal}}\|_{\mathcal V_\varepsilon^m}$ here by
$c_0^{-2}C_{T,\rm bank}\|\boldsymbol f\|_{\mathfrak F_\varepsilon^m}$.
The numbers $N=LN_M$ and $M$ are fixed. The constants depend on $T$, the geometric margins, and the conditioning of the modal and transducer maps; amplitude dependence remains in the intrinsic norm.  \hfill$\blacksquare$

\bigskip

\noindent
\textbf{Data availability.} No datasets were generated or analyzed in this study.
\bigskip

\noindent
\textbf{Conflict of interest.} The authors declare no conflict of interest.
\bigskip

\noindent \textbf{Declaration of generative AI and AI-assisted technologies in the manuscript preparation process.} During the preparation of this work the authors used Open-AI ChatGpt 5.6 in order to polish and optimize the presentation. After using this tool/service, the author(s) reviewed and edited the content as needed and take(s) full responsibility for the content of the published article.

\end{document}